%% file: main.tex
\documentclass[letterpaper,11pt]{article}

\input{prologue.tex}

\newif\ifaistats
\aistatsfalse
\ifaistats
\usepackage{aistats2020}
\else
\usepackage[margin=1in]{geometry}
\usepackage{times}
\date{}
\fi
\begin{document}

%

%
\ifaistats
\twocolumn[
\aistatstitle{Locally Accelerated Conditional Gradients}
\aistatsauthor{ Anonymous Author(s) }
\aistatsaddress{  } ]
\else
\title{Locally Accelerated Conditional Gradients}
\author{Jelena Diakonikolas\\
Univeristy of California, Berkeley\\
\texttt{jelena.d@berkeley.edu}
\and 
Alejandro Carderera\\
Georgia Institute of Technology\\
\texttt{alejandro.carderera@gatech.edu}
\and 
Sebastian Pokutta\\
Zuse Institute Berlin and Technische Universität Berlin\\
\texttt{pokutta@zib.de}}
\maketitle
\fi

\begin{abstract}
Conditional gradients constitute a class of 
\emph{projection-free} first-order algorithms for smooth convex optimization. As such, they are frequently used in solving smooth convex optimization problems over polytopes, for which the computational cost of orthogonal projections would be prohibitive. However, they do not enjoy the optimal convergence rates achieved by projection-based accelerated methods; moreover, \emph{achieving such globally-accelerated rates is information-theoretically impossible} for these methods. To address this issue, we present \emph{Locally Accelerated Conditional Gradients} -- an algorithmic framework that couples accelerated steps with conditional gradient steps to achieve \emph{local} acceleration on smooth strongly convex problems. Our approach does not require projections onto the feasible set, but only on (typically low-dimensional) simplices, thus keeping the computational cost of projections at bay. Further, it achieves the \emph{optimal accelerated local convergence}. Our theoretical results are supported by numerical experiments, which demonstrate significant speedups of our framework over state of the art methods in both per-iteration progress and wall-clock time. 
\end{abstract}

\section{Introduction}
Smooth convex optimization problems over polytopes are fundamental optimization problems that arise in a variety of settings, such as video co-localization~\cite{joulin2014efficient} and  structured energy minimization~\cite{swoboda2019map} in computer vision applications, greedy particle optimization in Bayesian inference~\cite{futami2019bayesian}, and structural SVMs~\cite{lacoste2013block}. The methods of choice in such settings are variants of the classical conditional gradient (CG) or Frank-Wolfe method~\cite{frank1956algorithm}, which eschew often computationally-prohibitive projection operations of standard first-order methods while enjoying favorable characteristics of the produced solutions such as sparse representation.  

Despite their simplicity and broad applicability, CG  methods do not attain the accelerated convergence of projection-based methods such as Nesterov accelerated gradient descent~\cite{nesterov19830} and variants thereof~\cite{AXGD,cohen2018acceleration,tseng2008,beck2009fast}. This limitation comes as a consequence of the access to the feasible polytope $\cx$ being restricted to a linear minimization oracle: in such a setting, global acceleration is information-theoretically impossible~\cite{jaggi2013revisiting,lan2013complexity} and improving the (global) rate of convergence of CG methods in various structured settings has been an active area of research~\cite{jaggi2013revisiting,garber2013linearly,lacoste2015global,LDLCC2016,lan2017conditional,freund2017extended,BPZ2017,bashiri2017decomposition,locatello2018revisiting,kerdreux2018frank,kerdreux2018restarting,garber2018fast,braun2018blended,guelat1986some}. While exciting progress has been made on a number of fronts, global convergence rate of CG cannot be dimension-independent in the worst case.

In this work, we depart from the path of improving global convergence guarantees of CG and instead ask:
\begin{center}
    \emph{Is local dimension-independent acceleration possible?}
\end{center}
To address this question, we focus on smooth strongly convex objective functions and introduce the \emph{Locally Accelerated Conditional Gradients} algorithmic framework that achieves the optimal dimension-independent locally accelerated convergence rate. 


\subsection{Limits to Global Acceleration}\label{sec:lower-bound}
Acceleration for conditional gradient methods has been an important topic of interest. Apart from several technical challenges, there is a strong lower bound~\cite{jaggi2013revisiting,lan2013complexity} that significantly limits what type of acceleration is achievable. This lower bound applies to arbitrary methods  whose access to the feasible region is limited to a linear optimization oracle. As an illustration, let $\cx \defeq \{\vx \in \rr^n \mid \sum_{i =1}^n x_i = 1, \vx \geq 0\}$ be the probability simplex on $n$ coordinates and consider the problem
\begin{equation}\label{eq:LB}
\tag{LB}
\min_{\vx \in \cx} f(\vx) \equiv \norm{\vx}_2^2.
\end{equation}
If a first-order method has  access to $\cx$ only by means of a linear optimization oracle, then each query to the oracle reveals a single vertex of $\cx$ (a standard basis vector in the case of simplex), and it is guaranteed that after $k < n$ iterations the primal gap satisfies:
$$
f(\vx_k) - f(\vx^*) \geq \frac{1}{k} - \frac{1}{n}.
$$
In particular, for $k = n/2$ (assuming w.l.o.g.~that $n$ is even), the primal gap is  bounded below by $f(\vx_{n/2}) - f(\vx^*) \geq \frac{1}{n}$, and even the $O(1/k^2)$ rate cannot be achieved in full generality. 

Note that the objective in~\eqref{eq:LB} is also strongly convex. Thus, if we have a generic algorithm that is linearly convergent, contracting the primal gap as $f(\vx_k)  - f(\vx^*) \leq e^{-\rho k}(f(\vx_0)  - f(\vx^*))$ with a \emph{global rate} $\rho$, then it follows that $\rho \leq 2 \frac{\log n}{n}$ via the lower bound  above. The commonly used \emph{Fully-Corrective Frank-Wolfe} algorithm converges with (roughly) $\rho = \frac{1}{2n}$
~\cite{lacoste2015global}, and similar rates apply for \emph{Away-Step Frank-Wolfe} (AFW)~\cite{guelat1986some,lacoste2015global} and \emph{Pairwise Conditional Gradients} (PFW)~\cite{mitchell1974finding,lacoste2015global}. Given the lower bound, it follows  that up to logarithmic factors these \emph{global} rates cannot be further improved and, in particular, acceleration with the rate parameter $\rho = \sqrt{\nicefrac{1}{(2n)}}$ is not possible. 

At the same time, it is known that, e.g., AFW can be globally accelerated using the~\emph{Catalyst} framework~\cite{lin2015universal} (see also a discussion in \cite{lacoste2015global}). However, the obtained accelerated rate involves large dimension-dependent constants to be compatible with the lower bound, making the algorithm impractical. Another form of acceleration in the case of linear optimization based methods is achieved by \emph{Conditional Gradient Sliding}~\cite{lan2016conditional}. In~\cite{lan2016conditional}, the complexity is separated into calls to a first-order oracle and calls to the linear optimization oracle. While the optimal first-order oracle complexity of $O(1/\sqrt{\epsilon})$ is achieved in the case of smooth (non-strongly) convex minimization, the  linear optimization oracle complexity is $O(1/\epsilon)$, compatible with the lower bound.

\subsection{Contributions and Related Work}

We show that \emph{dimension-independent} acceleration for CG methods \emph{is possible} for smooth strongly convex minimization after a burn-in phase whose length does not depend on the target accuracy $\epsilon$ (but could potentially depend on the dimension). This allows for local acceleration, achieving the \emph{asymptotically optimal rate}. Our contributions are summarized as follows.

\paragraph{LaCG.} We introduce a new algorithmic framework for the class of conditional gradient methods, which we dub \emph{Locally Accelerated Conditional Gradients (LaCG)}. The framework couples active-set-based linearly convergent CG methods for smooth  strongly convex minimization (such as, e.g., AFW or PFW) with the accelerated algorithm $\mu$AGD+~\cite{cohen2018acceleration}. The coupling ensures that the resulting sequence of algorithm updates makes at least as much progress as the employed CG-type method, while being able to seamlessly transition to faster locally-accelerated convergence once suitable conditions are met and without ever having to explicitly test such conditions. 
LaCG 
achieves an asymptotically optimal iteration complexity of $K + O\left(\sqrt{\frac{L}{\mu}} \log \frac{1}{\epsilon}\right)$ to solve $\min_{\vx \in \cx}  f(\vx)$ up to accuracy $\epsilon$, where $f$ is $L$-smooth and $\mu$-strongly convex and $K$ is a constant that only depends on $\cx$ and $f$. 

Slightly simplifying, we achieve acceleration once we identify the optimal face of $\cx$ and are reasonably close to the optimal solution. Our reported complexity depends on the actual smoothness $L$ and strong convexity $\mu$ parameters of $f$ \emph{independent of the dimension of $\cx$}. This stands in contrast to the previously reported  linearly convergent methods, whose problem parameters had to be adjusted 
to account for the geometry of $\cx$, e.g., via the \emph{pyramidal width} (see \cite{lacoste2015global}) or \emph{smoothness and strong-convexity relative to the polytope} \cite{pena2018polytope,gutman2019condition}, which both bring in a dimension dependence. Note that such dimension-dependent terms are \emph{unavoidable} if global linear rates of convergence are sought, due to the aforementioned lower bound. 

We bypass the limitations of global linear convergence by accelerating the methods \emph{locally}: the faster convergence applies after a constant number of burn-in iterations with possibly weaker rates. 
Following the burn-in phase, the method requires no more than $O\left(\sqrt{\frac{L}{\mu}} \log \frac{1}{\epsilon}\right)$ queries to either the first-order oracle or the linear minimization oracle. 
This dimension-independent acceleration following a slower phase can be clearly observed in our numerical experiments, particularly in comparison with Catalyst-augmented AFW. 
As discussed before, Catalyst acceleration is necessarily weakened by a geometric correction leading to the (either first-order or linear optimization) oracle complexity of $O\left(\sqrt{\frac{L}{\mu}} \frac{D^2}{\delta^2} \log \frac{1}{\epsilon}\right)$, where $D$ is the diameter of $\cx$ and $\delta$ is the pyramidal width of $\cx$; see \cite{lacoste2015global}.


To achieve local acceleration, we assume that we can project relatively efficiently onto simplices spanned by sets of vertices of small size. However, while we do employ projections onto the convex hulls of maintained active sets, we stress that the feasible region is only accessed via a linear optimization oracle, i.e., our method is an LO-based method and the active sets are typically small, so that projection onto those sets is  cheap. In this sense of employing projections \emph{internally} but not over the original feasible region, our algorithm is similar to \emph{Conditional Gradient Sliding}~\cite{lan2017conditional}. 

\paragraph{Generalized Accelerated Method.} While there is an extensive literature on accelerated methods in optimization (see, e.g.,~\cite{betancourt2018symplectic,cohen2018acceleration,AXGD,tseng2008,beck2009fast,Bubeck2015,nesterov2018introductory,drusvyatskiy2016optimal,polyak1964some}), none of these approaches directly applies to local acceleration of CG. Most relevant to our work is~\cite{cohen2018acceleration}, and, in the process of constructing LaCG, we generalize the  $\mu$AGD+~\cite{cohen2018acceleration} algorithm in a few important ways (note that $\mu$AGD+ is also a generalization of Nesterov's method). 

First, we show that $\mu$AGD+  retains its convergence guarantees when coupled with an arbitrary alternative algorithm, where the coupling selects the point with the lower function value between the two algorithms, in each iteration. This is crucial for turning $\mu$AGD+ into a descent method and ensures that it makes at least as much progress per iteration as the CG-type method with which it is coupled. This coupling also allows us to achieve acceleration without any explicit knowledge of the parameters of the polytope $\cx$ or the position of the function minimizer $\vx^*.$ 

Second, we show that $\mu$AGD+  tolerates inexact projections. While this is not surprising, as similar results have been shown in the past for proximal methods~\cite{schmidt2011convergence}, this generalization of $\mu$AGD+ is a necessary ingredient to ensure that the practical per-iteration complexity of LaCG does not become too high. 

Finally, we prove that $\mu$AGD+ converges to the optimal solution at no computational loss even if the convex set on which the projections are performed changes between the iterations, as long as the convex set in iteration $k$ is contained in the convex set from iteration $k-1$ and it contains the minimizer $\vx^*.$ We are not aware of any other results of this type. Note that this result allows us to update the projection simplex with each update of the active set, and, as vertices are dropped from the active set by away steps in AFW or PFW, the iterations become less expensive.

\paragraph{Computational Experiments.} We compare our methods to other CG methods as well as a Catalyst-augmented AFW and provide computational evidence that our algorithms achieve a practical speed-up, both in per-iteration progress and in wall-clock time, significantly outperforming state of the art methods.

\section{Preliminaries}
\label{sec:prelim}
We consider problems of the form 
$
\min_{\vx \in \cx} f(\vx),
$
where $f$ is a smooth (gradient Lipschitz) strongly convex function and $\cx \subseteq \rr^n$ is a convex polytope. 
We assume (i) first-order access to $f$: given $\vx \in \cx$, we can compute $f(\vx)$ and $\nabla f(\vx)$, and (ii) linear optimization oracle access to $\cx$: given a vector $\vc \in \rr^n$, we can compute
$
\vv = \argmin_{\vu \in \cx} \innp{\vc, \vu}.
$

Let $\norm{\cdot}$ be the Euclidean norm and let $\cb(\vx,r)$ denote the ball around $\vx$ with radius $r$ with respect to $\norm{\cdot}$. We say that $\vx$ is \emph{$r$-deep} in a convex set $\cx \subseteq \rr^n$ if $\cb(\vx,2r) \cap \Aff(\cx) \subseteq \cx$. The point $\vx$ is contained in the \emph{relative interior of $\cx$}, written as $\vx \in \relinterior(\cx)$, if there exists an $r>0$ such that $\vx$ is $r$-deep in $\cx$; if $\Aff(\cx) = \rr^n$, then $\vx$ is contained in the \emph{interior of $\cx$}, written as $\vx \in \interior(\cx)$. 

Further, given a polytope $\cx$, let $\vertex(\cx) \subseteq \cx$ denote the (finite) set of vertices of $\cx$ and, given a point $\vx \in \cx,$ let $\ct(\vx)$ denote the smallest face of $\cx$ containing $\vx$, which is defined as a subset of $\vertex(\cx)$ of minimal cardinality whose convex hull contains $\vx$. Finally, let $\Delta_n \defeq \{\vx \in \rr^n \mid \sum_{i = 1}^n x_i = 1, x \geq 0\} \subseteq \rr^n $ denote the \emph{probability simplex of dimension $n$}. 

\subsection{Conditional Gradient Descent}
We provide a very brief introduction to the \emph{Conditional Gradient Descent algorithm} \cite{levitin1966constrained} (also known as \emph{Frank-Wolfe algorithm}; see \cite{frank1956algorithm}). Assume that $f$ is $L$-smooth with $L < \infty$. The Frank-Wolfe algorithm with step sizes $\eta_k \in [0, 1]$ is defined via the following updates:
\begin{equation}\label{eq:standard-CG-step}
\vx_{k+1} = (1-\eta_k) \vx_k + \eta_k \vv_k = \vx_k + \eta_k(\vv_k - \vx_k),
\end{equation}
where $\vx_0 \in \cx$ is an arbitrary initial point from the feasible set $\cx$ and $\vv_k$ is computed using the linear optimization oracle as:  
$$
\vv_k = \argmin_{\vu \in \cx} \innp{\nabla f(\vx_{k}, \vu}.
$$

\subsection{Approximate Duality Gap Technique}\label{sec:ADGT-description}

To analyze the convergence of LaCG, we employ the \emph{Approximate Duality Gap Technique (ADGT)} \cite{thegaptechnique}. The core idea behind ADGT is to ensure that $A_k G_k$ is non-increasing with iteration count $k,$ where $G_k$ is an upper approximation of the optimality gap (namely, $f(\vx_k) - f(\vx^*) \leq G_k$, where $\vx_k$ is the solution point output by the algorithm at iteration $k$), and $A_k$ is a positive strictly increasing function of $k$. If such a condition is met, we immediately have $A_k G_k \leq A_0G_0,$ which implies $f(\vx_k) - f(\vx^*) \leq \frac{A_0 G_0}{A_k}.$ Thus, as long as $A_0 G_0$ is bounded (it typically corresponds to some initial distance to the minimizer $\vx^*$), we have that the algorithm converges at rate $1/A_k.$ This also means that, to obtain the highest rate of convergence, one should always aim to obtain the fastest-growing $A_k$ for which it holds that $A_k G_k \leq A_{k-1}G_{k-1},$ for all $k.$

The approximate gap $G_k$ is defined as the difference of an upper bound $U_k$ on the function value at the output point $\vx_k,$ $U_k \geq f(\vx_k),$ and a lower bound $L_k$ on the minimum function value, $L_k \leq f(\vx^*).$ Clearly, this choice ensures that $f(\vx_k) - f(\vx^*) \leq G_k.$ In all the algorithms analyzed in this paper, we will use $U_k = f(\vx_k)$. The lower bound requires more effort; however, it is similar to those used in previous work~\cite{cohen2018acceleration,thegaptechnique}, and its detailed construction is provided in Appendix~\ref{appx:adgt}.  

Because of its generality, ADGT is well-suited to our setting, as it allows coupling different types of steps (different variants of CG and accelerated steps) and performing a more fine-grained and local analysis than typical approaches. Further, it allows accounting for inexact minimization oracles invoked as part of the algorithm subroutines in a generic way. 

\section{LaCG Framework}
\label{sec:accg}

In this section, we establish our main result. 
We first consider
a simple case in which 
$\vx^* \in \interior(\cx )$ as a
warm-up to explain our approach and derive LaCG specifically for this case (Section~\ref{sec:warm-optim-inter}). Then, in
Section~\ref{sec:fullResult}, we consider the more general
case where $\vx^* \in \relinterior(\cf)$ with $\cf$ being a face of
$\cx$. Together  Section~\ref{sec:warm-optim-inter},
this covers all cases of interest (except some degenerate cases).
Further, the general LaCG framework presented in Section~\ref{sec:fullResult} applies to either of the two cases.

\subsection{Warm-up: Optimum in the Interior of $\cx$}
\label{sec:warm-optim-inter}
In the case where the optimum is contained in the interior of $\cx,$
we have $\nabla f(\vx^*) = \zeros$. As such, the unconstrained optimum
and the constrained optimum coincide. Thus, one might be tempted to
assert that there is no need for an accelerated CG algorithm
in this case. However, whether the optimum is contained in the
interior is not known \emph{a priori}. The presented algorithm is adaptive, as
it accelerates if $\vx^* \in \interior(\cx),$ and
otherwise it converges with the standard $1/k$ rate.  

The main idea can be summarized as follows. Suppose that 
$\vx^*$ is contained $2r$-deep in the \emph{interior} of $\cx.$ Due to the function's strong convexity, any method that contracts the optimality gap $f(\vx_k) - f(\vx^*)$ over $k$ also contracts the distance $\|\vx_k - \vx^*\|.$ In particular, this follows by:
$
\frac{\mu}{2}\|\vx_k - \vx^*\|^2 \leq f(\vx_k) - f(\vx^*).
$ 
%
Hence, roughly, after an iterate $\vx_k$ is guaranteed to be inside the $2r$-ball, we can switch to a faster (accelerated) method for \emph{unconstrained} minimization. This idea, however, requires a careful formalization, for the following reasons:
\begin{enumerate}
\item In general, we cannot assume that the algorithm has knowledge of
  $r$ and $D$, or access to information on $\|\vx_k - \vx^*\|,$ as $\vx^*$
  is unknown -- it is what the algorithm is trying to determine.
    \item The algorithm should be able to converge even if $r = 0$;
      we cannot in general assume that it is a priori known that
      $\vx^* \in \interior(\cx).$ Because, if that were the case and
      we knew that $r \geq \epsilon$, 
     we would be able to run an
      accelerated algorithm for $O(\sqrt{L/\mu}\log(1/\epsilon))$
      iterations, without any need to worry about whether the solution
      that the algorithm outputs belongs to $\cx.$
\end{enumerate}

We show that both issues can be resolved by implementing a monotonic
version of a hybrid algorithm that can choose at each iteration
whether to perform an accelerated step or a CG step from Eq.~\eqref{eq:standard-CG-step}. Monotonicity is crucial to ensure contraction of the
distance to the optimal solution. In \emph{this subsection only}, we
assume that the algorithm has access to a membership
oracle for $\cx$ -- namely, that for any point $\vx$ it can
determine whether $\vx \in \cx$. This is generally a mild assumption,
especially when $\cx$ is a polytope, which is a standard setting for
CG.\footnote{For generic LP solvers applied to
  polynomially-sized LPs, checking membership amounts to evaluating
  the linear (in)equalities describing $\cx$, which is typically
   much cheaper than linear optimization. For structured LPs that are solvable in polynomial time
  but have exponential representation (e.g., matching over
  non-bipartite graphs~\cite{rothvoss2017matching}), there typically
  exist 
  membership oracles with running times
  that are comparable to the running time of the corresponding
  minimization oracle used in CG steps
  (e.g.,~\cite{fleischer2006polynomial}).} 
The convergence of the resulting algorithm is summarized in the following theorem. Full technical details are deferred to Appendix~\ref{appx:omitted-proofs}.
\begin{restatable}{theorem}{prelimthm}
Let $\vx_k$ be the solution output by Algorithm~\ref{algo:acc-FW} (Appendix~\ref{appx:prelim-algo}) for $k \geq 1$. If:
$$
k \geq \min \left\{ \frac{2LD^2}{\epsilon},\; \frac{LD^2}{\mu r^2} + \sqrt{\frac{L}{\mu}} \log \left( \frac{2(L+\mu) r^2}{\mu\epsilon}\right)\right\},
$$
then 
$
f(\vx_k) - f(\vx^*) \leq \epsilon.
$
\end{restatable}

\subsection{Optimum in the Relative Interior of a Face of $\cx$}
\label{sec:fullResult}

We now formulate the general case that subsumes the case from the previous subsection. Note that when $\vx^*$ is in the relative interior, it may not be the case that $\nabla f(\vx^*) = \zeros$ anymore; for an example of when this happens, consider the lower bound instance~\eqref{eq:LB} from Section~\ref{sec:lower-bound}. \ifaistats Due to space constraints, w\else W\fi e only state the main ideas here, while full technical details are deferred to Appendix~\ref{appx:fullResult}.  

We assume that, given points $\vx_1, ..., \vx_m$ and a point $\vy,$ the following  problem is easily solvable:
\begin{equation} \label{eq:convProject}
  \min_{\substack{\vu = \sum_{i=1}^m \lambda_i \vx_i, \\ \vlambda \in \Delta_m}} \frac{1}{2}\|\vu - \vy\|^2. \end{equation}
In other words, we assume that the projection onto the convex hull of a given set of vertices can be implemented efficiently; however, we do not require access to a membership oracle anymore. Solving this problem amounts to minimizing a quadratic function over the probability simplex.  The size of the program $m$ from Eq.~\eqref{eq:convProject} corresponds to the size of the active set of the CG-type method employed within LaCG. Note that $m$ is never larger than the iteration count $k,$ and is often much lower than the dimension of the original problem. Further, there exist multiple heuristics for keeping the size of the active set small in practice (see, e.g.,~\cite{BPZ2017}). The projection from Eq.~\eqref{eq:convProject} does not require access to either the first-order oracle or the linear optimization oracle. Finally, due to Lemma~\ref{lemma:modified-agdp} stated below, we only need to solve this problem to accuracy of the order $\frac{\epsilon}{\sqrt{\mu L}}$, where $\epsilon$ is the target accuracy of the program. 

For simplicity, we illustrate the framework using AFW as the coupled CG method. However, the same ideas can be applied to other active-set-based methods such as PFW in a straightforward manner. Unlike in the previous subsection, the assumption that $\cx$ is a polytope is crucial here, as the linear convergence for the AFW algorithm established in~\cite{lacoste2015global} relies on a
constant, the \emph{pyramidal width}, that is only known to be bounded away
from $0$ for polytopes. For completeness, we provide the pseudocode for one iteration of
AFW (as stated in~\cite{lacoste2015global}) in Algorithm~\ref{algo:away-FWAppx},  Appendix~\ref{appx:fullResult}.

A simple observation, which turns out to be key for the coupling to work is that when running  AFW, there exists an iteration \(K_0\)  such that for all $k \geq K_0$ it holds $\vx^* \in \co(\cs_k^{\mathrm{AFW}})$, where $\cs_k^{\mathrm{AFW}}$ denotes the active set maintained by AFW in iteration $k$. This iteration \(K_0\) only depends on the feasible region \(\cx\) and \(\vx^*\) and, as such, it is a burn-in period of constant length.

To achieve local acceleration, we couple the AFW steps with a modification of the $\mu$AGD+ algorithm~\cite{cohen2018acceleration} that we introduce here. Unlike its original version~\cite{cohen2018acceleration}, the version  provided here (Lemma~\ref{lemma:modified-agdp}) allows coupling of the method with an arbitrary  sequence of points from the feasible set, it supports inexact minimization oracles, and it supports changes in the convex set (which correspond to active sets from AFW) on which projections are performed. These modifications are crucial to being able to achieve local acceleration without any additional knowledge about the polytope or the position of the minimizer $\vx^*.$ Further, we are not aware of any other methods that allow changes to the feasible set as described here, and, thus, the result from Lemma~\ref{lemma:modified-agdp} may be of independent interest.

%
\begin{restatable}{lemma}{keylemma}\label{lemma:modified-agdp}
(Convergence of the modified $\mu$AGD+) Let $f: \cx \rightarrow \rr$ be $L$-smooth and $\mu$-strongly convex, and let $\cx$ be a closed convex set.  Let $\vx^* = \argmin_{\vu \in \cx}f(\vx),$ and let $\{\cc_i\}_{i=0}^k$ be a sequence of convex subsets of $\cx$ such that $\cc_i \subseteq \cc_{i-1}$ for all $i$ and $\vx^* \in \bigcap_{i=0}^k \cc_i.$ Let $\{\Tilde{\vx}_i\}_{i=0}^k$ be any (fixed) sequence of points from $\cx.$ Let $a_0 = 1,$ $\frac{a_k}{A_k} = \theta$ for $k \geq 1,$ where $A_k = \sum_{i=0}^k a_i$ and $\theta = \sqrt{\frac{\mu}{2L}}.$ 
Let $\vy_0 \in \cx$, $\vx_0 = \vw_0,$ and $\vz_0 = L\vy_0 -  \nabla f(\vy_0)$. For $k \geq 1,$ define 
iterates $\vx_k$ by:
\begin{equation}\label{eq:modified-magdp}
    \begin{gathered}
        \vy_k = \frac{1}{1+\theta} \vx_{k-1} + \frac{\theta}{1+\theta} \vw_{k-1},\\
        \vz_k = \vz_{k-1} -a_k \nabla f(\vy_k) + \mu a_k \vy_k,\\
        \vxh_k = (1-\theta) \vx_{k-1} + \theta \vw_k,\\
        \vx_k = \argmin \{f(\vxh_k), \, f(\Tilde{\vx}_k)\}
    \end{gathered}
\end{equation}
where, for all $k \geq 0$, $\vw_k$ is defined as an $\epsilon^m_k$-approximate solution of:
\begin{equation}\label{eq:magdp-equiv-min}
    \min_{\vu \in \cc_k} \Big\{ -\innp{\vz_k, \vu} + \frac{\mu A_k + \mu_0}{2}\|\vu\|^2 \Big\},
\end{equation}
with $\mu_0 \defeq L - \mu$. Then, for all $k \geq 0,$ $\vx_k \in \cx$ and:
\begin{align*}
 f(\vx_k) - f(\vx^*)  \leq &\;(1-\theta)^k \frac{(L-\mu)\|\vx^* - \vy_0\|^2}{2} \\ 
 & + \frac{2\sum_{i=0}^{k-1}\epsilon^m_i + \epsilon^m_k}{A_k}.   
\end{align*}
\end{restatable}

To obtain locally accelerated convergence, we  show that from some iteration onwards, we can apply the accelerated method from
Lemma~\ref{lemma:modified-agdp} with $\cc_k$ being the convex hull of the vertices from the active set and the
sequence $\Tilde{\vx}_k$ being the sequence of the AFW steps. The pseudocode for the LaCG-AFW algorithm is provided in
Algorithm~\ref{algo:acc-FW-rel-int}. For completeness, pseudocode for one iteration of the accelerated method (ACC), based on Eq.~\eqref{eq:modified-magdp} is provided in Algorithm~\ref{algo:acc-stepAppx} (Appendix~\ref{appx:fullResult}).

\begin{algorithm*}[tbh]
\caption{Locally Accelerated Conditional Gradients with Away-Step Frank-Wolfe (LaCG-AFW)}
\label{algo:acc-FW-rel-int}
\begin{algorithmic}[1]
\State Let $\vx_0 \in \cx$ be an arbitrary point, $\cs_0^\mathrm{AFW} = \{\vx_0\}$, $\vlambda_0^\mathrm{AFW} = [1]$
\State Let $\vy_0 = \vxh_0 = \vw_0 = \vx_0$, $\vz_0 = -\nabla f(\vy_0) + L \vy_0$, $\cc_1 = \mathrm{co}(\cs_0^\mathrm{AFW})$
\State $a_0 = A_0 = 1,$ $\theta = \sqrt{\frac{\mu}{2L}},$ $\mu_0 = L-\mu$
\State $H = \frac{2}{\theta}\log(1/(2\theta^2) - 1)$ \Comment{Minimum restart period}
\State $r_f = \false$, $r_c = 0$ \Comment{Restart flag and restart counter initialization}
\For{$k=1$ to $K$} 
\State $\vx_k^{\mathrm{AFW}}, \, \cs_k^{\mathrm{AFW}}, \, \vlambda_k^{\mathrm{AFW}} = \mathrm{AFW}(\vx_{k-1}^{\mathrm{AFW}}, \, \cs_{k-1}^{\mathrm{AFW}}, \, \vlambda_{k-1}^{\mathrm{AFW}})$ \Comment{Independent AFW update}
\State $A_k = A_{k-1}/(1-\theta),$ $a_k = \theta A_k$
\State $\vxh_k,\, \vz_k,\, \vw_k = \mathrm{ACC}(\vx_{k-1}, \vz_{k-1}, \vw_{k-1}, \mu, \mu_0, a_k, A_k, \cc_k)$
\If{$r_f$ and $r_c \geq H$}
\Comment{Restart
  criterion is met}  \label{restart:project}
\State $\vy_k = \argmin\{f(\vx_k^{\mathrm{AFW}}),\, f(\vxh_k)\}$ 
\State $\cc_{k+1} = \mathrm{co}(\cs_k^{\mathrm{AFW}})$ \Comment{Updating feasible set for the accelerated sequence}
\State $a_k = A_k = 1$, $\vz_k = -\nabla f(\vy_k) + L\vy_k$ \Comment{Restarting accelerated sequence}
\State $\vxh_k = \vw_k = \argmin_{\vu \in \cc_{k+1}}\{-\innp{\vz_k, \vu} + \frac{L}{2}\|\vu\|^2\}$
\State $r_c = 0$, $r_f = \false$ \Comment{Resetting the restart indicators}
\Else
\If{$\cs_k^{\mathrm{AFW}} \setminus \cs_{k-1}^{\mathrm{AFW}} \neq \emptyset$} \Comment{If a vertex was added to the active set}
\State $r_f = \true$ \Comment{Raise restart flag}
\EndIf
\If{$r_f = \false$} \Comment{If AFW did not add a vertex since last restart}
\State $\cc_{k+1} = \mathrm{co}(\cs_k^{\mathrm{AFW}})$ \Comment{Update the feasible set}
\Else 
\State $\cc_{k+1} = \cc_k$ \Comment{Freeze the feasible set}
\EndIf
\EndIf
\State $\vx_k = \argmin\{f(\vx_k^{\mathrm{AFW}}), \, f(\vxh_k),\, f(\vx_{k-1})\}$ \Comment{Choose the better step + monotonicity} \label{Monotonicity2}
\State $r_c = r_c + 1$ \Comment{Increment the restart counter}
\EndFor
\end{algorithmic}
\end{algorithm*}

 Our main theorem is stated below, with a proof sketch. The full proof is deferred to Appendix~\ref{appx:fullResult}.

\begin{restatable}{maintheorem}{mainthm}(Convergence analysis of \emph{Locally Accelerated Frank-Wolfe})
\label{thm:main}
  Let $\vx_k$ be the solution output by
  Algorithm~\ref{algo:acc-FW-rel-int} and $r_0$ be the critical radius (see Fact~\ref{fact:asConvergence} in Appendix~\ref{appx:fullResult}).
    If:
\begin{align*}
k \geq \min \bigg\{& \frac{8 L}{\mu} \Big( \frac{D}{\delta} \Big)^2 \log \Big(  \frac{f(\vx_0) - f(\vx^*)}{\epsilon} \Big) , \\
 & K_0 + H + 2\sqrt{\frac{2L}{\mu}} \log \left( \frac{(L-\mu) {r_0}^2 }{2\epsilon}\right)\bigg\}, 
\end{align*}
where $H 
= 2\sqrt{\nicefrac{2L}{\mu}}\log(L/\mu -1 )$ and $K_0 = \frac{8 L}{\mu} \left( \frac{D}{\delta} \right)^2 \log \left(
  \frac{2(f(\vx_0) - f(\vx^*))}{\mu {r_0}^2} \right),$ then:
$$
f(\vx_k) - f(\vx^*) \leq \epsilon.
$$
\end{restatable}
\begin{proof}[Proof Sketch]

The statement of the theorem is a direct consequence of the following observations about Algorithm~\ref{algo:acc-FW-rel-int}. First, observe that the AFW algorithm is run independently of the accelerated sequence, and, in particular, the accelerated sequence has no effect on the AFW-sequence whatsoever. Further, in any iteration, the set $\cc_k$ that we project onto is the convex hull of some active set $\cs_i^{\mathrm{AFW}} \subseteq \cx$ for some $0\leq i \leq k-1$ implying $\vxh_k \in \cx$ -- each $\vxh_k$ is hence feasible. 

Now, as in any iteration $k$ the solution outputted by the algorithm is $\vx_k = \argmin\{f(\vx_k^{\mathrm{AFW}}), \, f(\vxh_k)\},$ the algorithm never makes less progress than AFW. This immediately implies (by a standard AFW  guarantee; see~\cite{lacoste2015global} and Proposition~\ref{proposition:BurnIn-Phase}) that for $k \geq  \frac{8 L}{\mu} \left( \frac{D}{\delta} \right)^2 \log \big(  \frac{f(\vx_0) - f(\vx^*)}{\epsilon} \big)$, it must be that $f(\vx_k) - f(\vx^*) \leq \epsilon$, which establishes the unaccelerated part of the minimum in the asserted rate. 

Further, there exists an iteration $K \leq K_0$ such that for all $k \geq K$ it holds $\vx^* \in \co(\cs_k^{\mathrm{AFW}})$ (see Proposition~\ref{proposition:BurnIn-Phase}). Let $K$ be the first such iteration. Then, the AFW  algorithm must have added a vertex in iteration $K$ as otherwise $\vx^* \in \co(\cs_{k-1}^{\mathrm{AFW}})$, contradicting the minimality of $K$. 
Due to the restarting criterion from Algorithm~\ref{algo:acc-FW-rel-int}, a restart must happen by iteration $K_0 + H.$ 
Thus, for $k \geq K_0 + H$, it must be $\vx^* \in \cc_k$.

The rest of the proof invokes Lemma~\ref{lemma:modified-agdp} and its corollary (Corollary~\ref{cor:restarts}), which ensures accelerated convergence following iteration $K_0 + H.$
\end{proof}

\begin{remark}[Inexact projection oracles.]
For simplicity, we stated Theorem~\ref{thm:main} assuming exact minimization oracle ($\epsilon_i^m = 0$ in Lemma~\ref{lemma:modified-agdp}). Clearly, it suffices to have $\epsilon_i^m = \frac{a_i}{8}\epsilon$ and invoke Theorem~\ref{thm:main} for target accuracy $\epsilon/2.$
\end{remark}

\begin{remark}[Further improvements to the practical performance.] 
If in any iteration the Wolfe gap of the accelerated sequence on $\cc_k$, $\max_{\vu \in \cc_k}\innp{\nabla f(\vx_k), \vx_k - \vu},$ is smaller than the target accuracy of the projection subproblem (order-$\frac{\epsilon}{\sqrt{\mu L}}$), then $f$ cannot be reduced by more than order-$\frac{\epsilon}{\sqrt{\mu L}}$ on $\cc_k,$ and one can safely perform an early restart without affecting the theoretical convergence guarantee. 
\end{remark}
  
\begin{remark}[Running Algorithm~\ref{algo:acc-FW-rel-int} when $\vx^* \in \interior(\cx)$]
  Usually we do not know ahead of time whether
  $\vx^* \in \interior(\cx)$ or whether $\vx^*$ is in the relative
  interior of a face of $\cx$. However, we can simply run
  Algorithm~\ref{algo:acc-FW-rel-int} agnostically, as in the case where
  $\vx^* \in \interior(\cx)$ we still exhibit local acceleration with
  an argumentation and convergence analysis analogous to the one in
  Section~\ref{sec:warm-optim-inter}. In particular, the assumptions
  of Section~\ref{sec:fullResult} are only needed to establish a bound
  for the estimation in Proposition~\ref{proposition:BurnIn-Phase}.
\end{remark}

\begin{remark}[Variant relying exclusively on a linear optimization oracle]
Similar as in the Conditional Gradient Sliding (CGS) algorithm \cite{lan2016conditional} we can also solve the projection problems using (variants of) CG. The resulting algorithm is then fully projection-free similar to CGS. 
In fact, a variant of CGS is recovered if we  ignore the AFW steps and only run the accelerated sequence with such projections realized by CG. 
\end{remark}


\section{Computational Results}
\label{sec:comp}

The theoretical results from the previous section showed that the LaCG framework leads to the asymptotically optimal local convergence on the class of smooth strongly convex optimization problems over polytopes. However, this acceleration comes at the cost of more computationally-intensive iterations, as each iteration performs projections onto the convex hull of the active set. 
In this section, we show that the improvement in the convergence rate is sufficiently high to offset the increased cost of iterations. In most settings, LaCG outperforms other CG variants not only in terms of the  per-iteration progress, but also in wall-clock time. 


We implemented Algorithm~\ref{algo:acc-FW-rel-int} in Python 3, employing the $\mathcal{O}\left( n\log n \right)$ projections onto the simplex described in \cite[Algorithm 1]{duchi2008efficient} and Nesterov accelerated gradient descent \cite{nesterov2018introductory} to solve the projection subproblems. 
As mentioned before, LaCG can be used with any CG variant that maintains an active set; 
this excludes e.g., \emph{Decomposition invariant CG (DiCG)} variants (e.g., \cite{LDLCC2016,bashiri2017decomposition}). Additional information can be found in Appendix~\ref{appx:compresults}. 

Results for each example are shown in a figure with four plots. The first two plots compare  standard LaCG variants to AFW, PFW, and DiCG (if applicable) in iterations (log-log scale) and wall-clock time (log scale). The second two plots compare the standard and lazified\footnote{See the end of the section for details; lazified algorithms are denoted by (L) in the plots.} variants of LaCG to CGS and Catalyst in iterations (log-log scale) and wall-clock time (log scale). As shown in Figs.~\ref{fig:Birkhoffall-examples}-\ref{fig:simplex}, LaCG can make up for more expensive iterations thanks to its higher convergence rate following the burn-in phase, and it outperforms other methods in many important examples. 

There are also examples in which LaCG is outperformed by other methods; in particular, by DiCG (see Appendix~\ref{appx:compresults}). This happens when the linear optimization oracle is very efficient, while the active sets can get large. In this case, DiCG outperforms all active set based methods. Note that, however, DiCG is not broadly applicable; see, e.g., the discussion for the structured regression example below. 
Further, in the examples where the cost is not dominated by the cost of maintaining the active set (see Fig.~\ref{fig:Birkhoffall-examples} and Fig.~\ref{fig:simplex}), LaCG outperforms DiCG. 

\paragraph{Birkhoff polytope.}

The first example, shown in  Fig.~\ref{fig:Birkhoffall-examples}\subref{fig:birkhoff1-it-cnt}-\ref{fig:Birkhoffall-examples}\subref{fig:birkhoff2-time}, corresponds to minimization over the Birkhoff polytope in $\mathbb{R}^{n\times n}$ where $n^2=1600$  with $f(\vx) = \vx^T\frac{M^TM+I}{2}\vx$, $M\in \mathbb{R}^{n\times n}$ is a sparse matrix whose $1\%$ of the elements are drawn from a standard Gaussian distribution, and $I$ is the identity matrix (the matrix $M^TM$ has $15\%$ non-zero elements). The resulting condition number is $L/\mu \approx 100.$ The linear optimization oracle in this instance uses the Hungarian algorithm, which has complexity  $\mathcal{O}(n^{3/2})$.

\paragraph{Structured Regression.}

The second example, shown in Fig.~\ref{fig:MIPLIBall-examples}\subref{fig:MIPLIB1-it-cnt}-\ref{fig:MIPLIBall-examples}\subref{fig:MIPLIB2-time}, corresponds to a structured regression problem over the convex hull of the feasible region defined by integer and linear constraints, where the linear optimization oracle corresponds to solving a mixed integer program (MIP) with a linear objective function. The specific instance used corresponds to \texttt{ran14x18-disj-8}, of dimension $504$, from the MIPLIB library. The objective function $f(\vx) = \vx^T\frac{M}{2}\vx + \vb$ was obtained by first generating an orthonormal basis $\mathcal{B} = \{ \vu_1, \cdots, \vu_n \}$ in $\rr^n$ and a set of $n$ uniformly distributed values $\{ \lambda_1, \cdots, \lambda_n \}$ between $\mu$ and $L$ and setting $M = \sum_{i=1}^n \lambda_i \vu_i \vu_i^T$. The vector $\vb$ was set to be outside of the feasible region, and has random entries uniformly distributed from $0$ to $1$. Note that DiCG
~\cite{LDLCC2016,bashiri2017decomposition} is not applicable here, as the representation of the MIPLIP polytope mixes continuous and integer variables, so that the away step oracle cannot be readily implemented. In fact, for polytopes over which linear optimization is NP-hard (e.g., TSP polytope), efficiently computing away steps with an away step oracle as in~\cite{bashiri2017decomposition} is not possible unless \(\text{NP} = \text{co-NP}\).

\begin{figure*}[ht!]
    \centering
    \vspace{-10pt}
    \hspace{\fill}
    \subfloat[Iteration]{{\includegraphics[width=3.95cm]{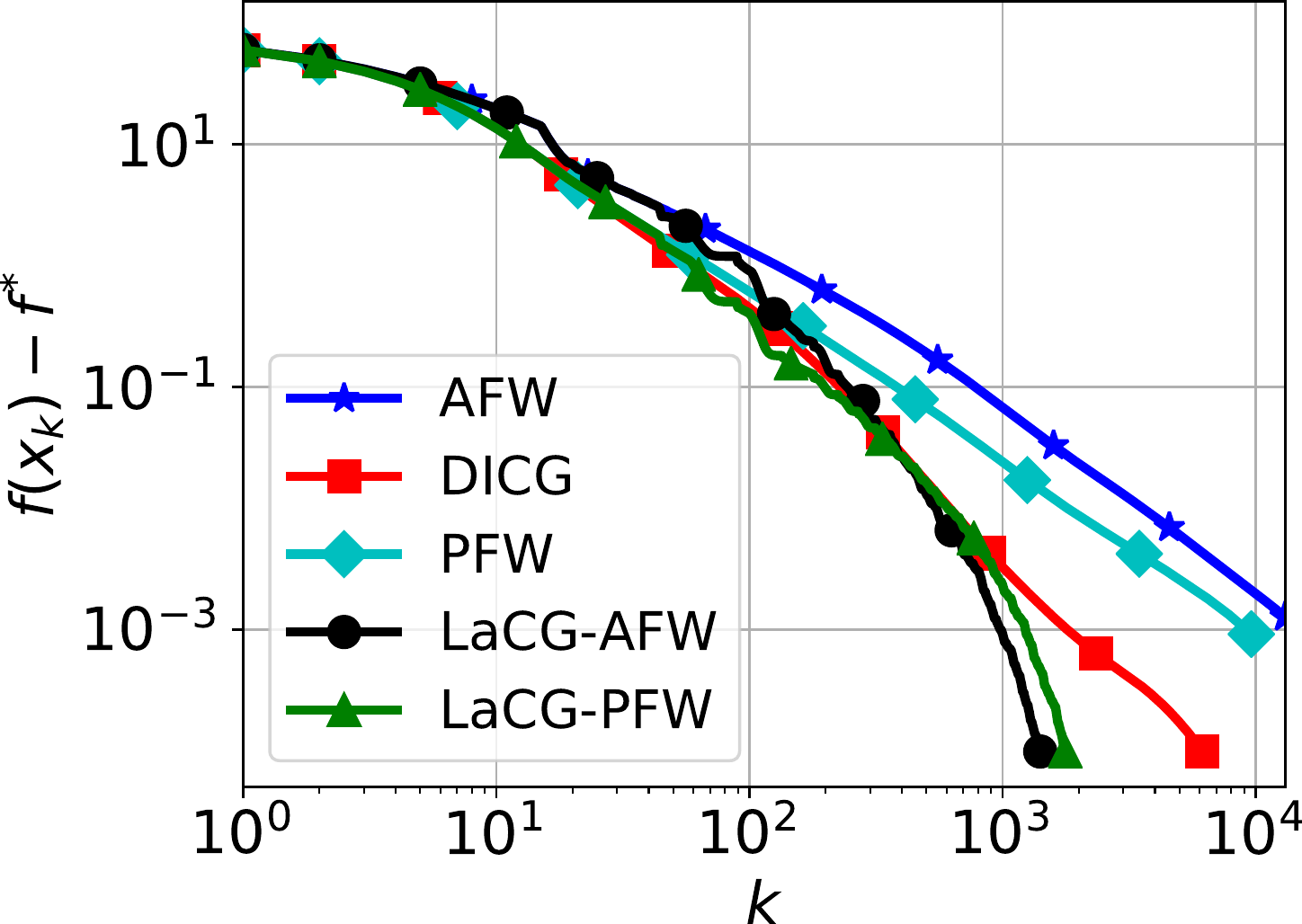} }\label{fig:birkhoff1-it-cnt}}%
    \hspace{\fill}
    \subfloat[Wall-clock time]{{\includegraphics[width=3.85cm]{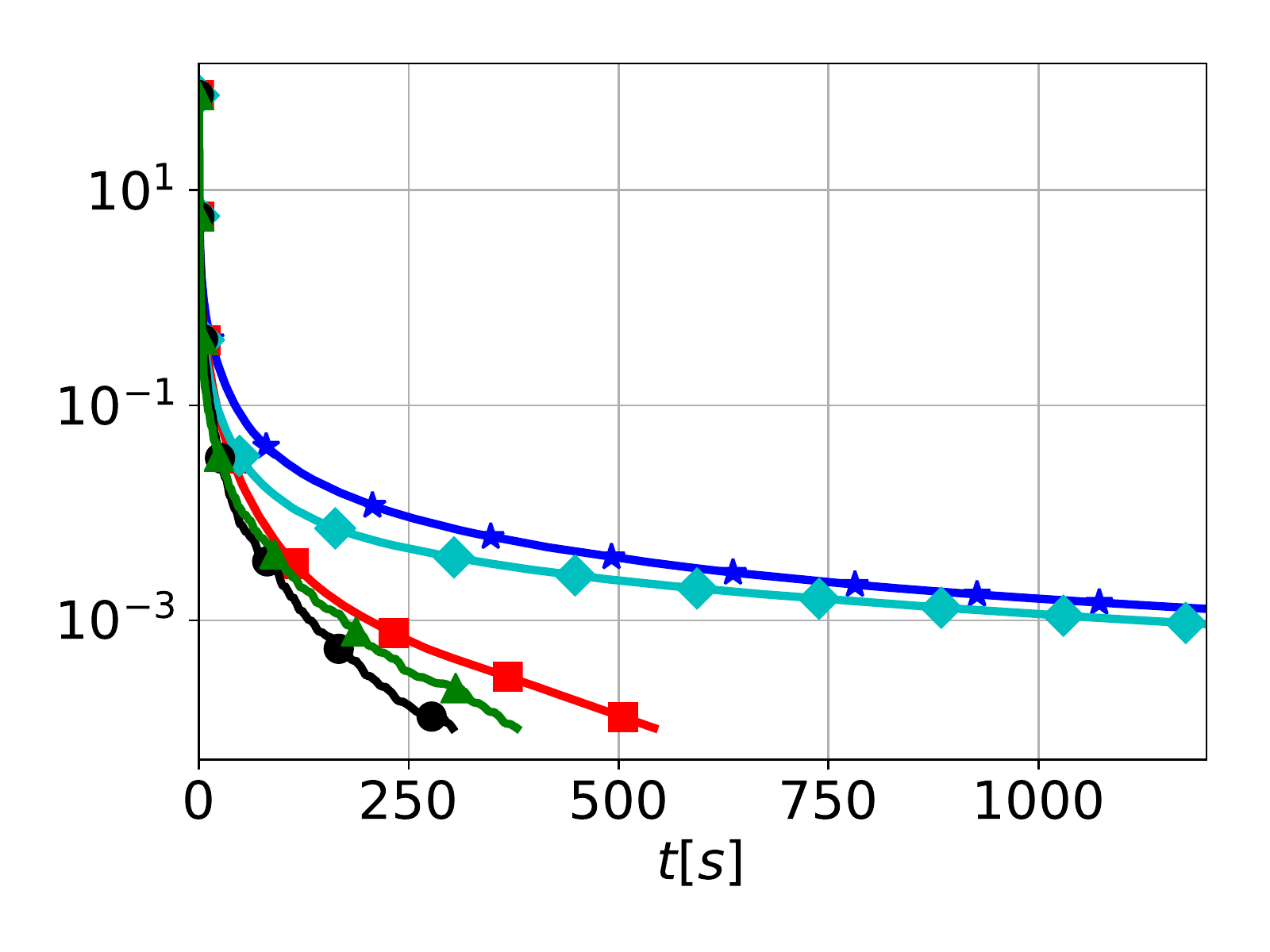} }\label{fig:birkhoff1-time}}%
    \hspace{\fill}
    \subfloat[Iteration]{{\includegraphics[width=3.9cm]{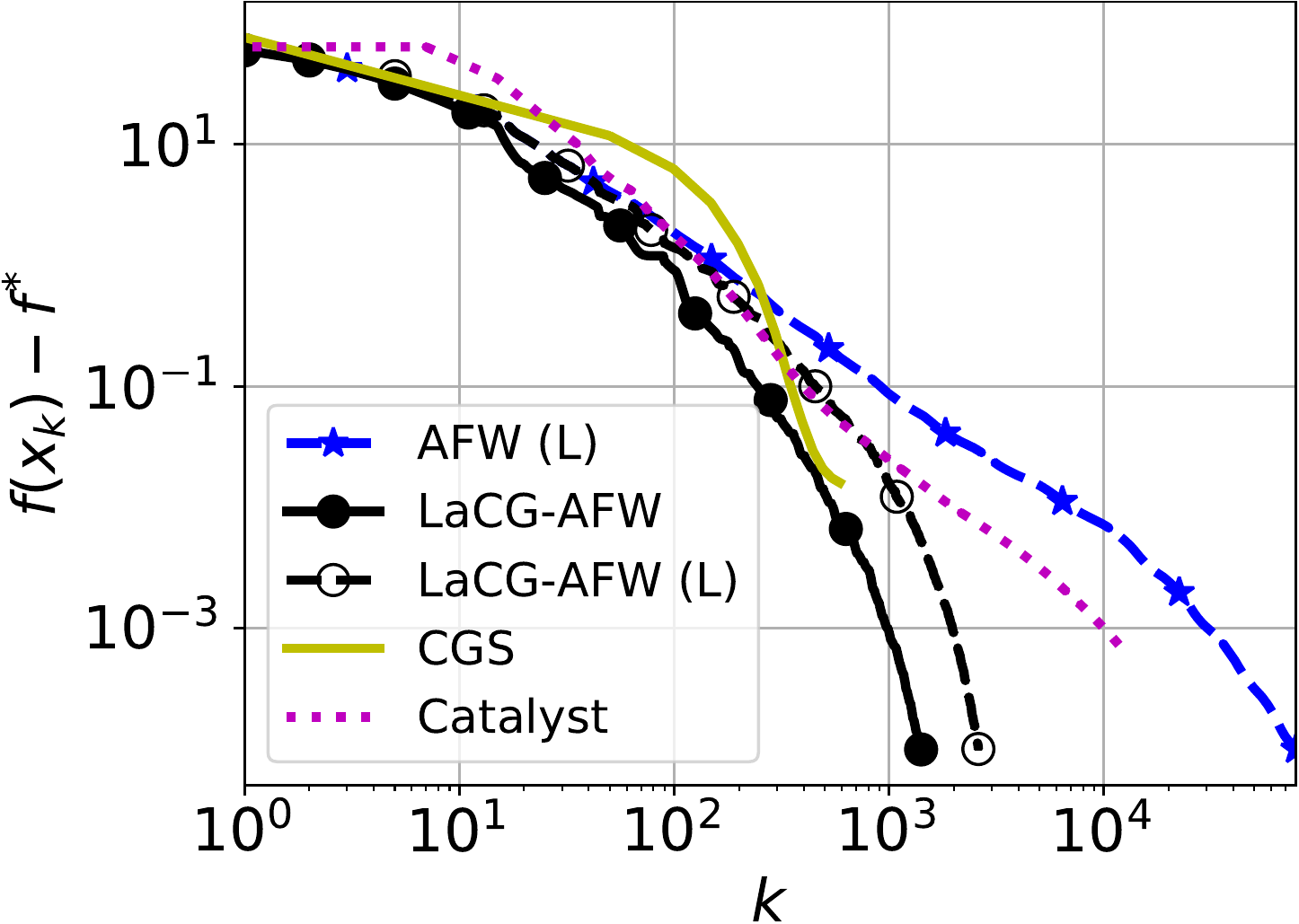} }\label{fig:birkhoff2-it-cnt}}%
    \hspace{\fill}
    \subfloat[Wall-clock time]{{\includegraphics[width=3.85cm]{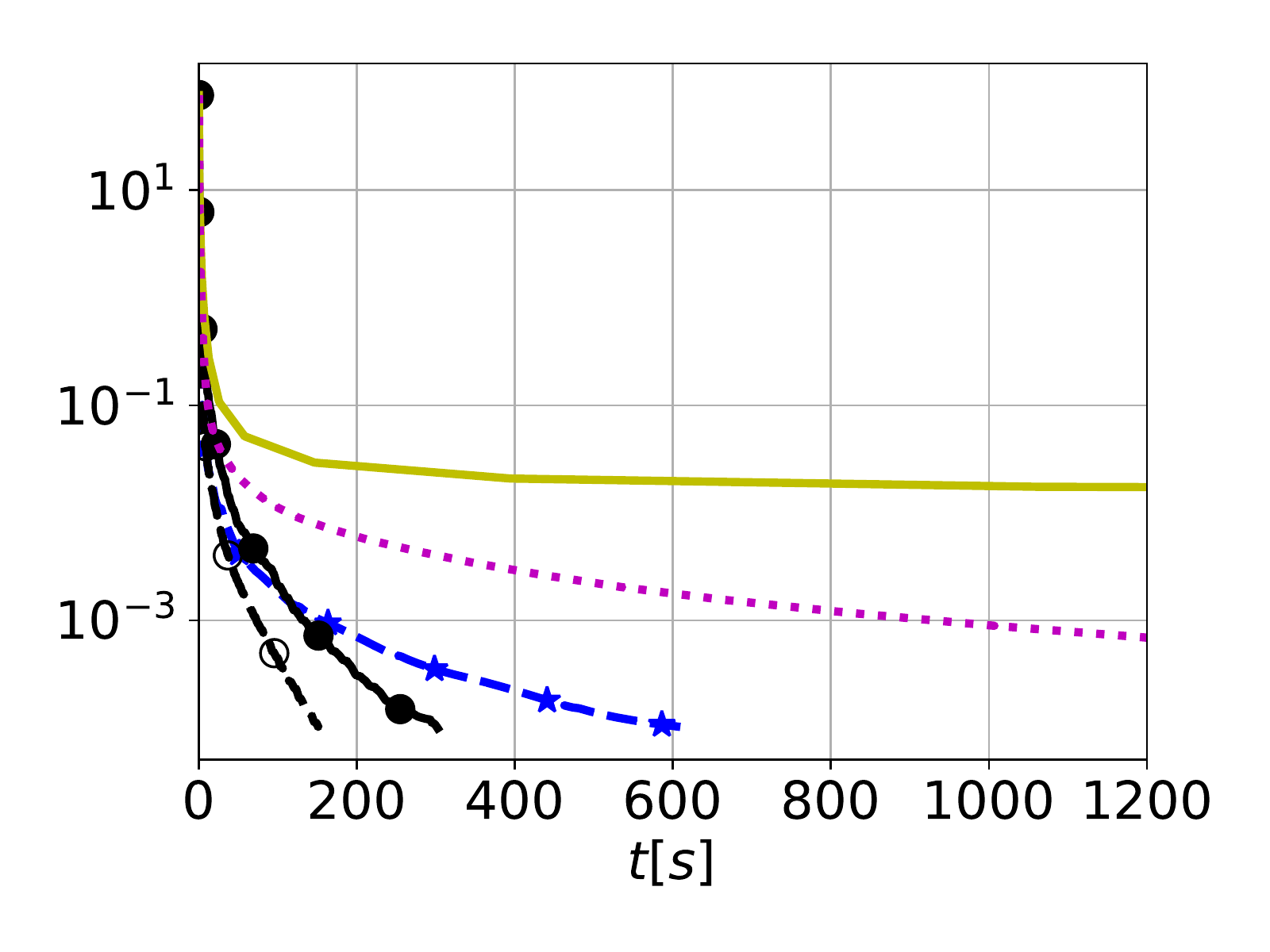} }\label{fig:birkhoff2-time}}%
    \hspace*{\fill}
    \caption{Birkhoff polytope: Algorithm comparison in terms of \protect\subref{fig:birkhoff1-it-cnt},\protect\subref{fig:birkhoff2-it-cnt} iteration count and \protect\subref{fig:birkhoff1-time},\protect\subref{fig:birkhoff2-time} wall-clock time.}%
    \label{fig:Birkhoffall-examples}%
\end{figure*}
\begin{figure*}[h!]
    \centering
    \vspace{-10pt}
    \hspace{\fill}
    \subfloat[Iteration]{{\includegraphics[width=3.95cm]{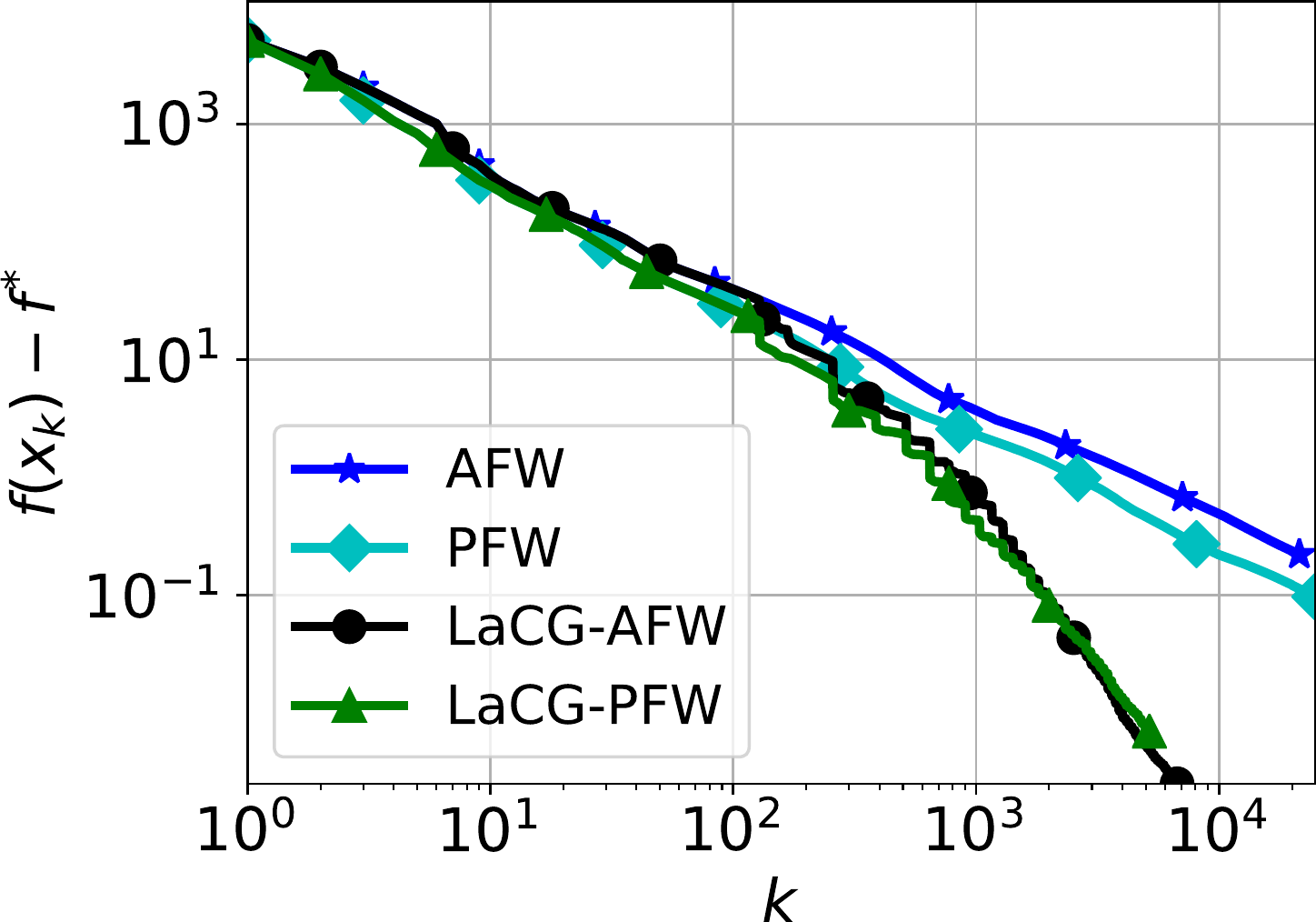} }\label{fig:MIPLIB1-it-cnt}}%
    \hspace{\fill}
    \subfloat[Wall-clock time]{{\includegraphics[width=3.85cm]{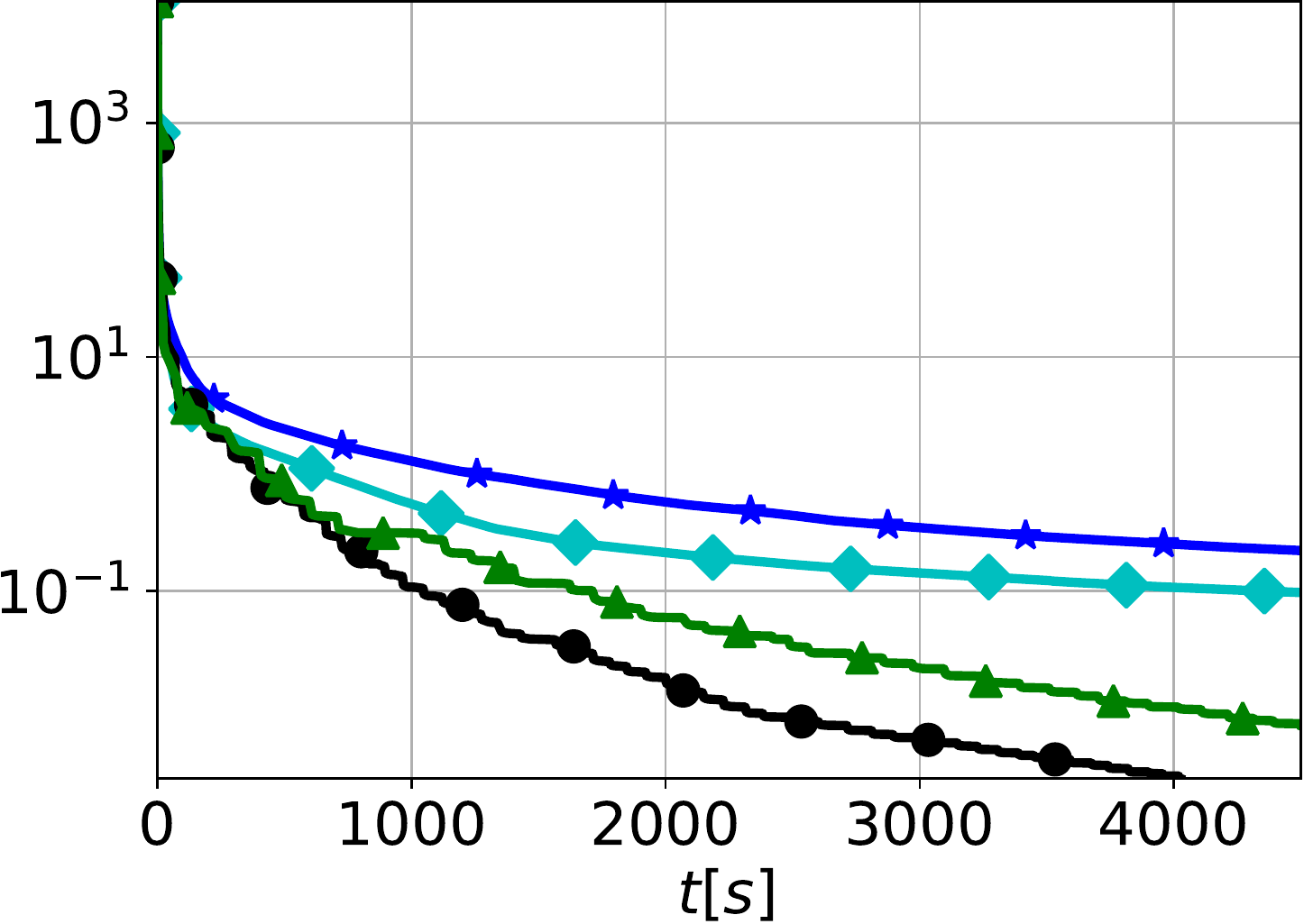} }\label{fig:MIPLIB1-time}}%
    \hspace{\fill}
    \subfloat[Iteration]{{\includegraphics[width=3.9cm]{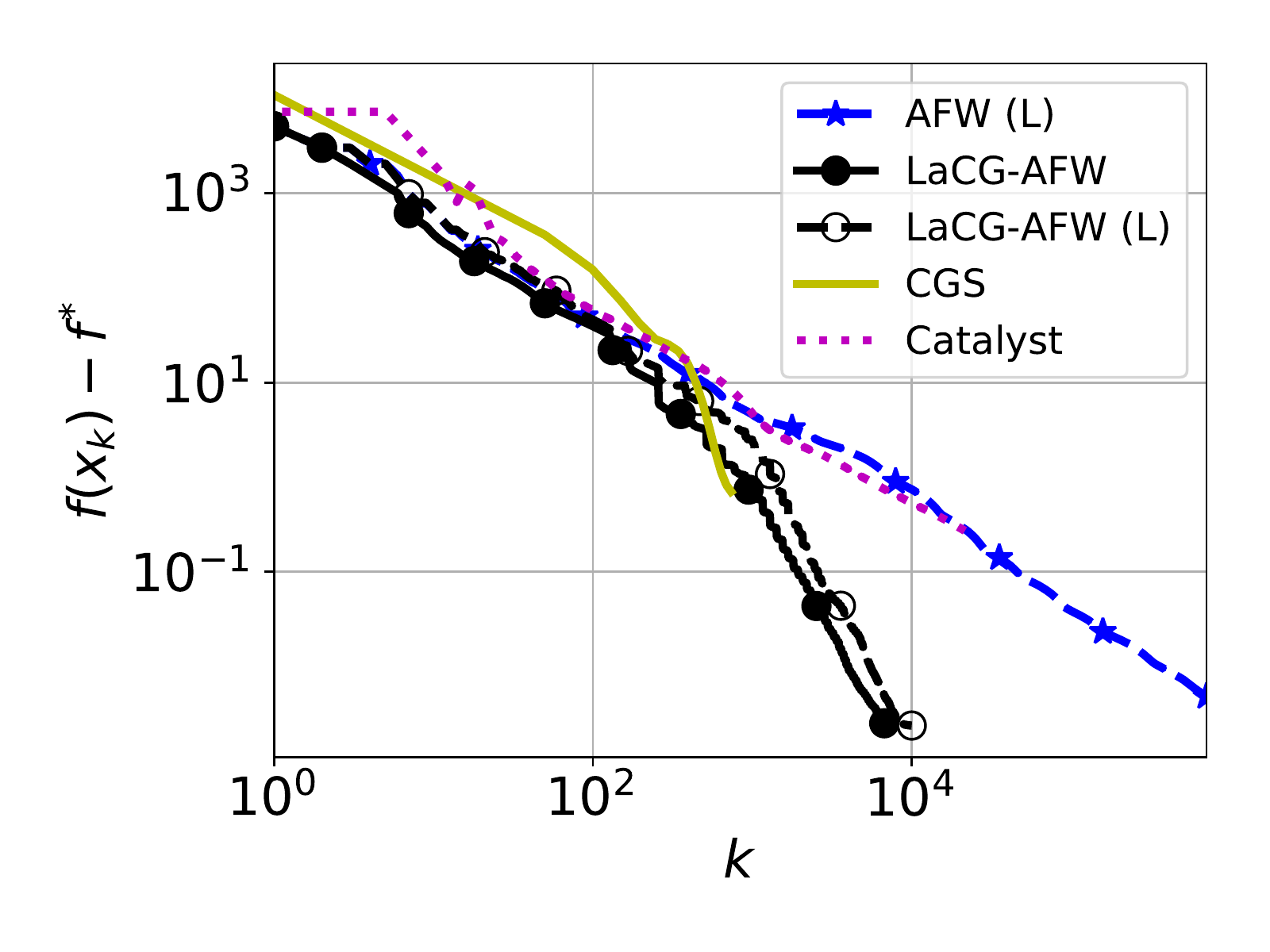} }\label{fig:MIPLIB2-it-cnt}}%
    \hspace{\fill}
    \subfloat[Wall-clock time]{{\includegraphics[width=3.85cm]{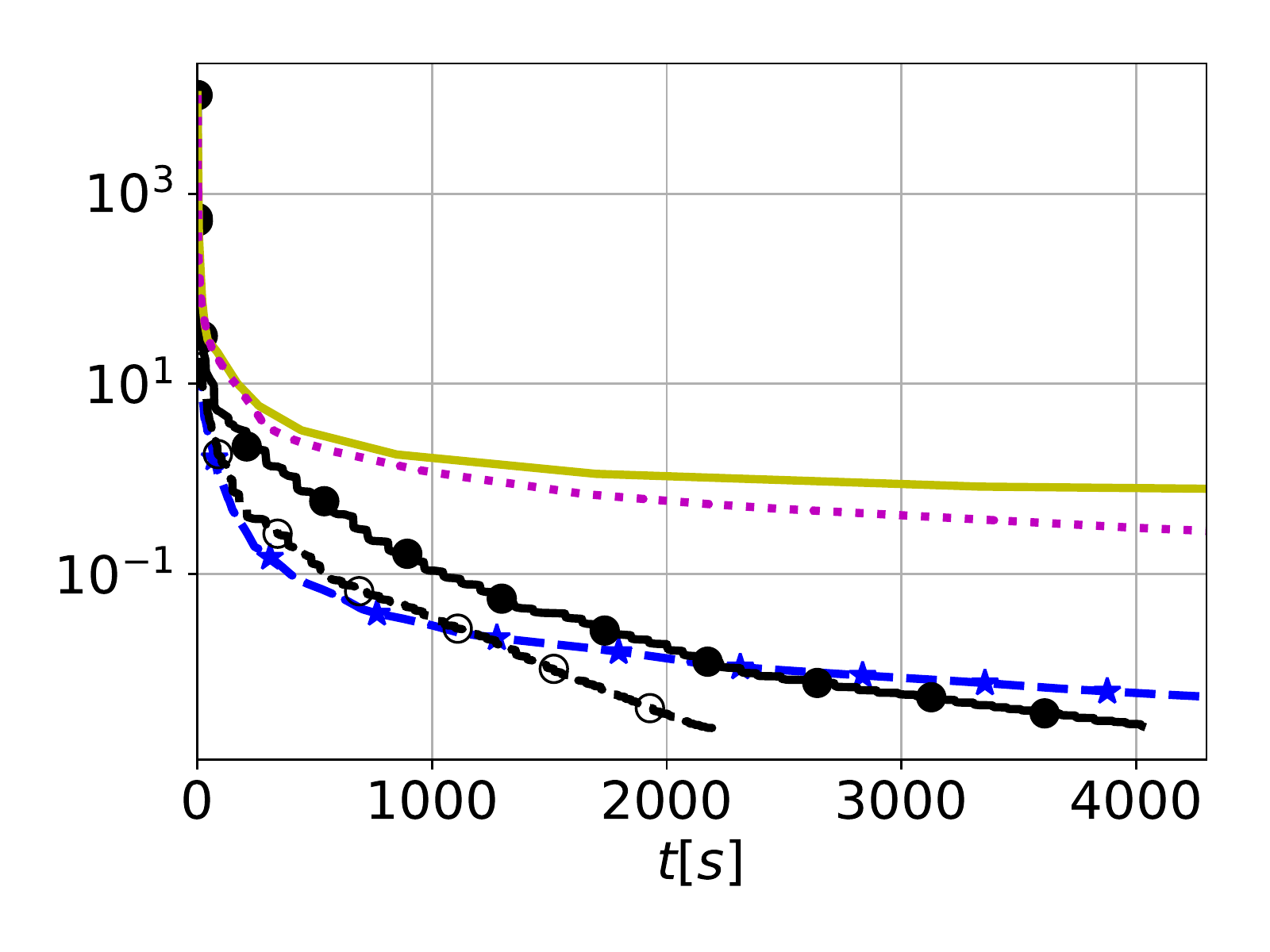} }\label{fig:MIPLIB2-time}}%
    \hspace*{\fill}
    \caption{MIPLIB polytope: Algorithm comparison in terms of \protect\subref{fig:MIPLIB1-it-cnt},\protect\subref{fig:MIPLIB2-it-cnt} iteration count and \protect\subref{fig:MIPLIB1-time},\protect\subref{fig:MIPLIB2-time} wall-clock time for the \texttt{ran14x18-disj-8} polytope from the MIPLIB library.}%
    \label{fig:MIPLIBall-examples}%
\end{figure*}
\begin{figure*}[h!]
    \centering
    \vspace{-10pt}
    \hspace{\fill}
    \subfloat[Iteration]{{\includegraphics[width=3.95cm]{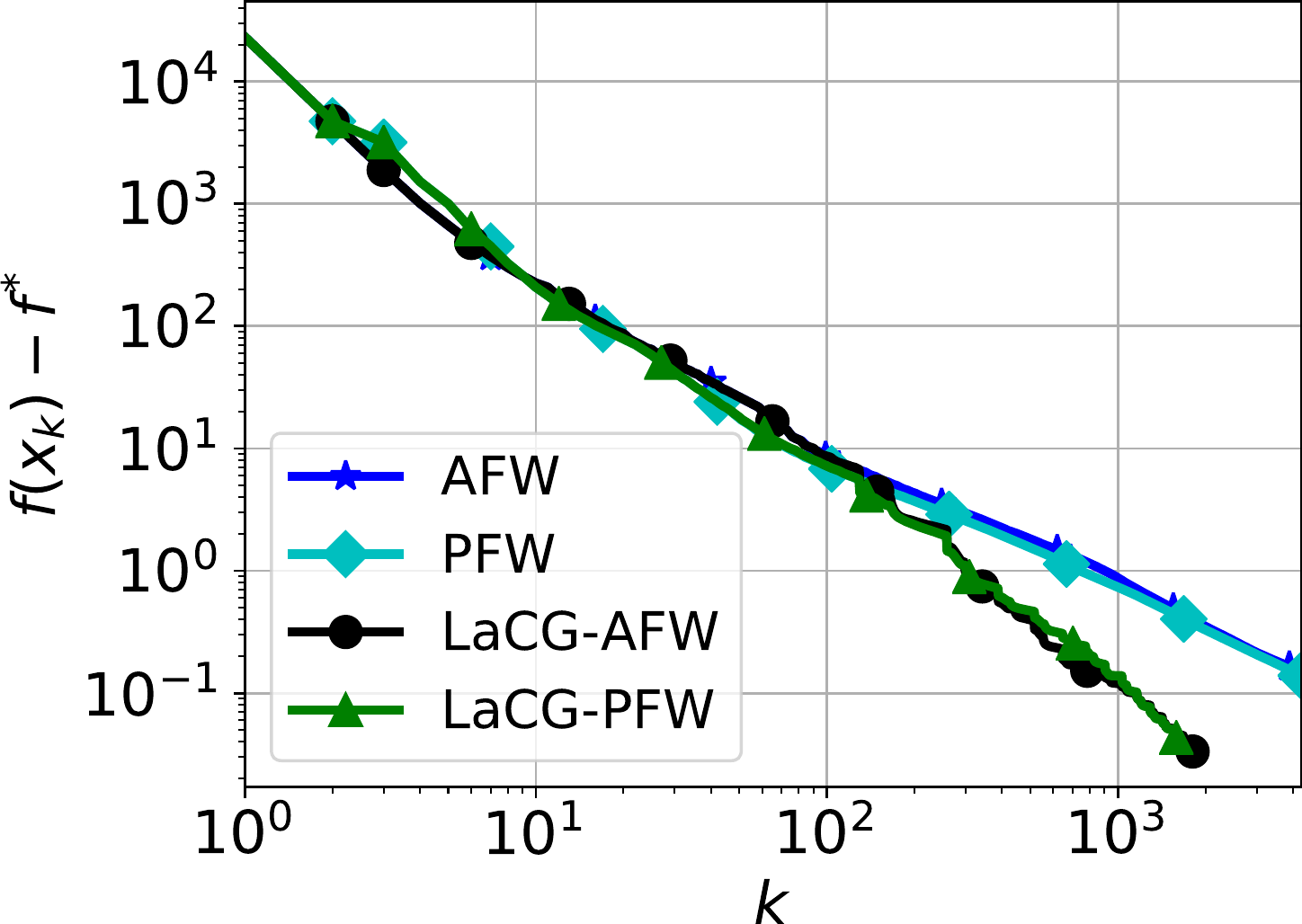} }\label{fig:video1-it-cnt}}%
    \hspace{\fill}
    \subfloat[Wall-clock time]{{\includegraphics[width=3.85cm]{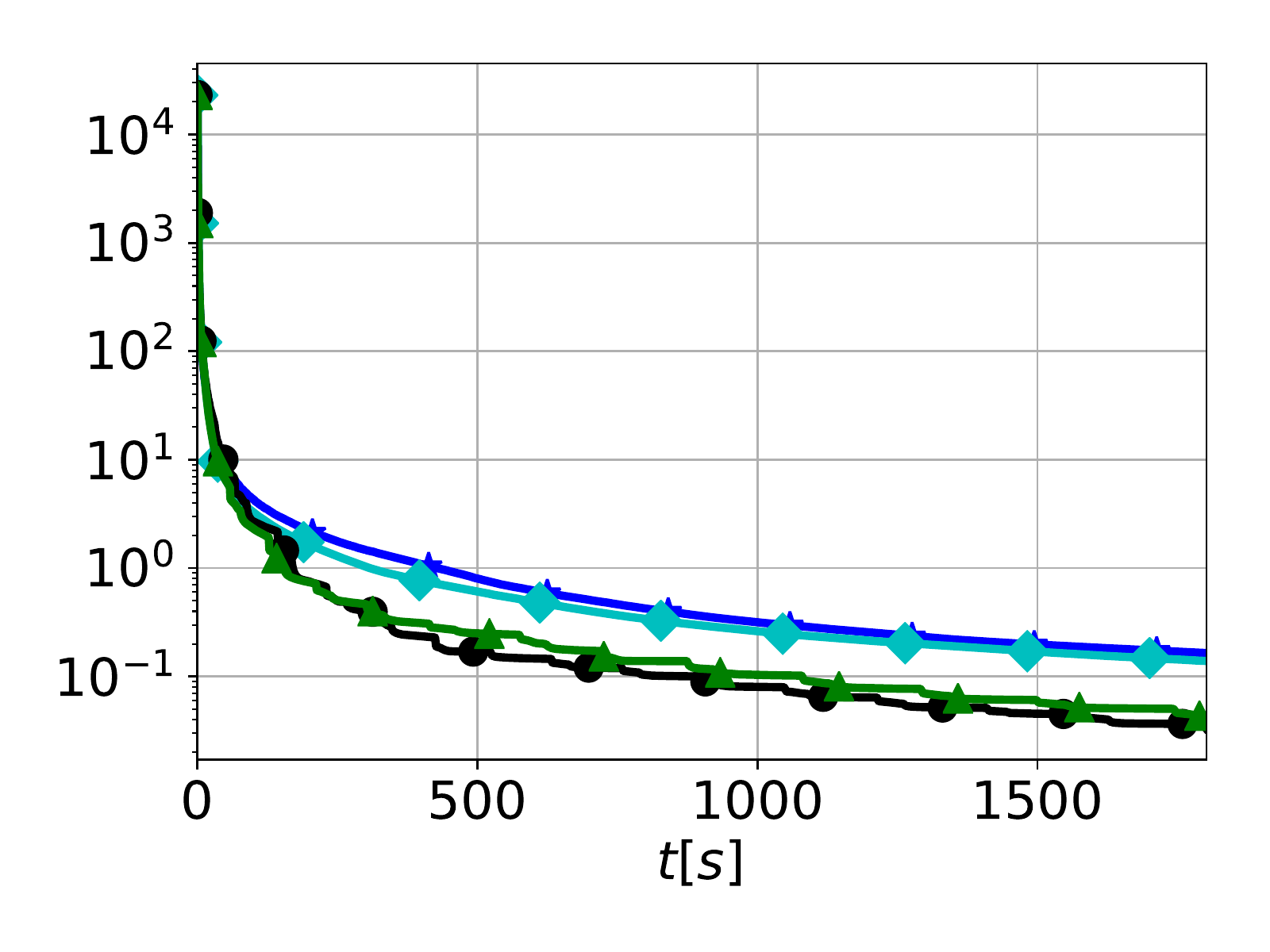} }\label{fig:video1-time}}%
    \hspace{\fill}
    \subfloat[Iteration]{{\includegraphics[width=3.9cm]{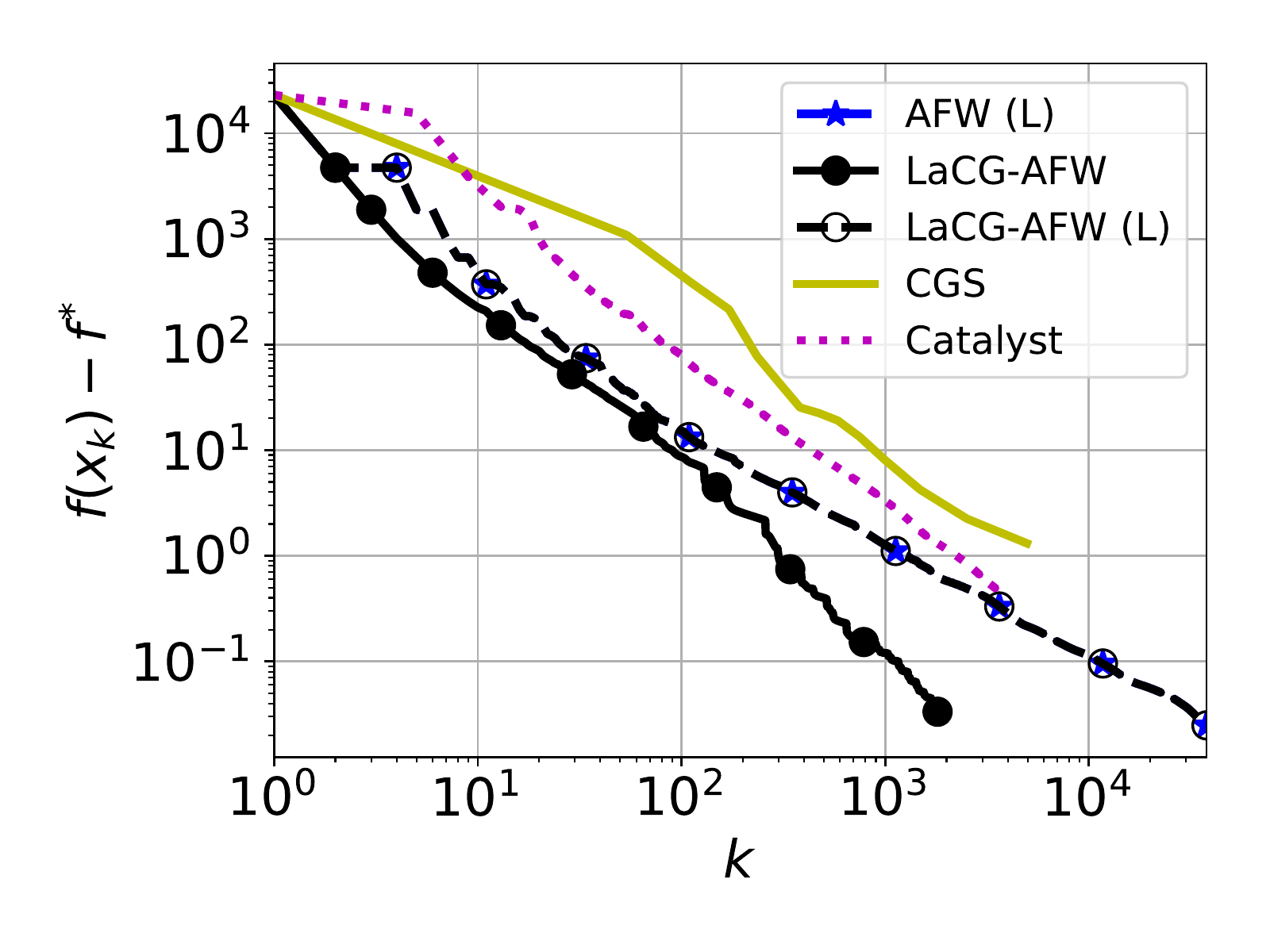} }\label{fig:video2-it-cnt}}%
    \hspace{\fill}
    \subfloat[Wall-clock time]{{\includegraphics[width=3.85cm]{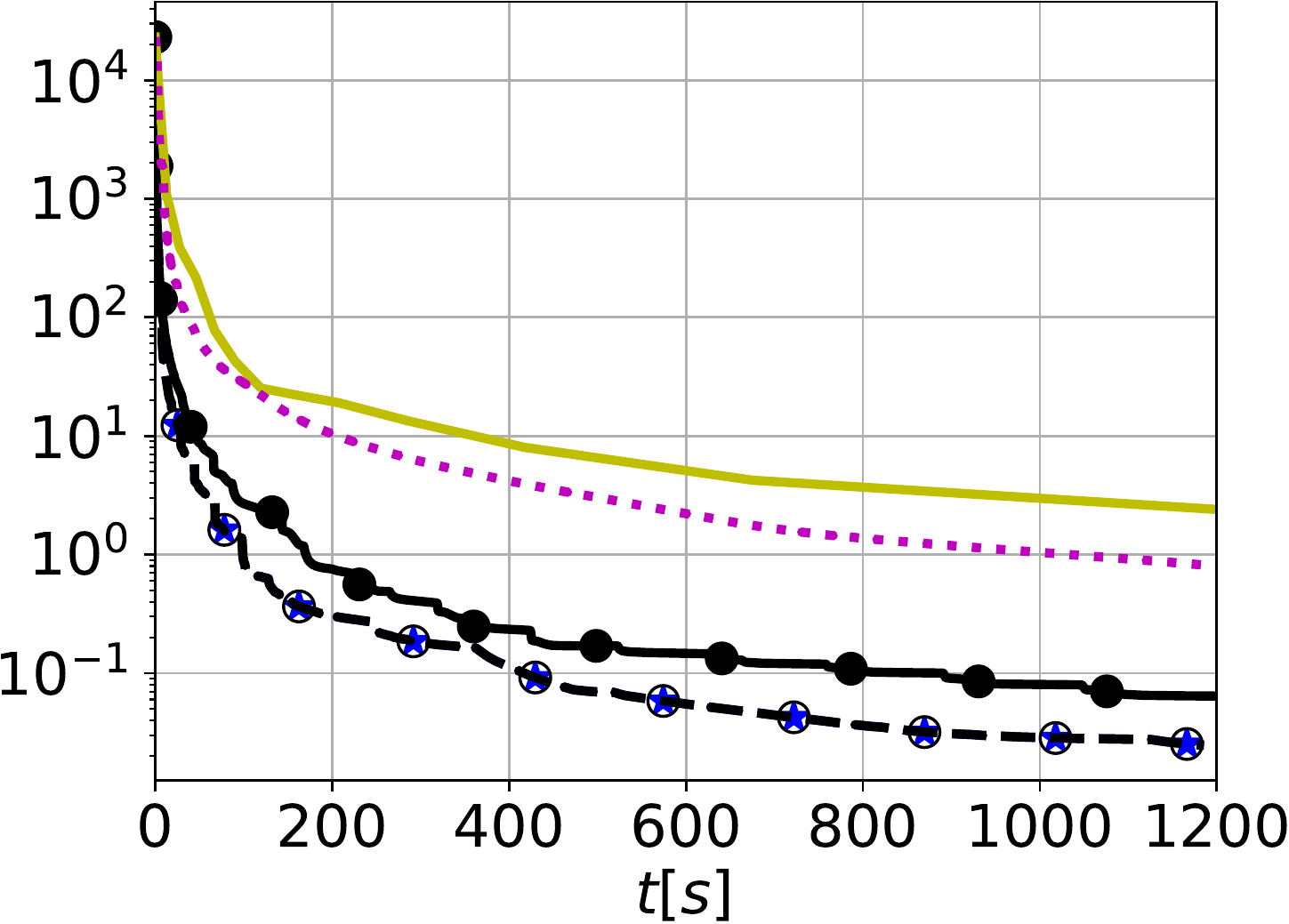} }\label{fig:video2-time}}%
    \hspace*{\fill}
    \caption{Traffic network: Algorithm comparison in terms of \protect\subref{fig:video1-it-cnt},\protect\subref{fig:video2-it-cnt} iteration count and \protect\subref{fig:video1-time},\protect\subref{fig:video2-time} wall-clock time.}%
    \label{fig:video}%
\end{figure*}
\begin{figure*}[h!]
    \centering
    \vspace{-10pt}
    \hspace{\fill}
    \subfloat[Iteration]{{\includegraphics[width=3.95cm]{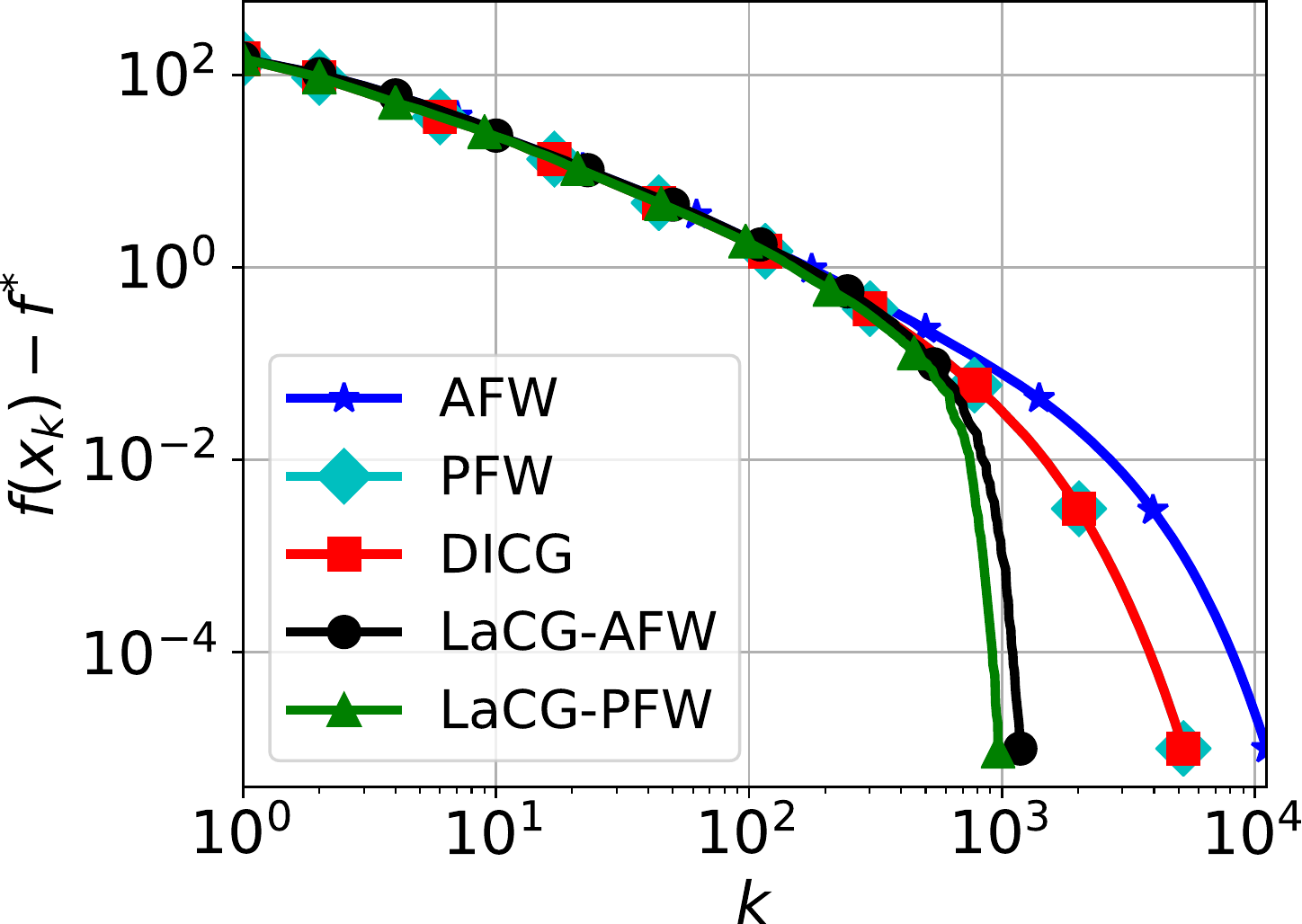} }\label{fig:simplex1-it-cnt}}%
    \hspace{\fill}
    \subfloat[Wall-clock time]{{\includegraphics[width=3.85cm]{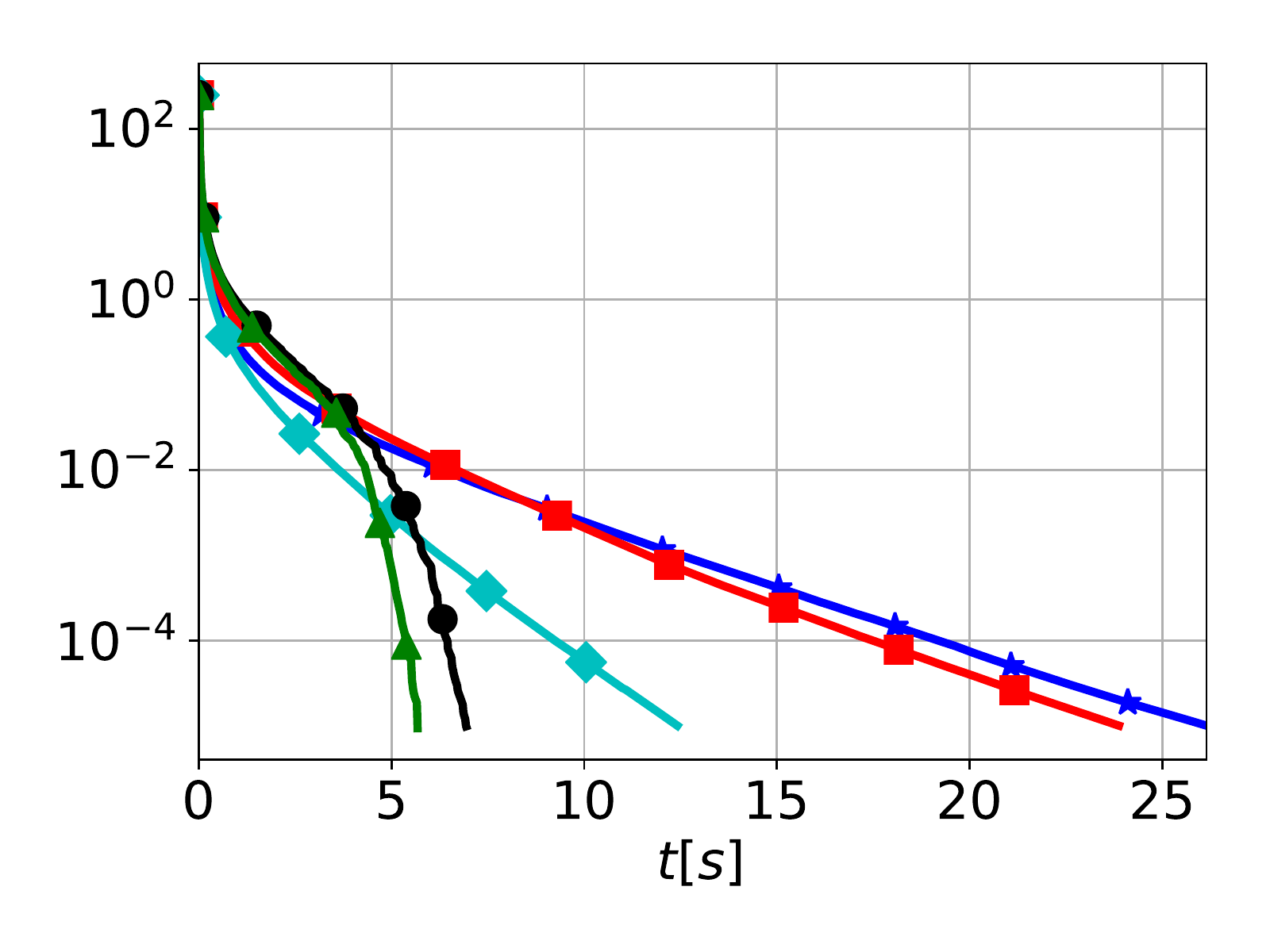} }\label{fig:simplex1-time}}%
    \hspace{\fill}
    \subfloat[Iteration]{{\includegraphics[width=3.9cm]{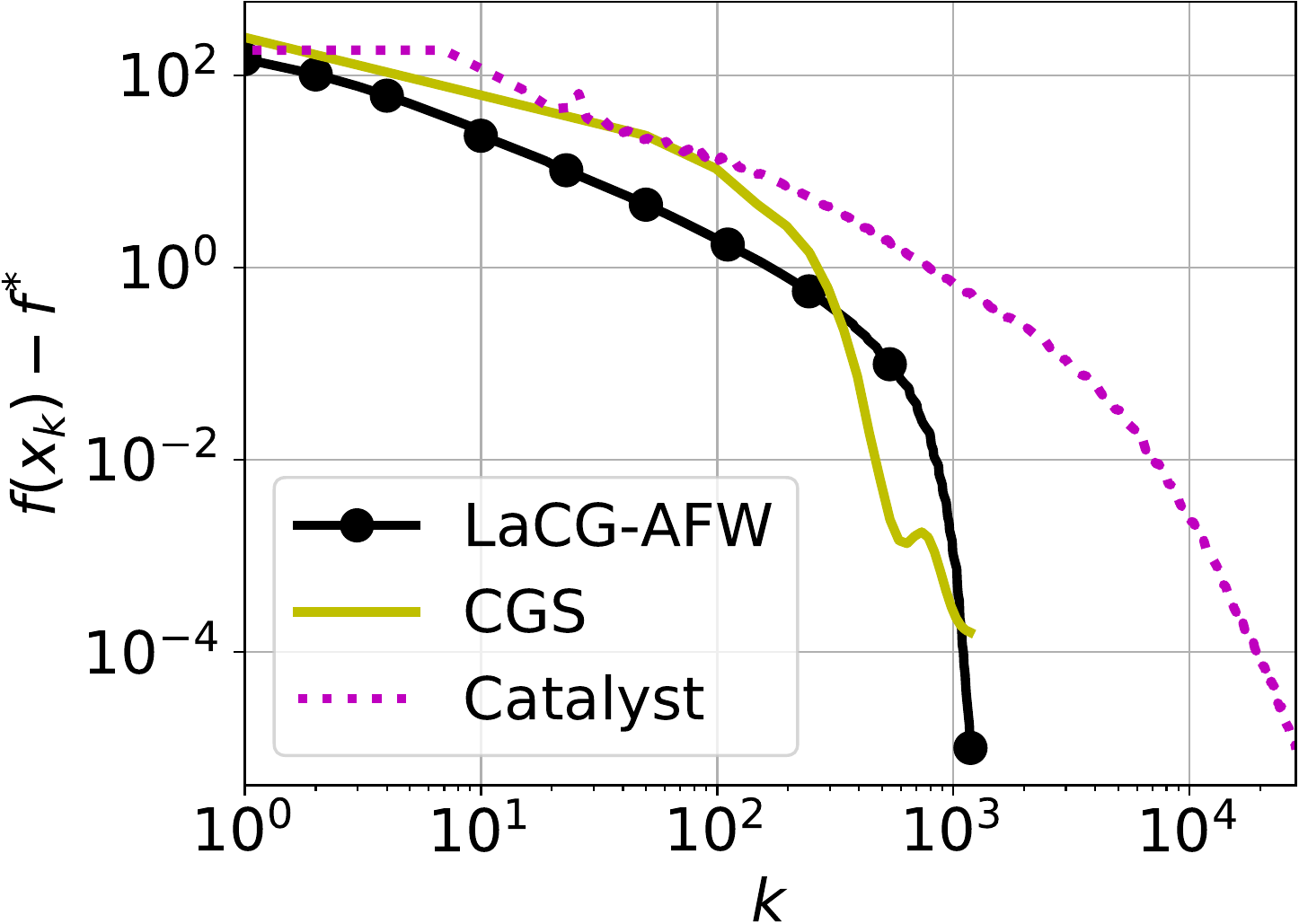} }\label{fig:simplex2-it-cnt}}%
    \hspace{\fill}
    \subfloat[Wall-clock time]{{\includegraphics[width=3.85cm]{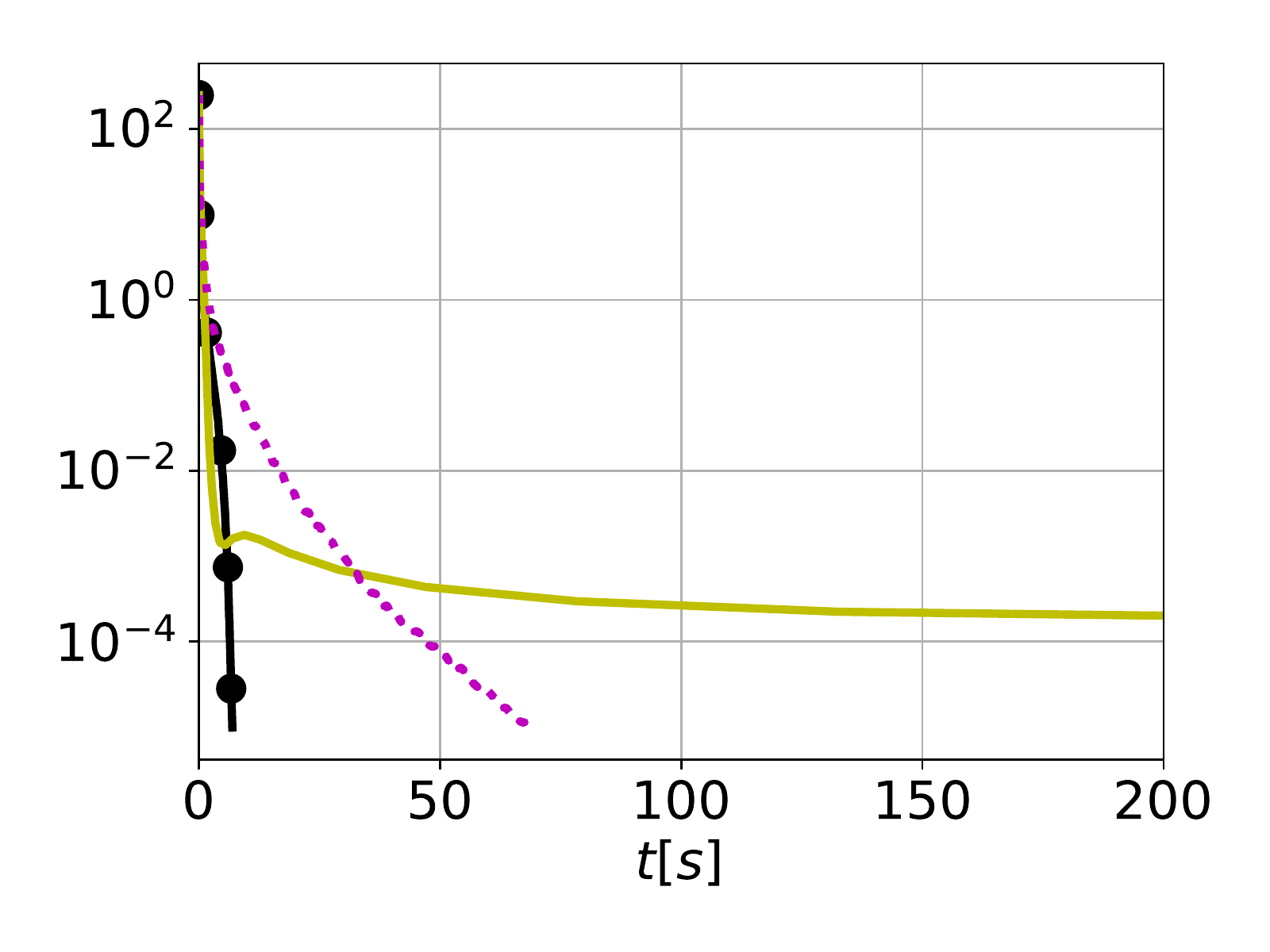} }\label{fig:simplex2-time}}%
    \hspace*{\fill}
    \caption{Simplex: Algorithm comparison in terms of \protect\subref{fig:simplex1-it-cnt},\protect\subref{fig:simplex2-it-cnt} iteration count and \protect\subref{fig:simplex1-time},\protect\subref{fig:simplex2-time} wall-clock time.}%
    \label{fig:simplex}%
\end{figure*}

{\paragraph{Congestion Balancing in Traffic Networks.} We consider the problem of congestion balancing in traffic networks. This problem can be posed as a feasible flow with multiple source-sink pairs and with convex costs on flows over the edges (see, e.g.,~\cite[Chapter 14]{Ahuja:1993:NFT:137406}). We choose the costs to be weighted quadratic, as such costs have the role of balancing the congestion over the network edges (see, e.g.,~\cite{diakonikolas2018width}). The weights are chosen randomly between $\mu$ and $L$, with $L/\mu = 100$, leading to a separable quadratic objective with the condition number 100. The problem instance is the \texttt{road\_paths\_01\_DC\_a} flow polytope (same as in \cite{lan2017conditional} and \cite{braun2018blended}, with dimension $n = 29682$). As the flow polytope in this case is not a 0/1 polytope, DiCG~\cite{LDLCC2016} does not apply to this problem instance.}

\paragraph{Probability Simplex.}
The last example, shown in Fig.~\ref{fig:simplex}\subref{fig:simplex1-it-cnt}-\ref{fig:simplex}\subref{fig:simplex2-time}, is the probability simplex, which, while a toy example, lends itself nicely for comparisons between methods, as expected behavior here is well-understood. Due to the structure of this polytope, there is no need to explicitly maintain an active set in the AFW, PFW or the LaCG algorithm. This greatly speeds up all of the algorithms, and eliminates the main advantage of the DiCG algorithm (i.e., that it does not need to maintain an active set). There are also further algorithmic enhancements that can be made to the LaCG algorithm in this case, which we detail in  Appendix~\ref{appx:simplexEnhancements}. We generate the objective function in the same way as in the Structured Regression, with $n = 1500$, $L/\mu = 1000$ and with $\vb$ having random integer values of $-1, 0,$ or $1$.

\paragraph{Lazification.}

Our proposed approach is also compatible with the lazification technique from \cite{BPZ2017}. To demonstrate the advantage of lazification in terms of wall-clock time, we run it as an alternative version of LaCG when comparing against CGS and Catalyst for each example. As a reference and for comparison, we also include the lazified AFW algorithm.



\section{Discussion}

We presented the Locally-Accelerated Conditional Gradients framework that achieves asymptotically optimal rate in a local region around the minimum and improves upon the existing conditional gradients methods, both in theory and in experiments. As discussed before, such an accelerated rate cannot be achieved globally. 

Some interesting questions for future research remain. For example, can similar guarantees to ours be obtained for smooth (non-strongly) convex minimization? Further, it would be interesting to understand whether the version of $\mu$AGD+ from Lemma~\ref{lemma:modified-agdp}, which allows changing the projection set $\cc_k$, can speed up the practical performance of accelerated methods in other (possibly projection-based) optimization settings.

\section*{Acknowledgements}

Research reported in this paper was partially supported by the NSF grant CCF-1740855, NSF CAREER Award CMMI-1452463, and Award W911NF-18-1-0223 from the Army Research Office. Part of it was done while JD and SP were visiting Simons Institute for the Theory of Computing.

\newpage
\balance
\bibliographystyle{abbrv}
\bibliography{references}

\newpage

\onecolumn

\appendix

\section{Lower Bound in ADGT}\label{appx:adgt}

In this section, we provide the construction of the lower bounds on the minimum function value $f(\vx^*)$ that are used in our analysis. By $\mu$-strong convexity of $f,$ we have that, $\forall \vx \in \cx:$
\begin{equation}\label{eq:strong-convexity-1}
    f(\vx^*) \geq f(\vx) + \innp{\nabla f(\vx), \vx^* - \vx} + \frac{\mu}{2}\|\vx - \vx^*\|^2.
\end{equation}
Further, if $\vx^*$ belongs to the interior of $\cx,$ then $\nabla f(\vx^*) = 0,$ and $L$-smoothness of $f$ implies, $\forall \vx \in \cx:$
\begin{equation}\label{eq:strong-convexity-2}
    f(\vx^*) \geq f(\vx) - \frac{L}{2}\|\vx - \vx^*\|^2.
\end{equation}
Let $\{\vx_i\}_{i=0}^k$ be a sequence of points from some feasible set $\cx$ and let $\{a_i\}_{i=0}^k$ be a sequence of positive numbers with $a_0 = 1$. Define $A_k \defeq \sum_{i=0}^k a_i.$

Assume first that $\vx^*$ belongs to the interior of the feasible set $\cx$. Then, taking a convex combination of Eq.~\eqref{eq:strong-convexity-2} with $\vx = \vx_0$ and Eq.~\eqref{eq:strong-convexity-1} with $\vx = \vx_i$, $1\leq i \leq k,$ we get:
\begin{align*}
    f(\vx^*) \geq &\frac{\sum_{i=0}^k a_i f(\vx_i) + \sum_{i=1}^k a_i(\innp{\nabla f(\vx_i), \vx^* - \vx_i} + \frac{\mu}{2}\|\vx_i - \vx^*\|^2)-\frac{L}{2}\|\vx_0 - \vx^*\|^2}{A_k}\\
    &+ \frac{\mu}{2 A_k}\|\vx_0 - \vx^*\|^2 - \frac{\mu}{2 A_k}\|\vx_0 - \vx^*\|^2\\
    \geq & \frac{\sum_{i=0}^k a_i f(\vx_i) + \min_{\vu \in \rr^d}\{\sum_{i=1}^k a_i(\innp{\nabla f(\vx_i), \vu - \vx_i} + \frac{\mu}{2}\|\vx_i - \vu\|^2) + \frac{\mu}{2}\|\vx_0 - \vu\|^2\}}{A_k}\\
    &- \frac{L+\mu}{2 A_k}\|\vx_0 - \vx^*\|^2.
\end{align*}
The last expression corresponds to the lower bound used in the proof of Lemma~\ref{lemma:acc-conv}.

Now assume that $\vx^*$ is not necessarily from the interior of $\cx.$ Take a convex combination (with weights $a_i/A_k$) of Eq.~\eqref{eq:strong-convexity-1} for $\vx = \vx_i,$ $0 \leq i \leq k.$ Let $\cc_k$ be any convex subset of $\cx$ that contains $\vx^*.$ Then, we have:
\begin{align*}
    f(\vx^*) \geq &\frac{\sum_{i=0}^k a_i f(\vx_i) + \sum_{i=0}^k a_i(\innp{\nabla f(\vx_i), \vx^* - \vx_i} + \frac{\mu}{2}\|\vx_i - \vx^*\|^2)}{A_k}\\
    &+ \frac{\mu_0}{2A_k}\|\vx_0 - \vx^*\|^2 - \frac{\mu_0}{2A_k}\|\vx_0 - \vx^*\|^2\\
    \geq & \frac{\sum_{i=0}^k a_i f(\vx_i) + \min_{\vu \in \cc_k}\{\sum_{i=0}^k a_i(\innp{\nabla f(\vx_i), \vu - \vx_i} + \frac{\mu}{2}\|\vx_i - \vu\|^2) + \frac{\mu_0}{2}\|\vx_0 - \vu\|^2\}}{A_k}\\
    &- \frac{\mu_0}{2A_k}\|\vx_0 - \vx^*\|^2.
\end{align*}
The last expression corresponds to the lower bound used in the proof of Lemma~\ref{lemma:modified-agdp}.




%
%
\section{Omitted Proofs from Section~\ref{sec:accg}}\label{appx:omitted-proofs}
%
%
\subsection{Proofs and Results for Warm-up: Optimum in the Interior of $\cx$}\label{appx:prelim-algo}
Starting at point $\vx_k,$ the Frank-Wolfe step $\vx^{\mathrm{FW}}_{k+1}$ is defined via:
\begin{equation}\label{eq:FW-step}
    \begin{gathered}
    \vv_k = \argmin_{\vu \in \cx} \innp{\nabla f(\vx_k), \vu},\\
    \vx^{\mathrm{FW}}_{k+1} = (1-\eta_k)\vx_k + \eta_k \vv_k,
    \end{gathered}
\end{equation}
where 
$$\eta_k = \argmin_{\eta \in [0, 1]} \Big\{f(\vx_k) + \innp{\nabla f(\vx_k), \eta_k (\vv_k - \vx_k)} + \frac{L}{2}{\eta_k}^2\|\vx_k - \vv_k\|^2\Big\}.$$

On the other hand, the accelerated step $\vxh_{k+1}$ is defined as:
\begin{equation}\label{eq:agd-step}
    \begin{gathered}
    \vy_{k+1} = \frac{1}{1+\theta}\vx_{k} + \frac{\theta}{1 + \theta} \vw_k,\\ 
    \vw_{k+1} = (1-\theta)\vw_k + \theta \Big(\vy_{k+1} - \frac{1}{\mu}\nabla f(\vy_{k+1})\Big),\\
    \vxh_{k+1} = (1-\theta)\vx_k + \theta \vw_{k+1},
    \end{gathered}
\end{equation}
where $\theta  = \sqrt{\frac{\mu}{L}}$ and $\vw_k$ and $\vx_k$ are appropriately initialized. 
We now proceed to describe the algorithm.
\begin{algorithm}
\caption{Preliminary Locally Accelerated Frank-Wolfe for $\vx^* \in \interior(\cx)$}\label{algo:acc-FW}
\begin{algorithmic}[1]
\Statex Input: $\vx_0 \in \cx$, $\mu,$ $L,$ $\cx$
\Statex Initialization: $\vw_{0}  = \vx_0,$ $\theta = \sqrt{\mu/L}$
\For{$k = 0$ to $N-1$}
\State Compute $\vx^{\mathrm{FW}}_{k+1}$ based on Eq.~\eqref{eq:FW-step} and $\vxh_{k+1}$ based on Eq.~\eqref{eq:agd-step}
\If{$\vxh_{k+1} \in \cx$}
\State $\vx_{k+1} = \argmin\{f(\vx^{\mathrm{FW}}_{k+1}),\, f(\vxh_{k+1})\}$
\Else
\State $\vx_{k+1} = \vx^{\mathrm{FW}}_{k+1}$
\State $\vw_{k+1} = \vx_{k+1}$
\EndIf
\EndFor
\end{algorithmic}
\end{algorithm}

Note that the ``else'' branch in Algorithm~\ref{algo:acc-FW} effectively restarts the accelerated sequence.

Let us now argue about the convergence of the algorithm. Observe first
that the algorithm makes at least as much progress as Frank-Wolfe,
since, whatever the step is,
$f(\vx_{k+1}) \leq f(\vx^{\mathrm{FW}}_{k+1}).$ We thus have the
following simple proposition, which bounds the length of the so-called
\emph{burn-in phase}.

\begin{proposition}\label{prop:initial-steps}
Assume that $r > 0.$ Then, after at most $K_0 = \lfloor \frac{L D^2}{\mu r^2}\rfloor$ steps of Algorithm~\ref{algo:acc-FW}, $f(\vx_{K_0}) - f(\vx^*) \leq 2\mu r^2.$ Further, in every subsequent iteration $k \geq K_0$, $\|\vx_k - \vx^*\| \leq 2r.$
\end{proposition}
\begin{proof}
  As in every iteration the algorithm makes at least as much progress
  as standard Frank-Wolfe (since
  $f(\vx_{k+1}) \leq f(\vx^{\mathrm{FW}}_{k+1})$), by standard
  Frank-Wolfe guarantees (see e.g., \cite{jaggi2013revisiting}), we
  have that after $K_0$ steps
  $f(\vx_{K_0}) - f(\vx^*) \leq \frac{2LD^2}{K_0 + 4},$ which gives the first part of the lemma. Since none of the
  iterations of the algorithm can increase the function value, we have
  that in every subsequent iteration
  $f(\vx_k) - f(\vx^*) \leq 2\mu r^2.$ By strong convexity and
  $\nabla f(\vx^*) = 0$, this implies $\|\vx_k - \vx^*\| \leq 2r.$
\end{proof} 
We can conclude that if $r>0,$ for $k > K_0 = \lfloor \frac{L D^2}{\mu r^2}\rfloor$ Algorithm~\ref{algo:acc-FW} never enters the else branch, as ${\cal B}(\vx^*,2r)\cap \mathrm{\Aff}(\cx)\subseteq\cx$. This is precisely what allows us to obtain accelerated convergence in the remaining iterations. This is formally established by the following lemma.

\begin{restatable}{lemma}{lemconvint}
  \label{lemma:acc-conv}
Assume that $r > 0$ and let $K_0 = \lfloor \frac{L D^2}{\mu r^2}\rfloor$. Then, for all $k \geq K_0:$
$$
f(\vx_k) - f(\vx^*) \leq 2\frac{L+\mu}{\mu} r^2\left(1 - \sqrt{\frac{\mu}{L}}\right)^{k- K_0}. 
$$
\end{restatable}
\begin{proof}
Let $k_0 \leq K_0$ be the last iteration in which Algorithm~\ref{algo:acc-FW} enters the ``else'' branch -- as already argued, for $k >K_0,$ this cannot happen. Then, from Algorithm~\ref{algo:acc-FW}, we have that $\vw_{k_0} = \vx_{k_0},$ and for all iterations $k \geq k_0 + 1:$
\begin{equation}\label{eq:eff-afw}
    \begin{gathered}
    \vy_{k} = \frac{1}{1+\theta}\vx_{k-1} + \frac{\theta}{1+\theta} \vw_{k-1},\\
    \vw_k = (1-\theta)\vw_{k-1} + \theta \Big(\vy_k - \frac{1}{\mu}\nabla f(\vy_k)\Big),\\
    \vxh_k = (1-\theta)\vx_{k-1} + \theta \vw_k,\\
    \vx_k = \argmin \{f(\vxh_k),\; f(\vx^{\mathrm{FW}}_k) \}.
    \end{gathered}
\end{equation}
To analyze the convergence of~\eqref{eq:eff-afw}, we use the
approximate duality gap technique, as described in Section~\ref{sec:ADGT-description}. Let $a_{k_0} = A_{k_0} = 1$ and
$A_{k} = \sum_{i=k_0}^k a_i,$ $\frac{a_k}{A_k} = \theta$ for
$k \geq k_0 + 1.$ Recall that the approximate duality gap $G_k$ is defined as the
difference between a lower bound on $f(\vx^*),$ $L_k$ and an upper
bound on the algorithm output, $U_k.$ Define $U_k = f(\vx_k)$ and
$L_k$ via (see Appendix~\ref{appx:adgt}):
\begin{equation}
\begin{aligned}
    L_k \defeq &\frac{\sum_{i={k_0}}^k a_i f(\vy_i) + \min_{\vu \in \rr^d}\{\sum_{i={k_0+1}}^k a_i(\innp{\nabla f(\vy_i), \vu - \vy_i} + \frac{\mu}{2}\|\vu - \vy_i\|^2) + \frac{\mu}{2}\|\vu - \vx_{k_0}\|^2\}}{A_k}\\
    &- \frac{L + \mu}{2 A_k}\|\vx^* - \vx_{k_0}\|^2.
\end{aligned}
\end{equation}
We claim that:
\begin{equation}\label{eq:equiv-def-w}
\begin{aligned}
\vw_k &= \argmin_{\vu \in \rr^d}\Big\{\sum_{i={k_0+1}}^k a_i(\innp{\nabla f(\vy_i), \vu - \vy_i} + \frac{\mu}{2}\|\vu - \vy_i\|^2) + \frac{\mu}{2}\|\vu - \vx_{k_0}\|^2\Big\}\\
&= \frac{\vx_{k_0} + \sum_{i=k_0+1}^k a_i(\vy_i - \frac{1}{\mu}\nabla f(\vy_i))}{A_k}.
\end{aligned}
\end{equation}
Indeed, Eq.~\eqref{eq:equiv-def-w} implies that $\vw_{k_0} = \vx_{k_0},$ while for $k > k_0$ it gives: $A_k\vw_k = A_{k-1}\vw_{k-1} + a_k (\vy_k - \frac{1}{\mu}\nabla f(\vy_k)).$ As $A_k = A_{k-1} + a_k$ and $\frac{a_k}{A_k} = \theta,$ \eqref{eq:equiv-def-w} implies that $\vw_k = (1-\theta) \vw_{k-1} + \theta (\vy_k - \frac{1}{\mu}\nabla f(\vy_k)),$ which is equivalent to the definition from Eq.~\eqref{eq:eff-afw}. 

Further, observe from~\eqref{eq:eff-afw} that $\vx_{k-1} = (1+\theta)\vy_k - \theta \vw_{k-1},$ which, combined with $\vw_k = (1-\theta)\vw_{k-1} + \theta (\vy_k - \frac{1}{\mu}\nabla f(\vy_k))$ and $\theta = \sqrt{\mu/L},$ implies
\begin{equation}\label{eq:grad-progress-y}
    \vxh_k = \vy_k - \frac{1}{L}\nabla f(\vy_k).
\end{equation}

The rest of the proof bounds the initial gap $G_{k_0}$ and shows that for $k > k_0,$ $G_k \leq (1 - \theta)G_{k-1}.$ Note that, by construction, $f(\vx_k) - f(\vx^*) \leq G_k.$

The initial gap equals $G_{k_0} = \frac{L + \mu}{2}\|\vx^* - \vx_{k_0}\|^2$. This follows by simply evaluating $U_{k_0} - L_{k_0}.$  

Now let $k>k_0.$ As $f(\vx_k) \leq f(\vxh_k)$ and using~\eqref{eq:grad-progress-y}:
\begin{align}
    A_k U_k - A_{k-1}U_{k-1} &\leq A_k f(\vxh_k) - A_{k-1}f(\vx_{k-1})\notag\\
    &= a_k f(\vy_k) + A_k(f(\vxh_k) - f(\vy_k)) + A_{k-1}(f(\vy_k) - f(\vx_{k-1}))\notag\\
    &\leq a_k f(\vy_k) - \frac{A_k}{2L}\|\nabla f(\vy_k)\|^2 + A_{k-1}(f(\vy_k) - f(\vx_{k-1})).\label{eq:ub-change}
\end{align}
To bound the change in the lower bound, denote by:
$$
m_k(\vu) = \sum_{i={k_0+1}}^k a_i(\innp{\nabla f(\vy_i), \vu - \vy_i} + \frac{\mu}{2}\|\vu - \vy_i\|^2) + \frac{\mu}{2}\|\vu - \vx_{k_0}\|^2
$$
the function inside the minimum in the definition of $L_k.$ Hence:
$$
m_k(\vw_k) = m_{k-1}(\vw_k) + a_k \innp{\nabla f(\vx_k), \vw_k - \vy_k} + a_k \frac{\mu}{2}\|\vw_k - \vy_k\|^2.
$$
As $\vw_{k-1}$ minimizes $m_{k-1}(\cdot),$ expanding $m_{k-1}(\vw_k)$ around $\vw_{k-1}$, we have:
$$
m_{k-1}(\vw_k) = m_{k-1}(\vw_{k-1}) + \innp{\nabla m_{k-1}(\vw_{k-1}), \vw_k - \vw_{k-1}} + \frac{A_{k-1} \mu }{2}\|\vw_k - \vw_{k-1}\|^2,
$$
leading to:
\begin{align*}
 m_k(\vw_k) - m_{k-1}(\vw_{k-1}) &= a_k \innp{\nabla f(\vy_k), \vw_k - \vy_k} + a_k \frac{\mu}{2}\|\vw_k - \vy_k\|^2 + \frac{A_{k-1}\mu}{2}\|\vw_k - \vw_{k-1}\|^2   \\
 &\geq a_k \innp{\nabla f(\vy_k), \vw_k - \vy_k} + \frac{A_k \mu}{2}\Big\|\vw_k - \frac{A_{k-1}}{A_k}\vw_{k-1} - \frac{a_k}{A_k}\vy_k\Big\|^2,
\end{align*}
where the second line is by Jensen's inequality. As $\frac{a_k}{A_k} = \theta = \sqrt{\mu/L},$ using the definition of $\vw_k,$ we have:
$$
m_k(\vw_k) - m_{k-1}(\vw_{k-1}) \geq a_k \innp{\nabla f(\vy_k), \vw_k - \vy_k} + \frac{A_k}{2L}\|\nabla f(\vy_k)\|^2.
$$
Combining with the definition of $L_k,$ we thus have:
\begin{equation}\label{eq:lb-change}
    A_k L_k - A_{k-1}L_{k-1} \geq a_k f(\vy_k) + a_k \innp{\nabla f(\vy_k), \vw_k - \vy_k} + \frac{A_k}{2L}\|\nabla f(\vy_k)\|^2.
\end{equation}
Combining \eqref{eq:ub-change} and \eqref{eq:lb-change}, we have:
\begin{align*}
    A_k G_k - A_{k-1}G_{k-1} &\leq A_{k-1}(f(\vy_k) - f(\vx_{k-1})) - a_k\innp{\nabla f(\vy_k), \vw_k - \vy_k} - \frac{A_k}{L}\|\nabla f(\vy_k)\|^2\\
    &\leq \innp{\nabla f(\vy_k), A_k\vy_k -A_{k-1}\vx_{k-1} - a_k \vw_k} - \frac{A_k}{L}\|\nabla f(\vy_k)\|^2\\
    &= A_k \innp{\nabla f(\vy_k), \vy_k - \vxh_k} - \frac{A_k}{L}\|\nabla f(\vy_k)\|^2\\
    &= 0,
\end{align*}
where the second line is by convexity of $f$ (namely, by $f(\vy_k) - f(\vx_k)\leq \innp{\nabla f(\vy_k), \vy_k - \vx_k}$), the third line is by the definition of $\vxh_k$ and $\theta = \frac{a_k}{A_k},$ and the last line is by~\eqref{eq:grad-progress-y}. 

As $\frac{A_{k-1}}{A_k} = 1-\theta,$ we have that $G_k \leq (1-\theta)^{k-k_0}G_{k_0} = (1-\theta)^{k-k_0}\frac{L + \mu}{2} \|\vx^* - \vx_{k_0}\|^2,$ and, thus:
$$
f(\vx_k) - f(\vx^*) \leq \Big(1 - \sqrt{\frac{\mu}{L}}\Big)^{k-k_0}\frac{L+\mu}{2} \|\vx^* - \vx_{k_0}\|^2.
$$
By the same arguments as in the proof of Proposition~\ref{prop:initial-steps}, $f(\vx_{k_0}) - f(\vx^*) \leq \frac{2L D^2}{k_0 + 4}.$ By strong convexity of $f,$ this implies that also $\mu\|\vx_{k_0} - \vx^*\|^2 \leq \frac{4L D^2}{k_0 + 4}.$ To complete the proof, it remains to argue that $\big(1 - \sqrt{\frac{\mu}{L}}\big)^{K_0-k_0} \mu\|\vx_{k_0} - \vx^*\|^2 \leq \big(1 - \sqrt{\frac{\mu}{L}}\big)^{K_0-k_0}\frac{4L D^2}{k_0 + 4} \leq 4r^2.$ This simply follows by  arguing that for the choice of $k_0$ from the statement of the lemma and $k_0 \leq K_0$, we have $\big(1 - \sqrt{\frac{\mu}{L}}\big)^{K_0-k_0} \frac{1}{k_0 + 4} \leq \frac{1}{K_0 + 4},$ while the rest follows from Proposition~\ref{prop:initial-steps}. This is not hard to show and is omitted.
\end{proof}
Finally, we have the following bound on the convergence of Algorithm~\ref{algo:acc-FW}.
\prelimthm*
\begin{proof}
Follows directly by applying the standard convergence bound for FW, Proposition~\ref{prop:initial-steps}, and Lemma~\ref{lemma:acc-conv}.
\end{proof}

Note that in the argument in Proposition~\ref{prop:initial-steps} we
could have also used the Away-Step Frank-Wolfe algorithm achieving
linear convergence for the burn-in phase. However, for the easy of
exposition we used the simpler bound for the warm-up; we will use the
Away-Step Frank-Wolfe algorithm in Section~\ref{sec:fullResult}. 

%
%
\subsection{Proofs and Results for Optimum in the Relative Interior of a Face of $\cx$ from Section~\ref{sec:fullResult}}
\label{appx:fullResult}

In this section we provide full technical details for the results in Section~\ref{sec:fullResult} and we also restate material from that section here once again to facilitate reading. 

We will now formulate the general case that subsumes the case from above. We assume that, given points $\vx_1, ..., \vx_m$ and a point $\vy,$ the following  problem is easily solvable:
\begin{equation} \tag{\ref{eq:convProject}}
  \min_{\substack{\vu = \sum_{i=1}^m \lambda_i \vx_i, \\ \vlambda \in \Delta_m}} \frac{1}{2}\|\vu - \vy\|^2. \end{equation}
In other words, we assume that the projection onto the convex hull of a given set of vertices can be implemented efficiently; however, we do not require access to a membership oracle anymore. Solving this problem amounts to minimizing a quadratic function over the probability simplex.  The size of the program $m$ from Eq.~\eqref{eq:convProject} corresponds to the size of the active set of the CG-type method employed within LaCG. Note that $m$ is never larger than the iteration count $k,$ and is often much lower than the dimension of the original problem. Further, there exist multiple heuristics for keeping the size of the active set small in practice (see, e.g.,~\cite{BPZ2017}). The projection from Eq.~\eqref{eq:convProject} does not require access to either the first-order oracle or the linear optimization oracle. Finally, due to Lemma~\ref{lemma:modified-agdp} stated below, we only need to solve this problem to accuracy of the order $\frac{\epsilon}{\sqrt{\mu L}}$, where $\epsilon$ is the target accuracy of the program. 

For simplicity, we illustrate the framework using AFW as the coupled CG method. However, the same ideas can be applied to other active-set-based methods such as PFW in a straightforward manner. Unlike in the previous subsection, the assumption that $\cx$ is a polytope is crucial here, as the linear convergence for the AFW algorithm established in~\cite{lacoste2015global} relies on a
constant, the \emph{pyramidal width}, that is only known to be bounded away
from $0$ for polytopes. For completeness, we provide the pseudocode for one iteration of
AFW (as stated in~\cite{lacoste2015global}) in Algorithm~\ref{algo:away-FWAppx} below.
%
%
%
In the following, the vector
$\vlambda_k \in \Delta_m$ with $m = |\cs_k|$ denotes the barycentric
coordinates of the current iterate $\vx_k$ over the active set
$\cs_k$.

\begin{algorithm}
\caption{Away-Step Frank-Wolfe Iteration: AFW($\vlambda, \cs, \vx$)}
\label{algo:away-FWAppx}
\begin{algorithmic}[1]
\State Set FW direction: $\vs = \argmin_{\vu \in \cx}\innp{\nabla f(\vx), \vu},$ $\vd^{\mathrm{FW}} = \vs - \vx$
\State Set Away direction: $\vv = \argmax_{\vu \in \cs}\innp{\nabla f(\vx), \vu},$ $\vd^{\mathrm{A}} = \vx - \vv$
\If{ $\innp{-\nabla f(\vx), \vd^{\mathrm{FW}}} \geq \innp{-\nabla f(\vx), \vd^{\mathrm{A}}}$}
\State $\vd = \vd^{\mathrm{FW}},$ $\gamma_{\max} = 1$
\Else 
\State $\vd = \vd^{\mathrm{A}},$ $\gamma_{\max} = \frac{\vlambda(\vv)}{1-\vlambda(\vv)}$
\EndIf
\State $\gamma' = \argmin_{\gamma \in [0, \gamma_{\max}]}f(\vx + \gamma \vd)$
\State $\vx' = \vx + \gamma' \vd$; update $\vlambda$ (to $\vlambda'$)
\State $\cs' = \{\vu \in \cs \cup \{\vs\}: \vlambda'(\vu) > 0\}$
\State \textbf{return} $\vx', \, \cs', \, \vlambda'$ 
\end{algorithmic}
\end{algorithm}

We will need the following fact that establishes the existence of a radius $r$ (and hence iteration $K_r$) from which onwards all active sets $\cs_k$ maintained by our algorithm ensure that $\vx^* \in \co(\cs_k)$ for all $k \geq K_r$.

\begin{fact}[Active set convergence]
  \label{fact:asConvergence}
  There exists $r > 0$ such that for any subset $\cs \subseteq \vertex(\cx)$ and point $\vx \in \cx$ with $\vx \in \co(\cs)$ and $\norm{\vx - \vx^*} \leq r$ it follows $\vx^* \in \co(\cs)$. 
  \begin{proof}
    Let $\cs \subseteq \vertex(\cx)$ be an arbitrary subset of vertices, so that \(\vx^* \not\in \co(\cs)\). As $\cs$ is closed, there exists \(0 < r_\cs \defeq \min_{\vx \in \cs} \norm{\vx - \vx^*}\). Let \(2r\) be the minimum over all such \(\cs\), which is bounded away from \(0\) as there are only finitely many such subsets. It follows that if $\norm{\vx - \vx^*} \leq r$ then $\vx^* \in \co(\cs_k)$. 
  \end{proof}
\end{fact}

Let $r_0$ denote the \emph{critical radius} from Fact~\ref{fact:asConvergence} and $K_0$ the \emph{critical iteration} so that \(\norm{\vx^* - \vx_{k}} \leq r_0\) is ensured for all $k \geq K_0$. The next proposition bounds the magnitude of $K_0$.

\begin{proposition}[Finite burn-in with linear rate] \label{proposition:BurnIn-Phase} Denote by
  $\delta$ the pyramidal width of the polytope $\cx$, as defined in
  \cite{lacoste2015global}. Then for all $k \geq K_0$ it holds \(\vx^* \in \co(\cs_k)\) and 
  for any algorithm that makes in each iteration at least
  as much progress as the Away-Step Frank-Wolfe Algorithm, we have the bound 
$$
K_0 \leq \frac{8 L}{\mu} \left( \frac{D}{\delta} \right)^2 \log \left(
  \frac{2(f(\vx_0) - f(\vx^*))}{\mu {r_0}^2} \right).
$$
\end{proposition}
\begin{proof}
  Since the algorithm makes at least as much progress as the Away-Step
  Frank-Wolfe algorithm, we can use the convergence rate of
  \cite{lacoste2015global} to bound the primal gap at step $k$. Using
  the $\mu$-strong convexity of $f$, we have that
  $f(\vx_k) - f(\vx^*) \geq \mu/2 \| \vx_k - \vx^*\|$, allowing us to
  relate the primal gap to the distance to the optimum. 
\end{proof}

%
To achieve local acceleration, we couple the AFW steps with a modification of the $\mu$AGD+ algorithm~\cite{cohen2018acceleration} that we introduce here. Unlike its original version~\cite{cohen2018acceleration}, the version  provided here (Lemma~\ref{lemma:modified-agdp}) allows coupling of the method with an arbitrary  sequence of points from the feasible set, it supports inexact minimization oracles, and it supports changes in the convex set (which correspond to active sets from AFW) on which projections are performed. These modifications are crucial to being able to achieve local acceleration without any additional knowledge about the polytope or the position of the minimizer $\vx^*.$ Further, we are not aware of any other methods that allow changes to the feasible set as described here, and, thus, the result from Lemma~\ref{lemma:modified-agdp} may be of independent interest.
\keylemma*
\begin{proof}
We first show by induction on $k$ that $\vx_k \in \cx.$ The claim is true initially, by the statement of the lemma. Assume that the claim is true for the iterates up to $k-1$. Then, $\vxh_k$ must be from $\cx,$ as it is a convex combination of $\vx_{k-1} \in \cx$ (by the inductive hypothesis) and $\vw_k \in \cc_{k} \subseteq \cx.$
By assumption, $\Tilde{\vx}_k \in \cx$, for all $k$. Thus, it must be $\vx_k \in \cx.$

The rest of the proof relies on showing that $A_k G_k \leq A_{k-1}G_{k-1} + \epsilon^m_{k} + \epsilon_{k-1}^m$ and on bounding $A_0 G_0,$ where $G_k$ is an approximate duality gap defined as $G_k = U_k - L_k.$ Here, the upper bound is defined as $U_k = f(\vx_k),$ while the lower bound on $L_k \geq f(\vx^*)$ can be defined as (see Appendix~\ref{appx:adgt}):
\begin{equation}\notag
    \begin{aligned}
        L_k \defeq & \frac{\sum_{i=0}^k a_i f(\vy_i) + \min_{\vu \in \cc_k}m_k(\vu) - \frac{\mu_0}{2}\|\vx^* - \vy_0\|^2}{A_k},
    \end{aligned}
\end{equation}
where $\mu_0 = L - \mu$ and 
$$
m_k(\vu) \defeq \sum_{i=0}^k a_i \innp{\nabla f(\vy_i), \vu - \vy_i} + \sum_{i=0}^k a_i \frac{\mu}{2}\|\vu - \vy_i\|^2 + \frac{\mu_0}{2}\|\vu - \vy_0\|^2.
$$
It is not hard to verify that:
$$
\argmin_{\vu \in \cc_k} \Big\{ -\innp{\vz_k, \vu} + \frac{\mu A_k + \mu_0}{2}\|\vu\|^2 \Big\} = \argmin_{\vu \in \cc_k} m_k(\vu), \quad \forall k.
$$
Let us start by bounding $A_0G_0.$ Recall that $a_0 = A_0 = 1$ and $\vx_0 = \vw_0.$ By smoothness of $f,$
\begin{equation}\label{eq:magdp-init-U}
    U_0 = f(\vx_0) = f(\vw_0) \leq f(\vy_0) + \innp{\nabla f(\vy_0), \vw_0 - \vy_0} + \frac{L}{2}\|\vw_0 - \vy_0\|^2.
\end{equation}
On the other hand, as $\mu_0 = L-\mu$ and $\vw_0$ is an $\epsilon^m_0$-approximate minimizer of $\argmin_{\vu \in \cc}m_0(\vu_0),$ we have:
\begin{equation}\label{eq:magdp-init-m}
    \min_{\vu \in \cc_0}m_0(\vu) \geq m_0 (\vw_0) - \epsilon^m_0 = \innp{\nabla f(\vy_0), \vw_0 - \vy_0} + \frac{L}{2}\|\vw_0 - \vy_0\|^2 - \epsilon^m_0.
\end{equation}
Combining Eqs.~\eqref{eq:magdp-init-U} and~\eqref{eq:magdp-init-m} with the definition of $L_k,$ we have that:
$$
A_0 G_0 \leq \frac{\mu_0 \|\vx^* - \vy_0\|^2}{2}  + \epsilon_0^m = \frac{(L-\mu)\|\vx^* - \vy_0\|^2}{2} + \epsilon_0^m. 
$$

To complete the proof, it remains to show that $G_k \leq \frac{A_{k-1}}{A_k}G_{k-1} = (1-\theta)G_{k-1}.$ Observe first, as $f(\vx_k) \leq f(\vxh_k),$ that we can bound the change in the upper bound as:
\begin{equation}\notag
    \begin{aligned}
        A_k U_k - A_{k-1}U_{k-1} &= A_k f(\vx_k) - A_{k-1}f(\vx_{k-1})\\
        &\leq a_k f(\vy_k) + A_k(f(\vxh_k) - f(\vy_k)) + A_{k-1}(f(\vy_{k}) - f(\vx_{k-1})).
    \end{aligned}
\end{equation}
Using smoothness and convexity of $f,$ we further have:
\begin{equation}\label{eq:magdp-change-in-U}
    \begin{aligned}
        A_k U_k - A_{k-1}U_{k-1} \leq & a_k f(\vy_k) + \innp{\nabla f(\vy_k), A_k \vxh_k - A_{k-1}\vx_{k-1} - a_k \vy_k} + \frac{A_k L}{2}\|\vxh_k - \vy_k\|^2.
    \end{aligned}
\end{equation}

By the definition of $L_k,$ the change in the lower bound is:
\begin{equation}\label{eq:magdp-change-in-L-1}
 A_k L_k - A_{k-1}L_{k-1} = a_k f(\vy_k) + m_k(\vw_k^*) - m_{k-1}(\vw_{k-1}^*),   
\end{equation}
where $\vw_k^* = \argmin_{\vu \in \cc_k} m_k(\vu)$. 

To bound $m_k(\vw_k^*) - m_{k-1}(\vw_{k-1}^*),$ observe first that:
\begin{align}\label{eq:inexact-m}
    m_k(\vw_k^*) - m_{k-1}(\vw_{k-1}^*) \geq m_k(\vw_k) - m_{k-1}(\vw_{k-1}^*) - \epsilon^m_k.
\end{align}
as $\vw_k \in \cc_k$ is an $\epsilon^m_k$-approximate minimizer of $m_k.$
Further, 
observe that $m_k(\vu) = m_{k-1}(\vu) + a_k \innp{\nabla f(\vy_k), \vu - \vy_k} + a_k \frac{\mu}{2}\|\vu - \vy_k\|^2$. Hence, we have:
\begin{equation}\label{eq:m-change-1}
\begin{aligned}
    m_k(\vw_k)& -  m_{k-1}(\vw_{k-1}^*)\\
    &= a_k \innp{\nabla f(\vy_k), \vw_k - \vy_k} + a_k \frac{\mu}{2}\|\vw_k - \vy_k\|^2
    + m_{k-1}(\vw_k) - m_{k-1}(\vw_{k-1}^*).
\end{aligned}
\end{equation}
As $m_k(\vu)$ can be expressed as the sum of $\frac{\mu A_k + \mu_0}{2}\|\vu\|^2$ and terms that are linear in $\vu,$ it is $(\mu_0 + \mu A_k)$-strongly convex. Observe that, as $\vw_{k-1}^*$ minimizes $m_{k-1}$ over $\cc_{k-1}$ and $\vw_k \in \cc_k \subseteq \cc_{k-1},$ by the first-order optimality condition, we have  $\innp{\nabla m_{k-1}(\vw_{k-1}^*), \vw_k - \vw_{k-1}^*} \geq 0.$ Thus, it further follows that:
\begin{equation}\label{eq:m-change-2}
\begin{aligned}
    m_{k-1}(\vw_k) 
    &\geq m_{k-1}(\vw_{k-1}^*) + \frac{\mu_0 + \mu A_{k-1}}{2}\|\vw_k - \vw_{k-1}^*\|^2 .
\end{aligned}
\end{equation}
Next, observe that, as $m_{k-1}$ is $(\mu_0 + \mu A_{k-1})$-strongly convex, $\vw_{k-1}^*$ minimizes $m_{k-1}$, and $\vw_{k-1}$ is an approximate minimizer, we have:
\begin{equation}\label{eq:m-strong-cvxity}
    \frac{\mu_0 + \mu A_{k-1}}{2}\|\vw_{k-1} - \vw_{k-1}^*\|^2 \leq m_{k-1}(\vw_{k-1}) - m_{k-1}(\vw_{k-1}^*) \leq \epsilon^m_{k-1}.
\end{equation}
Using Young's inequality ($(a+b)^2 \leq 2a^2 + 2b^2$ and so $a^2 \geq \frac{(a+b)^2}{2} - {b^2}$), we have, using Eq.~\eqref{eq:m-strong-cvxity}, that:
\begin{align*}
 \frac{\mu_0 + \mu A_{k-1}}{2}\|\vw_k - \vw_{k-1}^*\|^2 &\geq \frac{\mu_0 + \mu A_{k-1}}{4}\|\vw_k - \vw_{k-1}\|^2 - \frac{\mu_0 + \mu A_{k-1}}{2}\|\vw_{k-1} - \vw_{k-1}^*\|^2\\
 &\geq \frac{\mu_0 + \mu A_{k-1}}{4}\|\vw_k - \vw_{k-1}\|^2 - \epsilon^m_{k-1}.
\end{align*}
Combining the last inequality with Eqs.~\eqref{eq:inexact-m}--\eqref{eq:m-change-2},  we have:
\begin{align*}
    m_k(\vw_k^*) - m_{k-1}(\vw_{k-1}^*) \geq & a_k \innp{\nabla f(\vy_k), \vw_k - \vy_k} + a_k \frac{\mu}{2}\|\vw_k - \vy_k\|^2\\
    & + \frac{\mu_0 + \mu A_{k-1}}{4}\|\vw_k - \vw_{k-1}\|^2 - \epsilon^m_{k-1} - \epsilon^m_k.
\end{align*}
Using that $\mu_0 \geq 0,$ $\theta = \frac{a_k}{A_k},$ and applying Jensen's inequality to the last expression, 
%
\begin{align*}
m_{k}(\vw_k^*)& - m_{k-1}(\vw_{k-1}^*) \\
&\geq a_k \innp{\nabla f(\vy_k), \vw_k - \vy_k} + \frac{\mu A_k}{4}\|\vw_k - (1-\theta)\vw_{k-1} - \theta \vy_k\|^2 - \epsilon^m_k - \epsilon^m_{k-1}.
\end{align*}
It is not hard to verify that $\vxh_k - \vy_k = \theta(\vw_k - (1-\theta)\vw_{k-1} - \theta \vy_k).$ Hence, combining the last inequality with Eq.~\eqref{eq:magdp-change-in-L-1}:
\begin{equation}\label{eq:magdp-change-in-L-2}
    A_k L_k - A_{k-1}L_{k-1} \geq a_k f(\vy_k) + a_k \innp{\nabla f(\vy_k), \vw_k - \vy_k} + \frac{\mu A_k}{4\theta^2}\|\vxh_k - \vy_k\|^2 - \epsilon^m_k - \epsilon^m_{k-1}.
\end{equation}
Finally, combining Eqs.~\eqref{eq:magdp-change-in-U} and~\eqref{eq:magdp-change-in-L-2}, we have:
\begin{align*}
    A_kG_k - A_{k-1} G_{k-1} \leq & \innp{\nabla f(\vy_k), A_k \vxh_k - A_{k-1}\vx_{k-1} - a_k \vw_k} + \frac{A_k}{2}\Big(L - \frac{\mu}{2\theta^2}\Big)\|\vxh_k - \vy_k\|^2 \\
    &+ \epsilon^m_k + \epsilon^m_{k-1}\\
    \leq & \epsilon^m_k + \epsilon^m_{k-1},
\end{align*}
as $\vxh_k = \frac{A_{k-1}}{A_k}\vx_{k-1} + \frac{a_k}{A_k}\vw_k$ and $\theta = \sqrt{\frac{\mu}{2L}},$ completing the proof.
\end{proof}

A simple corollary of Lemma~\ref{lemma:modified-agdp} that will be useful for our analysis is as follows. It shows that if the algorithm from Lemma~\ref{lemma:modified-agdp} is not restarted too often, we do not lose more than a constant factor (two) in the final bound on the iteration count.

\begin{corollary}\label{cor:restarts}
Define a restart of the method from Lemma~\ref{lemma:modified-agdp} as setting $a_k = A_k = 1,$ $\vy_k = \vx_{k-1},$ $\vw_k = \vy_k,$ and $\vz_k = L\vy_k - \nabla f(\vy_k)$. Let $\epsilon^m_i = \frac{a_i}{2}\bar{\epsilon},$ for some $\epsilon^m \geq 0.$ If the method is restarted no more frequently than every $\frac{2}{\theta}\log(1/(2\theta^2) - 1)$ iterations, where $\theta = \sqrt{\mu/(2L)},$ then:
$$
f(\vx_k) - f(\vx^*) \leq \frac{L - \mu}{\mu}\big(1 - \theta\big)^{k/2}(f(\vx_0) - f(\vx^*)) + 2\bar{\epsilon}.
$$
\end{corollary}
\begin{proof}
Denote $H = \frac{2}{\theta}\log(1/(2\theta^2) - 1)$. Let the iterations at which the restarts happen be denoted as $k_0 = 0,$ $k_1,$ $k_2,...,$ and note that, by assumption, $k_i \geq k_{i-1} + H$, for all $i \geq 1$. Assume w.l.o.g.~that each $k_i$ is even. We first claim that we have the following contraction between the successive restarts:
\begin{equation}\label{eq:contraction-b/w-restarts}
    f(\vx_{k_i}) - f(\vx^*) \leq (1-\theta)^{(k_i-k_{i-1})/2}(f(\vx_{k_{i-1}})-f(\vx^*)) + \bar{\epsilon}.
\end{equation}
To prove the claim, observe first using $k_i - k_{i-1} \geq H$ that:
\begin{align}\label{eq:half-contraction}
    \frac{L-\mu}{\mu} (1-\theta)^{k_i - k_{i-1}} \leq \Big(\frac{1}{2\theta^2} - 1\Big)(1-\theta)^{\frac{1}{\theta} \log(\frac{1}{2\theta^2} - 1)}(1-\theta)^{\frac{k_i - k_{i-1}}{2}} \leq (1-\theta)^{\frac{k_i - k_{i-1}}{2}}.
\end{align}
Applying Lemma~\ref{lemma:modified-agdp} with $\vx_{k_{i-1}}$ as an initial point and using strong convexity of $f$ (which implies $f(\vx_{k_{i-1}}) - f(\vx^*) \geq \frac{\mu}{2}\|\vx_{k_{i-1}} - \vx^*\|^2$), we have:
$$
f(\vx_{k_i}) - f(\vx^*)  \leq \frac{L-\mu}{\mu} (1-\theta)^{k_i - k_{i-1}}(f(\vx_{k_{i-1}}) - f(\vx^*)) + \bar{\epsilon}.
$$
Thus, combining the last inequality with~\eqref{eq:half-contraction}, inequality~\eqref{eq:contraction-b/w-restarts} follows.

Applying Eq.~\eqref{eq:contraction-b/w-restarts} recursively and using that $k_i - k_{i-1}\geq H$, we further have:
\begin{equation}\label{eq:total-contraction-b/w-restarts}
\begin{aligned}
    f(\vx_{k_i}) - f(\vx^*) &\leq (1-\theta)^{k_i/2} (f(\vx_0) - f(\vx^*)) + \bar{\epsilon}\sum_{j=0}^i (1-\theta)^{j H/2}\\
    &\leq (1-\theta)^{k_i/2} (f(\vx_0) - f(\vx^*)) + 2\theta^2 \bar{\epsilon}.
\end{aligned}
\end{equation}
To complete the proof, fix an iteration $k$ and let $k_i$ be the last iteration up to $k$ in which a restart happened. Applying Lemma~\ref{lemma:modified-agdp} with $k_i$ as the initial point, we get:
\begin{align*}
    f(\vx_k) - f(\vx^*) &\leq \frac{L - \mu}{\mu}\big(1 - \theta\big)^{k-k_i}(f(\vx_{k_i}) - f(\vx^*)) + \bar{\epsilon}\\
    &\leq \frac{L - \mu}{\mu}\big(1 - \theta\big)^{k/2}(f(\vx_{0}) - f(\vx^*)) + (1+2\theta^2)\bar{\epsilon}.
\end{align*}
It remains to note that $\theta^2 = {\mu/(2L)} \leq 1/2$.
\end{proof}

To obtain locally accelerated convergence, we  show that from some iteration onwards, we can apply the accelerated method from
Lemma~\ref{lemma:modified-agdp} with $\cc_k$ being the convex hull of the vertices from the active set and the
sequence $\Tilde{\vx}_k$ being the sequence of the AFW steps. The pseudocode for the LaCG-AFW algorithm is provided in
Algorithm~\ref{algo:acc-FW-rel-int} (Algorithm~\ref{algo:acc-FW-rel-intAppx} in the appendix). For completeness, pseudocode for one iteration of the accelerated method (ACC), which is based on Eq.~\eqref{eq:modified-magdp} is provided in Algorithm~\ref{algo:acc-stepAppx}. 

\begin{algorithm}
\caption{Locally Accelerated Conditional Gradients with Away-Step Frank-Wolfe (LaCG-AFW)}
\label{algo:acc-FW-rel-intAppx}
\begin{algorithmic}[1]
\State Let $\vx_0 \in \cx$ be an arbitrary point, $\cs_0 = \{\vx_0\}$, $\vlambda_0 = [1]$
\State Let $\vy_0 = \vxh_0 = \vw_0 = \vx_0$, $\vz_0 = -\nabla f(\vy_0) + L \vy_0$, $\cc_1 = \mathrm{co}(\cs_0)$
\State $a_0 = A_0 = 1,$ $\theta = \sqrt{\frac{\mu}{2L}},$ $\mu_0 = L-\mu$
\State $H = \frac{2}{\theta}\log(1/(2\theta^2) - 1)$ \Comment{Minimum restart period}
\State $r_f = \false$, $r_c = 0$ \Comment{Restart flag and restart counter initialization}
\For{$k=1$ to $K$} 
\State $\vx_k^{\mathrm{AFW}}, \, \cs_k^{\mathrm{AFW}}, \, \vlambda_k^{\mathrm{AFW}} = \mathrm{AFW}(\vx_{k-1}^{\mathrm{AFW}}, \, \cs_{k-1}^{\mathrm{AFW}}, \, \vlambda_{k-1}^{\mathrm{AFW}})$ \Comment{Independent AFW update}
\State $A_k = A_{k-1}/(1-\theta),$ $a_k = \theta A_k$
\State $\vxh_k,\, \vz_k,\, \vw_k = \mathrm{ACC}(\vx_{k-1}, \vz_{k-1}, \vw_{k-1}, \mu, \mu_0, a_k, A_k, \cc_k)$
\If{$r_f$ and $r_c \geq H$}
\Comment{Restart
  criterion is met}  \label{restart:project-appx}
\State $\vy_k = \argmin\{f(\vx_k^{\mathrm{AFW}}),\, f(\vxh_k)\}$ 
\State $\cc_{k+1} = \mathrm{co}(\cs_k^{\mathrm{AFW}})$ \Comment{Updating feasible set for the accelerated sequence}
\State $a_k = A_k = 1$, $\vz_k = -\nabla f(\vy_k) + L\vy_k$ \Comment{Restarting accelerated sequence}
\State $\vxh_k = \vw_k = \argmin_{\vu \in \cc_{k+1}}\{-\innp{\vz_k, \vu} + \frac{L}{2}\|\vu\|^2\}$
\State $r_c = 0$, $r_f = \false$ \Comment{Resetting the restart indicators}
\Else
\If{$\cs_k^{\mathrm{AFW}} \setminus \cs_{k-1}^{\mathrm{AFW}} \neq \emptyset$} \Comment{If a vertex was added to the active set}
\State $r_f = \true$ \Comment{Raise restart flag}
\EndIf
\If{$r_f = \false$} \Comment{If AFW did not add a vertex since last restart}
\State $\cc_{k+1} = \mathrm{co}(\cs_k^{\mathrm{AFW}})$ \Comment{Update the feasible set}
\Else 
\State $\cc_{k+1} = \cc_k$ \Comment{Freeze the feasible set}
\EndIf
\EndIf
\State $\vx_k = \argmin\{f(\vx_k^{\mathrm{AFW}}), \, f(\vxh_k),\, f(\vx_{k-1})\}$ \Comment{Choose the better step + monotonicity} \label{Monotonicity1}
\State $r_c = r_c + 1$ \Comment{Increment the restart counter}
\EndFor
\end{algorithmic}
\end{algorithm}

\begin{algorithm}
\caption{Accelerated Step ACC($\vx_{k-1}, \vz_{k-1}, \vw_{k-1}, \mu, \mu_0, a_k, A_k, \cc_k$)}\label{algo:acc-stepAppx}
\begin{algorithmic}[1]
\State $\theta = a_k/A_k$
\State $\vy_{k} = \frac{1}{1+\theta}\vx_{k-1} + \frac{\theta}{1+\theta}\vw_{k-1}$
\State $\vz_k = \vz_{k-1} - a_k \nabla f(\vy_k) + \mu a_k \vy_k,$ $\;\vw_k = \argmin_{\vu \in \cc_k}\{- \innp{\vz_k, \vu} + \frac{\mu A_k + \mu_0}{2}\|\vu\|^2\}$
\State $\vxh_k = (1-\theta)\vx_{k-1} + \theta \vw_k$
\State\Return $\vxh_k, \vz_k, \vw_k$
\end{algorithmic}
\end{algorithm}

\mainthm*
\begin{proof}
The statement of the theorem is a direct consequence of the following observations about Algorithm~\ref{algo:acc-FW-rel-int} (Algorithm~\ref{algo:acc-FW-rel-intAppx} in the appendix). First, observe that the AFW algorithm is run independently of the accelerated sequence, and, in particular, the accelerated sequence has no effect on the AFW-sequence whatsoever. Further, in any iteration, the set $\cc_k$ that we project onto is the convex hull of some active set $\cs_i^{\mathrm{AFW}} \subseteq \cx$ for some $0\leq i \leq k-1$ implying $\vxh_k \in \cx$ -- each $\vxh_k$ is hence feasible. 

Now, as in any iteration $k$ the solution outputted by the algorithm is $\vx_k = \argmin\{f(\vx_k^{\mathrm{AFW}}), \, f(\vxh_k)\},$ the algorithm never makes less progress than AFW. This immediately implies (by a standard AFW  guarantee; see~\cite{lacoste2015global} and Proposition~\ref{proposition:BurnIn-Phase}) that for $k \geq  \frac{8 L}{\mu} \left( \frac{D}{\delta} \right)^2 \log \big(  \frac{f(\vx_0) - f(\vx^*)}{\epsilon} \big)$, it must be that $f(\vx_k) - f(\vx^*) \leq \epsilon$, which establishes the unaccelerated part of the minimum in the asserted rate. 

Further, there exists an iteration $K \leq K_0$ such that for all $k \geq K$ it holds $\vx^* \in \co(\cs_k^{\mathrm{AFW}})$ (see Proposition~\ref{proposition:BurnIn-Phase}). Let $K$ be the first such iteration. Then, the AFW  algorithm must have added a vertex in iteration $K$ as otherwise $\vx^* \in \co(\cs_{k-1}^{\mathrm{AFW}})$, contradicting the minimality of $K$. 
Due to the restarting criterion from Algorithm~\ref{algo:acc-FW-rel-int}, a restart must happen by iteration $K_0 + H.$ 
Thus, for $k \geq K_0 + H$, it must be $\vx^* \in \cc_k$.


Further, the restarting criterion implies that we perform at least $H = \frac{2}{\theta}\log(1/(2\theta^2) - 1)$ iterations between successive restarts of the accelerated sequence $\{\vxh_k\}$ and, unless a restart happens, we also have that $\cc_k \subseteq \cc_{k-1}$. Thus, starting from iteration $K_0 + H$, Lemma~\ref{lemma:modified-agdp} and Corollary~\ref{cor:restarts} apply and 
$\{\vx_k\}$ converges to to $\vx^*$ at an accelerated rate. 
The remaining $2 \sqrt{\frac{L}{\mu}} \log \left( \frac{(L-\mu) r_0^2 }{2\epsilon}\right)$ part of the minimum in the asserted rate follows now by Corollary~\ref{cor:restarts}. 
\end{proof}

\begin{remark}[Inexact projection oracles.]
For simplicity, we stated Theorem~\ref{thm:main} assuming exact minimization oracle ($\epsilon_i^m = 0$ in Lemma~\ref{lemma:modified-agdp}). Clearly, it suffices to have $\epsilon_i^m = \frac{a_i}{8}\epsilon$ and invoke Theorem~\ref{thm:main} for target accuracy $\epsilon/2.$
\end{remark}

\begin{remark}[Further improvements to the practical performance.] 
If in any iteration the Wolfe gap of the accelerated sequence on $\cc_k$, $\max_{\vu \in \cc_k}\innp{\nabla f(\vx_k), \vx_k - \vu},$ is smaller than the target accuracy of the projection subproblem (order-$\frac{\epsilon}{\sqrt{\mu L}}$), then $f$ cannot be reduced by more than order-$\frac{\epsilon}{\sqrt{\mu L}}$ on $\cc_k,$ and one can safely perform an early restart without affecting the theoretical convergence guarantee. 
\end{remark}
  
\begin{remark}[Running Algorithm~\ref{algo:acc-FW-rel-int} when $\vx^* \in \interior(\cx)$]
  Usually we do not know ahead of time whether
  $\vx^* \in \interior(\cx)$ or whether $\vx^*$ is in the relative
  interior of a face of $\cx$. However, we can simply run
  Algorithm~\ref{algo:acc-FW-rel-int} agnostically, as in the case where
  $\vx^* \in \interior(\cx)$ we still exhibit local acceleration with
  an argumentation and convergence analysis analogous to the one in
  Section~\ref{sec:warm-optim-inter}. In particular, the assumptions
  of Section~\ref{sec:fullResult} are only needed to establish a bound
  for the estimation in Proposition~\ref{proposition:BurnIn-Phase}.
\end{remark}

\begin{remark}[Variant relying exclusively on a linear optimization oracle]
Similar as in the Conditional Gradient Sliding (CGS) algorithm \cite{lan2016conditional} we can also solve the projection problems using (variants of) CG. The resulting algorithm is then fully projection-free similar to CGS. 
In fact, a variant of CGS is recovered if we  ignore the AFW steps and only run the accelerated sequence with such projections realized by CG. 
\end{remark}

\section{Computational Results}
\label{appx:compresults}

We provide a detailed comparison of the performance of different LaCG variants relative to other state of the art algorithms, comparing the primal gap and dual gap evolution both in terms of the iteration count and in terms of wall-clock time. For the example over the Birkhoff polytope, the MIPLIB instance and the probability simplex we use the stepsize rule $\gamma_t = \min\left\{\frac{\innp{\nabla f(\vx_t), \vx_t - \vv_t}}{L\norm{\vx_t - \vv_t}^2}, \gamma_{\text{max}}\right\}$ for all the algorithms that do not have a fixed step size rule, where $\gamma_{\text{max}}$ is the maximum step size that can be taken without leaving the polytope. In the video colocalization and the traffic network example we use exact linesearch for all the algorithms whenever possible (as was done in the video colocalization case in \cite{lacoste2015global} and \cite{joulin2014efficient}). Regarding the linear optimization oracle for the MIPLIB and video-colocalization instance, we used Gurobi to solve the MIP problem with linear cost function and the shortest path problem over the DAG respectively.

Note that when the algorithm starts running, the CG variant will add a vertex to its active set in the first iteration, at which point the set $\cc_k$ will be frozen and will contain only two vertices until the next restart happens. For this reason, the convergence in the first $H$ iterations is driven by the CG steps, and due to the low overhead of computing the accelerated steps over $\cc_k$ with $|\cc_k| = 2$, the wall-clock performance in the first $H$ iterations will be approximately equal to that of the CG variant running by itself.

\subsection{Computational Enhancements to Algorithm~\ref{algo:acc-FW-rel-int}}

In order to speed up the convergence of the algorithm when the burn-in phase has not been completed we can substitute Line~\ref{Monotonicity1} in Algorithm~\ref{algo:acc-FW-rel-int} (or Line~\ref{Monotonicity2} in Algorithm~\ref{algo:acc-FW-rel-intAppx}) with the following:

\begin{algorithm}[H]
\caption{LaCG Enhancement}
\label{algo:enhancementAppx}
\begin{algorithmic}[1]
\If{$f(\vx_{k-1}) \leq f(\vx_k^{\mathrm{AFW}})$ and $f(\vx_{k-1}) \leq f(\vxh_k)$}
\State $\vx_{k} = \vx_{k-1}$
\If{$f(\vx_k) \leq f(\vx_k^{\mathrm{AFW}})$ and $\cc_{k+1} \subseteq \cs_{k+1}^{\mathrm{AFW}}$}
\State $\cs_{k+1}^{\mathrm{AFW}} = \vertex \left( \cc_{k+1} \right)$ \label{Alg:culling}
\State $\vx_k^{\mathrm{AFW}} = \vx_k$ 
\EndIf
\Else 
\State $\vx_k = \argmin\{f(\vx_k^{\mathrm{AFW}}), \, f(\vxh_k)\}$
\EndIf
\end{algorithmic}
\end{algorithm}

Note that although this means that the AFW sequence is not independent of the accelerated steps anymore, this does not affect the theoretical guarantees shown in Theorem~\ref{algo:acc-stepAppx}. The previous operation leads to greater progress during the burn-in phase, as after a restart the accelerated sequence active set $\cc_{k+1}$ is usually frozen, and the accelerated steps tend to converge to the function minimizer over $\cc_{k+1}$ with relative ease, progressing very quickly before stagnating (as little further progress can be made over $\cc_{k+1}$). If this happens and $f(\vxh_k) \leq f(\vx_k^{\mathrm{AFW}})$ along with $\cc_{k+1} \subseteq \mathrm{co}(\cs_{k+1}^{\mathrm{AFW}})$ (this is equivalent to the AFW steps not having dropped a vertex contained in $\cc_{k+1}$) the AFW steps can pick up progress from where the accelerated sequence is located.

\subsection{Solving Problem~\eqref{eq:magdp-equiv-min}}

At each iteration, LaCG has to solve the following subproblem to accuracy $\frac{\epsilon a_i}{8}$:
\begin{align*}
\vw_k = \argmin_{u \in \cc_k } \left\{ -\innp{\vz_k, \vu} + \frac{\mu A_k + \mu_0}{2} \norm{\vu}^2 \right\},
\end{align*}
where $\vz_k \in \mathbb{R}^n$, $\mu$, $\mu_0,$ and $A_k$ are given. This problem over the convex hull of $\cc_k$ can transformed to one over the probability simplex by noting that $\vu = \mathcal{V}_k\vlambda$, where $\mathcal{V}_k$ is the matrix that contains the elements in $\cc_k$ as column vectors and $\vlambda \in \Delta_{|\cc_k|}$, where $|\cc_k|$ is the cardinality of the set $\cc_k$. Rewritting the previous subproblem leads to:
\begin{align}
\vlambda_k = \argmin_{\vlambda \in \Delta_{|\cc_k|}} \left\{ -\innp{ \mathcal{V}_k \vz_k,  \vlambda} + \frac{\mu A_k + \mu_0}{2} \vlambda^T\mathcal{V}_k^T\mathcal{V}_k\vlambda  \right\}, \label{eq:subproblemLambda}
\end{align}
where now the solution to our problem is provided by $\vw_k = \mathcal{V}_k \vlambda_k$. We begin by noting that whenever the AFW step adds a vertex to $\cs_k$, the set $\cc_k$ is frozen and remains fixed until the next restart happens. This effectively means that the term $\mathcal{V}_k^T\mathcal{V}_k$ in Equation~\ref{eq:subproblemLambda} remains fixed and only the term $\mathcal{V}_k \vz_k$ changes over iterations until the next restart happens. If, on the other hand, the AFW step removes a vertex from $\cs_k$ and the set $\cc_k$ is not frozen, both $\mathcal{V}_k^T\mathcal{V}_k$ and $\mathcal{V}_k \vz_k$ have to be updated at that iteration. In our experiments, the AFW step usually adds a vertex to $\cs_k$ immediately after a restart has happened and so the set $\cc_k$ remains frozen for most of the iterations, and we only need to update $\mathcal{V}_k \vz_k$ at each iteration, which has a complexity of $\mathcal{O}(|\cc_k| n)$

At each iteration, we solve Problem~\ref{eq:subproblemLambda} using Nesterov accelerated gradient descent \cite{nesterov2018introductory} for smooth convex or strongly convex functions. In order to do so, we calculate the largest and smallest eigenvalue of the matrix $\mathcal{V}_k^T\mathcal{V}_k$ when these are updated, and this is done with scipy's ARPACK package, which for symmetric matrices uses a variant of the Lanczos method. As an initial point for Nesterov accelerated gradient descent, we use the solution to the problem from the previous iteration, i.e., $\vlambda_{k-1}$, if $\cc_k = \cc_{k-1}$, which allows the algorithm to find a suitable solution after only a few iterations.

In order to project onto the simplex, we use the $\mathcal{O}\left( n\log n \right)$ algorithm described in \cite[Algorithm 1]{duchi2008efficient}. An alternative would be to use a negative entropy regularizer in the implementation of Nesterov method, for which Bregman projection steps can be computed in closed form. However, we found in our experiments that Euclidean projections using \cite[Algorithm 1]{duchi2008efficient} were faster, and thus opted for using them despite their slightly worse theoretical guarantee. 

\subsection{LaCG over the Probability Simplex} \label{appx:simplexEnhancements}

The probability simplex is a simple polytope for which efficient projection operators exist, with complexity $\mathcal{O}(n\log n)$. Due to the existence of these operators and the $\mathcal{O}(\sqrt{L/\mu}\log 1/\epsilon)$ global convergence guarantee of accelerated projected methods, CG methods are seldom used over this feasible region (despite the $\mathcal{O}(n)$ complexity for the linear optimization oracle over the simplex).

Despite being a toy-example, the structure of this polytope lends itself well to several computational simplifications. As we mentioned earlier, there is no need to maintain an active set in this case, as a single pass over the current iterate (which has complexity $\mathcal{O}(n)$) allows us to recover the active set by retrieving the non-zero components of the current point, whose corresponding standard orths correspond to the elements in the active set. This means that the away step oracle simply returns the largest gradient component over the non-zero coordinates of the current iterate.

Furthermore, if we consider the subproblem that needs to be solved at each iteration of the LaCG algorithm, shown in Equation~\eqref{eq:subproblemLambda}, and we note that $\mathcal{V}_k^T\mathcal{V}_k = I$, we see that the subproblem can be rephrased as follows:

\begin{align*}
\vlambda_k & = \argmin_{\vlambda \in \Delta_{|\cc_k|}} \left\{ -\innp{ \mathcal{V}_k \vz_k,  \vlambda} + \frac{\mu A_k + \mu_0}{2} \vlambda^T\vlambda  \right\} \\
& = \argmin_{\vlambda \in \Delta_{|\cc_k|}}  \norm{\vlambda - \frac{\mathcal{V}_k \vz_k}{\mu A_k + \mu_0}}_2^2,
\end{align*}
where the term $\mathcal{V}_k \vz_k$ can be efficiently computed in $\mathcal{O}(n)$ time, and a single call to the simplex projection over $\Delta_{|\cc_k|}$ will return the solution to the subproblem in barycentric coordinates.

\subsubsection{Going from the probability simplex to the $\ell_1$-ball}

Lasso regression is a problem of interest in the benchmarking of many  first-order methods, where the goal is to solve a quadratic problem with $f(\vx) = \norm{A\vx - \vb}^2_2 = \frac{\vx^T Q \vx}{2} + \vc^T\vx$ over a scaled $\ell_1^n$-ball in $\rr^n$, where $Q$ is positive semidefinite. It would be advantageous if all the computational simplifications that applied to the simplex could be extended to the $\ell_1$-ball. In fact, there is a simple change of variables that allows us to write a problem over the simplex that is equivalent to that of the lasso. Consider $\vz \in \Delta_{2n}$, and define $x_i = z_i - z_{n+i}$ for all $i \in \llbracket 1, n \rrbracket$, then:
\begin{align}
\min\limits_{\vx \in \ell_1^n} \frac{\vx^T Q \vx}{2} + \vc^T\vx = \min\limits_{\vz \in \Delta_{2n}}  \vz^T \left[
\begin{array}{c|c}
Q & -Q \\
\hline
-Q & Q
\end{array}
\right] \vz + \left[
\begin{array}{c}
c \\
\hline
-c
\end{array}
\right]^T \vz \label{lassoRegression}
\end{align}

In order for the LaCG algorithm to be useful, we need the objective function to be positive definite, in order to achieve acceleration. But the equivalent problem over the simplex is only convex, as the determinant of the block-matrix on the right-hand side of Equation~\eqref{lassoRegression} is equal to zero. Therefore, the LaCG algorithm has to be applied directly to the problem over the $\ell_1$-ball, shown on the left-hand side. This means that there is no way of rewriting the lasso problem as a strongly convex problem over the probability simplex.

Note that the DICG algorithm in \cite{garber2013linearly} is applicable to $0/1$ polytopes (such as the probability simplex) with $L$-smooth and $\mu$-strongly convex objective functions. If we perform this change of variables to go from a lasso regression problem over the $\ell_1$-ball to one over the probability simplex, the resulting objective function over the simplex will no longer be strongly convex, and therefore the theoretical guarantees in \cite{garber2013linearly} no longer hold in this case. However, the DICG algorithm for general polytopes \cite{bashiri2017decomposition} can be applied to the lasso regression problem, and its theoretical guarantees hold true.

\subsection{On the Evolution of $\cc$ and $\cs$}

\begin{figure*}[ht!]
    \centering
    \hspace{\fill}
    \subfloat[Cardinality of the AFW active set for the Birkhoff example ($n =1600$).]{{\includegraphics[width= .45\textwidth]{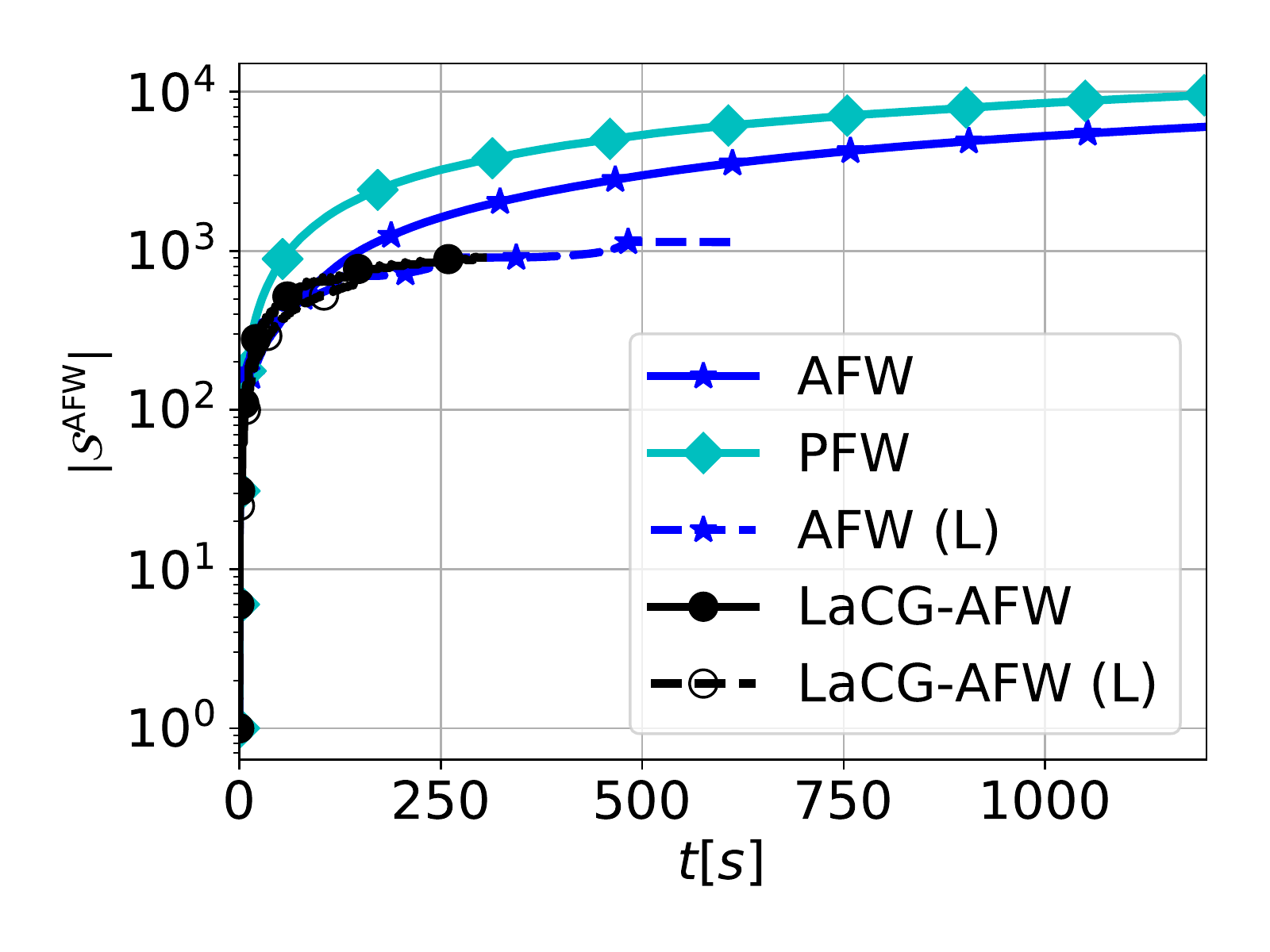} }\label{fig:cardinality1}}%
    \hspace{\fill}
    \subfloat[Cardinality of the AFW active set for the MIPLIB example ($n =504$).]{{\includegraphics[width = .45\textwidth]{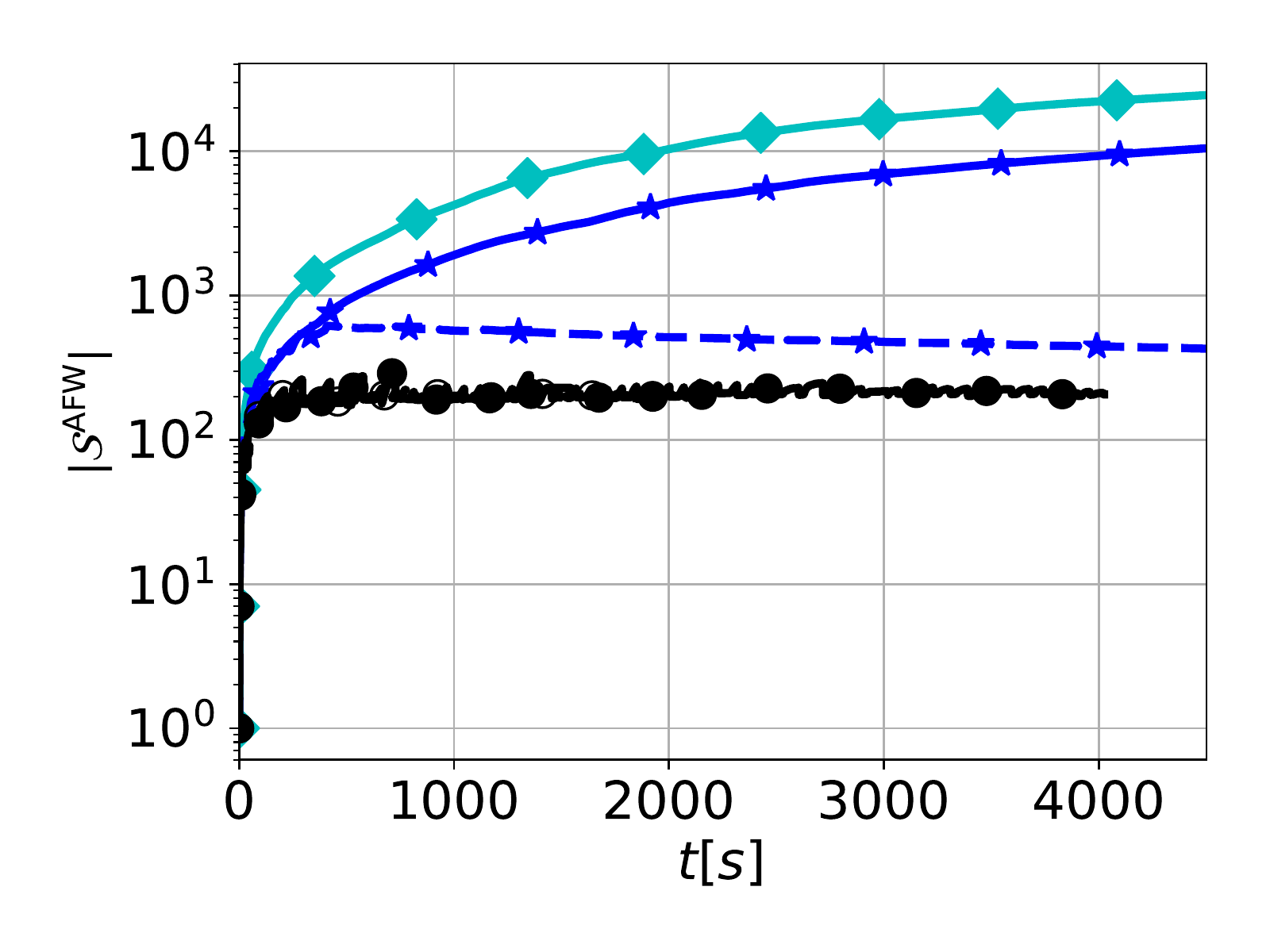} }\label{fig:cardinality2}}%
    \hspace*{\fill}
    \caption{Comparison of the evolution of the cardinality of the active set $\cs_{k+1}^{\mathrm{AFW}}$ for two of the examples.}%
    \label{fig:cardinality}%
\end{figure*}

The LaCG algorithm benefits from having a sparse active set. This is due to the fact that finding the away vertex in the AFW algorithm has complexity $\mathcal{O}(|\cs_{k+1}|n)$ and Problem~\eqref{eq:subproblemLambda} will in general be easier to solve the smaller the cardinality of the set $\cc_{k+1}$. Because of this, it is often useful to cull the active sets $\vertex(\cc_{k+1})$ and $\cs_{k+1}$ to promote sparsity. This can be done in conjunction with the operations in Algorithm~\ref{algo:enhancementAppx}. More specifically, before Line~\ref{Alg:culling}, we can discard the stale vertices in the convex decomposition of $\vxh_k$ over the set $\cc_{k+1}$, that is, we can eliminate the vertices in $\cc_{k+1}$ that have a zero coefficient when $\vxh_k$ is expressed as the convex combination of the elements in $\cc_{k+1}$ (also referred to as barycentric coordinates). This can be easily done as the accelerated sequence maintains a decomposition of the current point in the barycentric coordinates of $\cc_{k+1}$ in order to solve the projection subproblem. This culling can be performed regardless of which CG variant is used in the LaCG algorithm.

The evolution of the cardinality of the active set $\cs^{\mathrm{AFW}}$ in terms of time is shown in Fig.~\ref{fig:cardinality}\subref{fig:cardinality1}-\ref{fig:cardinality}\subref{fig:cardinality2} for two of the examples in the computational section. As can be seen in Fig.~\ref{fig:cardinality}, this culling of the active set is effective in keeping in check the tendency of the AFW steps in the LaCG algorithm to add more vertices.

\subsection{Additional Experiments}

We also consider the video co-localization problem, which can be shown to be equivalent to minimizing a quadratic objective function over a flow polytope~\cite{joulin2014efficient}. In this problem, the linear optimization oracle corresponds to finding a shortest path in a directed acyclic graph (DAG), for which there are algorithms that solve this problem in running time $\mathcal{O}(E+V),$ where $E$ and $V$ are the number of edges and vertices of the graph, respectively. We solve this problem over a directed acyclic graph with $3180$ edges and $227$ nodes that mimics the structure shown in shown in \cite{joulin2014efficient}, i.e., it has a source node connected to a layer of $15$ nodes, and each layer is fully connected with directed edges to the next layer, to make a total of $15$ layers (of $15$ nodes per layer). The last layer of nodes is connected to a node that acts as a sink. The quadratic was generated in the same way as in the Birkhoff polytope example, i.e. $f(\vx) = \vx^T\frac{M^TM+I}{2}\vx$, with $M$ having $1\%$ non-zero entries drawn from a standard Gaussian distribution. The matrix $M^TM$ has $27\%$ non-zero entries. The condition number in this instance is $L/\mu = 140$.

\begin{figure*}[ht!]
    \centering
    \vspace{-10pt}
    \hspace{\fill}
    \subfloat[Iteration]{{\includegraphics[width=3.95cm]{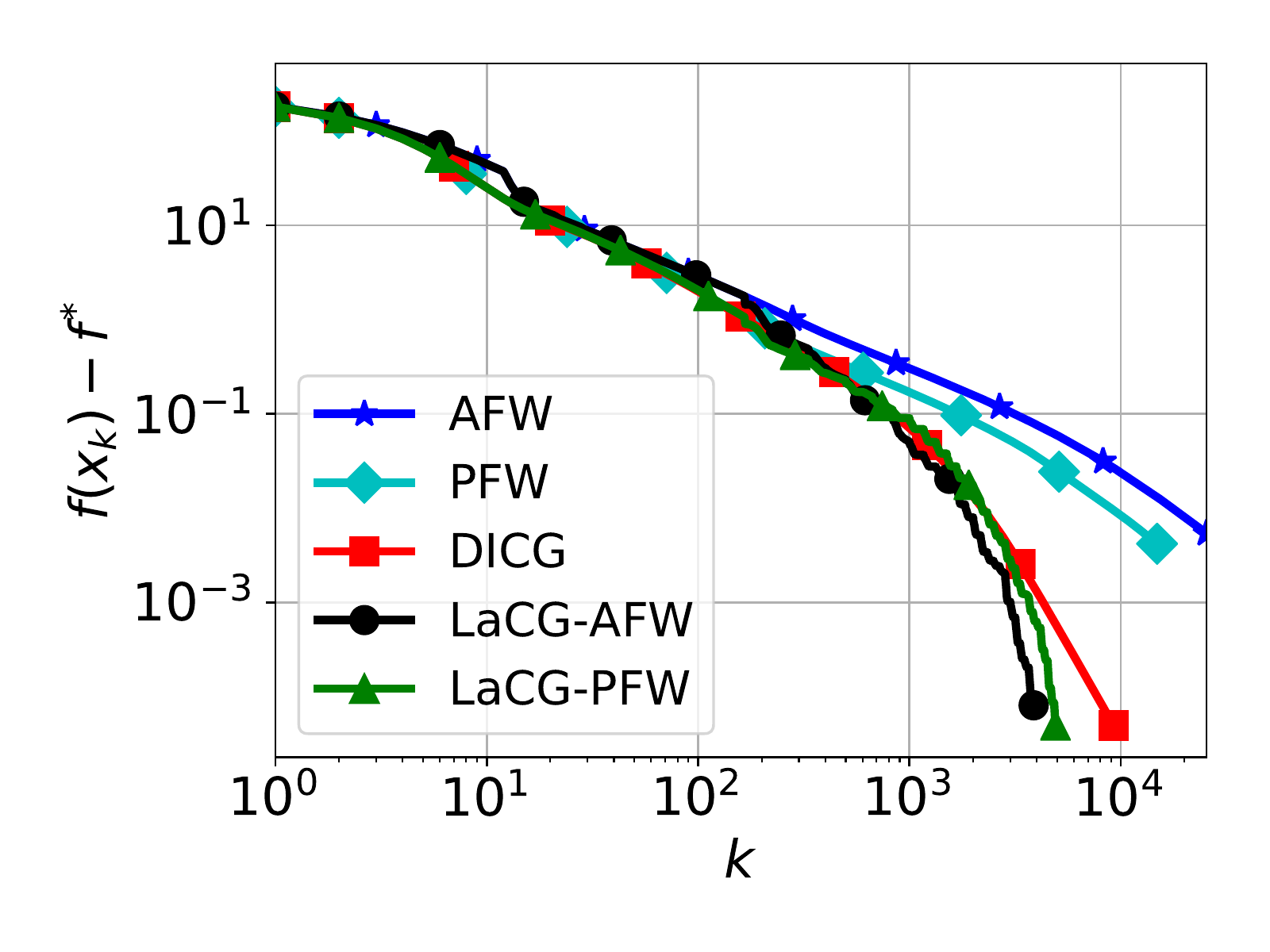} }\label{fig:truevideo1-it-cnt}}%
    \hspace{\fill}
    \subfloat[Wall-clock time]{{\includegraphics[width=3.85cm]{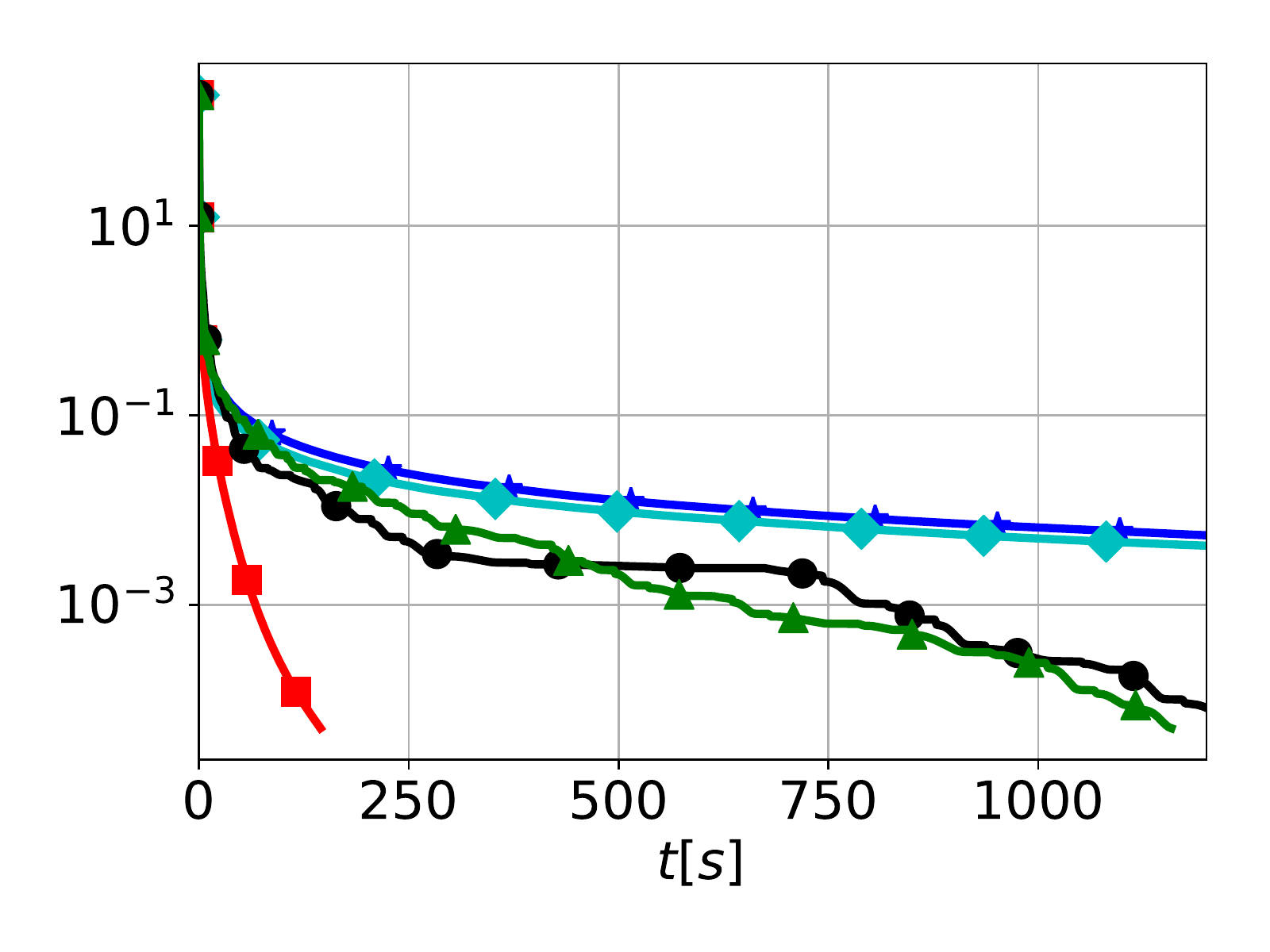} }\label{fig:truevideo1-time}}%
    \hspace{\fill}
    \subfloat[Iteration]{{\includegraphics[width=3.9cm]{Images/TrueVideoColocalization_PG_IterationLogLog.pdf} }\label{fig:truevideo2-it-cnt}}%
    \hspace{\fill}
    \subfloat[Wall-clock time]{{\includegraphics[width=3.85cm]{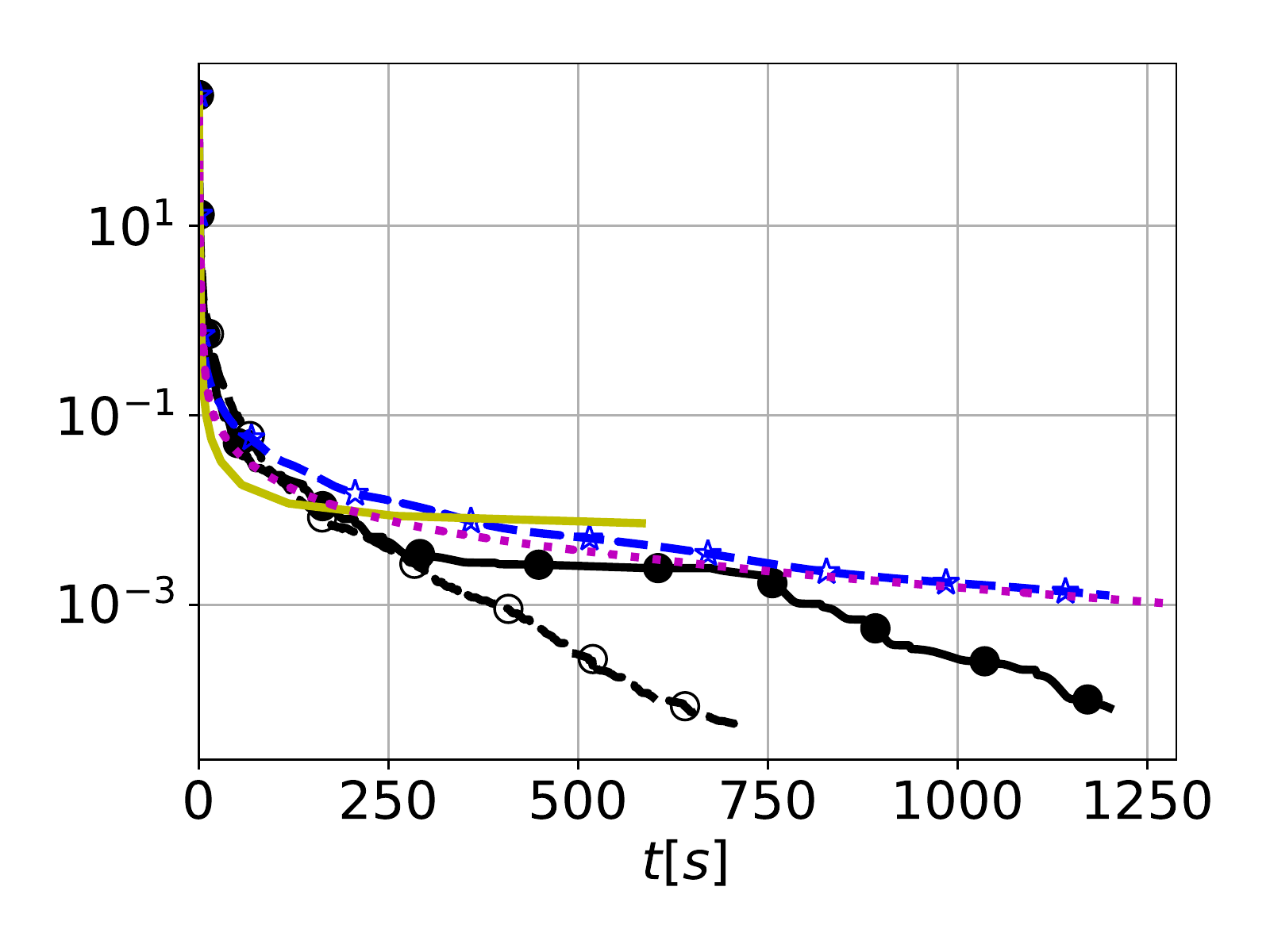} }\label{fig:truevideo2-time}}%
    \hspace*{\fill}
    \caption{Video co-localization: Algorithm comparison in terms of \protect\subref{fig:video1-it-cnt},\protect\subref{fig:video2-it-cnt} iteration count and \protect\subref{fig:video1-time},\protect\subref{fig:video2-time} wall-clock time.}%
    \label{fig:truevideo}%
\end{figure*}

In this example, DiCG is much more efficient than any of the active set based methods. Even though DiCG exhibits a slower convergence rate than LaCG variants, it greatly benefits from not maintaining an active set, which makes its iterations much more efficient. However, as discussed before, DiCG is not broadly applicable. Moreover, even in cases where it is applicable, DiCG can still be slower than LaCG variants once the linear optimization oracle is not very fast or the active sets do not get too large. It is an interesting open question whether local acceleration can be achieved with decomposition invariant methods.

\end{document}

%% file: prologue.tex
\usepackage[utf8]{inputenc} 
\usepackage[T1]{fontenc}    
\usepackage[colorlinks,urlcolor=blue,citecolor=blue,linkcolor=blue]{hyperref}       
\usepackage{url}            
\usepackage{booktabs}       
\usepackage{amsfonts}       
\usepackage{microtype}      
\usepackage{nicefrac}
\usepackage{subfig}

\usepackage[noend]{algpseudocode}
\usepackage{algorithm}
\usepackage{bm}
\usepackage[colorinlistoftodos]{todonotes}
\usepackage{amsmath}
\usepackage{amsthm}
\usepackage{amssymb}
\usepackage{thmtools}
\usepackage{thm-restate}
\usepackage{paralist}
\usepackage{comment}
\usepackage{cite}
\usepackage{float}
\usepackage{balance}
\usepackage{stmaryrd}

\newcommand{\true}{\operatorname{true}}
\newcommand{\false}{\operatorname{false}}
\newcommand{\innp}[1]{\left\langle #1 \right\rangle}

\newcommand{\zeros}{\textbf{0}}
\newcommand{\vx}{\mathbf{x}}
\newcommand{\vd}{\mathbf{d}}

\newcommand{\cx}{\mathcal{X}}
\newcommand{\cc}{\mathcal{C}}
\newcommand{\cs}{\mathcal{S}}

\newcommand{\vxh}{\mathbf{\hat{x}}}

\newcommand{\vy}{\mathbf{y}}
\newcommand{\vz}{\mathbf{z}}
\newcommand{\vv}{\mathbf{v}}
\newcommand{\vw}{\mathbf{w}}

\newcommand{\vb}{\mathbf{b}}
\newcommand{\vc}{\mathbf{c}}

\newcommand{\vu}{\mathbf{u}}
\newcommand{\vs}{\mathbf{s}}

\newcommand{\vlambda}{\bm{\lambda}}

\newcommand{\defeq}{\stackrel{\mathrm{\scriptscriptstyle def}}{=}}

\newcommand{\rr}{\mathbb{R}}
\newcommand{\norm}[1]{\left\| #1 \right\|}

\newcommand{\cf}{\mathcal{F}}
\newcommand{\ct}{\mathcal{T}}
\newcommand{\cb}{\mathcal{B}}

\newcommand*{\vsepfbox}[1]{%
  \begingroup
    \sbox0{\fbox{#1}}%
    \setlength{\fboxrule}{0pt}%
    \mbox{\kern-\fboxsep\fbox{\unhbox0}\kern-\fboxsep}%
  \endgroup
}

\theoremstyle{plain} \numberwithin{equation}{section}
\newtheorem{theorem}{Theorem}[section]
\numberwithin{theorem}{section}
\newtheorem{corollary}[theorem]{Corollary}

\newtheorem{lemma}[theorem]{Lemma}
\newtheorem{proposition}[theorem]{Proposition}

\newtheorem{fact}[theorem]{Fact}
\theoremstyle{definition}

\newtheorem{remark}[theorem]{Remark}

\DeclareMathOperator{\interior}{\mathrm{int}}
\DeclareMathOperator{\relinterior}{\mathrm{rel.int}}
\DeclareMathOperator*{\argmin}{argmin}
\DeclareMathOperator*{\argmax}{argmax}
\DeclareMathOperator{\Aff}{\mathrm{aff}}
\DeclareMathOperator{\vertex}{\mathrm{vert}}
\DeclareMathOperator{\co}{\mathrm{co}}
\graphicspath{{imgs/}}
\makeatletter
\def\mathcolor#1#{\@mathcolor{#1}}
\def\@mathcolor#1#2#3{%
  \protect\leavevmode
  \begingroup
    \color#1{#2}#3%
  \endgroup
}
\makeatother


%% file: main.bbl
\begin{thebibliography}{10}

\bibitem{Ahuja:1993:NFT:137406}
R.~K. Ahuja, T.~L. Magnanti, and J.~B. Orlin.
\newblock {\em Network Flows: Theory, Algorithms, and Applications}.
\newblock Prentice-Hall, Inc., 1993.

\bibitem{bashiri2017decomposition}
M.~A. Bashiri and X.~Zhang.
\newblock Decomposition-invariant conditional gradient for general polytopes
  with line search.
\newblock In {\em Proc. NIPS'17}, 2017.

\bibitem{beck2009fast}
A.~Beck and M.~Teboulle.
\newblock A fast iterative shrinkage-thresholding algorithm for linear inverse
  problems.
\newblock {\em SIAM J. Imaging Sci.}, 2(1):183--202, 2009.

\bibitem{betancourt2018symplectic}
M.~Betancourt, M.~I. Jordan, and A.~C. Wilson.
\newblock On symplectic optimization.
\newblock {\em arXiv preprint arXiv:1802.03653}, 2018.

\bibitem{braun2018blended}
G.~Braun, S.~Pokutta, D.~Tu, and S.~Wright.
\newblock Blended conditional gradients: the unconditioning of conditional
  gradients.
\newblock In {\em Proc. ICML'19}, 2019.

\bibitem{BPZ2017}
G.~Braun, S.~Pokutta, and D.~Zink.
\newblock {Lazifying Conditional Gradient Algorithms}.
\newblock In {\em Proc. ICML'17}, 2017.

\bibitem{Bubeck2015}
S.~Bubeck, Y.~T. Lee, and M.~Singh.
\newblock {A geometric alternative to Nesterov's accelerated gradient descent}.
\newblock {\em arXiv preprint, arXiv:1506.08187}, 2015.

\bibitem{cohen2018acceleration}
M.~B. Cohen, J.~Diakonikolas, and L.~Orecchia.
\newblock On acceleration with noise-corrupted gradients.
\newblock In {\em Proc. ICML'18}, 2018.

\bibitem{diakonikolas2018width}
J.~Diakonikolas, M.~Fazel, and L.~Orecchia.
\newblock Width-independence beyond linear objectives: Distributed fair packing
  and covering algorithms.
\newblock {\em arXiv preprint arXiv:1808.02517}, 2018.

\bibitem{AXGD}
J.~Diakonikolas and L.~Orecchia.
\newblock Accelerated extra-gradient descent: A novel, accelerated first-order
  method.
\newblock In {\em Proc. ITCS'18}, 2018.

\bibitem{thegaptechnique}
J.~Diakonikolas and L.~Orecchia.
\newblock The approximate duality gap technique: A unified theory of
  first-order methods.
\newblock {\em SIAM J. Optimiz.}, 29(1):660--689, 2019.

\bibitem{drusvyatskiy2016optimal}
D.~Drusvyatskiy, M.~Fazel, and S.~Roy.
\newblock An optimal first order method based on optimal quadratic averaging.
\newblock {\em SIAM J. Optimiz.}, 28(1):251--271, 2018.

\bibitem{duchi2008efficient}
J.~Duchi, S.~Shalev-Shwartz, Y.~Singer, and T.~Chandra.
\newblock Efficient projections onto the $\ell_1$-ball for learning in high
  dimensions.
\newblock In {\em Proc. NIPS'08}, 2008.

\bibitem{fleischer2006polynomial}
L.~K. Fleischer, A.~N. Letchford, and A.~Lodi.
\newblock Polynomial-time separation of a superclass of simple comb
  inequalities.
\newblock {\em Math. of Oper. Res.}, 31(4):696--713, 2006.

\bibitem{frank1956algorithm}
M.~Frank and P.~Wolfe.
\newblock An algorithm for quadratic programming.
\newblock {\em Naval research logistics quarterly}, 3(1-2):95--110, 1956.

\bibitem{freund2017extended}
R.~M. Freund, P.~Grigas, and R.~Mazumder.
\newblock An extended {F}rank-{W}olfe method with ``in-face'' directions, and
  its application to low-rank matrix completion.
\newblock {\em SIAM J. Optimiz.}, 27(1):319--346, 2017.

\bibitem{futami2019bayesian}
F.~Futami, Z.~Cui, I.~Sato, and M.~Sugiyama.
\newblock Bayesian posterior approximation via greedy particle optimization.
\newblock In {\em Proc. AAAI'19}, 2019.

\bibitem{garber2013linearly}
D.~Garber and E.~Hazan.
\newblock A linearly convergent variant of the conditional gradient algorithm
  under strong convexity, with applications to online and stochastic
  optimization.
\newblock {\em SIAM J. Optimiz.}, 26(3):1493--1528, 2016.

\bibitem{LDLCC2016}
D.~Garber and O.~Meshi.
\newblock Linear-memory and decomposition-invariant linearly convergent
  conditional gradient algorithm for structured polytopes.
\newblock In {\em Proc. NIPS'16}, 2016.

\bibitem{garber2018fast}
D.~Garber, S.~Sabach, and A.~Kaplan.
\newblock Fast generalized conditional gradient method with applications to
  matrix recovery problems.
\newblock {\em arXiv preprint arXiv:1802.05581}, 2018.

\bibitem{guelat1986some}
J.~Gu{\'e}lat and P.~Marcotte.
\newblock Some comments on {W}olfe's `away step'.
\newblock {\em Math. Program.}, 35(1):110--119, 1986.

\bibitem{gutman2019condition}
D.~H. Gutman and J.~F. Pena.
\newblock The condition number of a function relative to a set.
\newblock {\em arXiv preprint arXiv:1901.08359}, 2019.

\bibitem{jaggi2013revisiting}
M.~Jaggi.
\newblock Revisiting {Frank-Wolfe}: Projection-free sparse convex optimization.
\newblock In {\em Proc. ICML'13}, 2013.

\bibitem{joulin2014efficient}
A.~Joulin, K.~Tang, and L.~Fei-Fei.
\newblock Efficient image and video co-localization with {F}rank-{W}olfe
  algorithm.
\newblock In {\em Proc. ECCV'14}, 2014.

\bibitem{kerdreux2018restarting}
T.~Kerdreux, A.~d'Aspremont, and S.~Pokutta.
\newblock Restarting {F}rank--{W}olfe.
\newblock In {\em Proc. AISTATS'18}, 2018.

\bibitem{kerdreux2018frank}
T.~Kerdreux, F.~Pedregosa, and A.~D'Aspremont.
\newblock {Frank-Wolfe} with subsampling oracle.
\newblock In {\em Proc. ICML'18}, 2018.

\bibitem{lacoste2015global}
S.~Lacoste-Julien and M.~Jaggi.
\newblock On the global linear convergence of {Frank-Wolfe} optimization
  variants.
\newblock In {\em Proc. NIPS'15}, 2015.

\bibitem{lacoste2013block}
S.~Lacoste-Julien, M.~Jaggi, M.~W. Schmidt, and P.~Pletscher.
\newblock Block-coordinate {Frank-Wolfe} optimization for structural svms.
\newblock In {\em Proc. ICML'13}, 2013.

\bibitem{lan2013complexity}
G.~Lan.
\newblock The complexity of large-scale convex programming under a linear
  optimization oracle.
\newblock {\em arXiv preprint arXiv:1309.5550}, 2013.

\bibitem{lan2017conditional}
G.~Lan, S.~Pokutta, Y.~Zhou, and D.~Zink.
\newblock Conditional accelerated lazy stochastic gradient descent.
\newblock In {\em {Proc. ICML'17}}, 2017.

\bibitem{lan2016conditional}
G.~Lan and Y.~Zhou.
\newblock Conditional gradient sliding for convex optimization.
\newblock {\em SIAM J. Optimiz.}, 26(2):1379--1409, 2016.

\bibitem{levitin1966constrained}
E.~S. Levitin and B.~T. Polyak.
\newblock Constrained minimization methods.
\newblock {\em USSR Computational mathematics and mathematical physics},
  6(5):1--50, 1966.

\bibitem{lin2015universal}
H.~Lin, J.~Mairal, and Z.~Harchaoui.
\newblock A universal catalyst for first-order optimization.
\newblock In {\em Proc. NIPS'15}, 2015.

\bibitem{locatello2018revisiting}
F.~Locatello, A.~Raj, S.~P. Reddy, G.~R{\"a}tsch, B.~Sch{\"o}lkopf, S.~U.
  Stich, and M.~Jaggi.
\newblock Revisiting first-order convex optimization over linear spaces.
\newblock {\em arXiv preprint arXiv:1803.09539}, 2018.

\bibitem{mitchell1974finding}
B.~Mitchell, V.~F. Dem’yanov, and V.~Malozemov.
\newblock Finding the point of a polyhedron closest to the origin.
\newblock {\em SIAM J. Control}, 12(1):19--26, 1974.

\bibitem{nesterov2018introductory}
Y.~Nesterov.
\newblock {\em Lectures on Convex Optimization}.
\newblock Springer, 2018.

\bibitem{nesterov19830}
Y.~E. Nesterov.
\newblock An ${O}(1/k)$-rate of convergence method for smooth convex functions
  minimization.
\newblock In {\em Dokl. Acad. Nauk SSSR}, volume 269, pages 543--547, 1983.

\bibitem{pena2018polytope}
J.~Pena and D.~Rodriguez.
\newblock Polytope conditioning and linear convergence of the {F}rank-{W}olfe
  algorithm.
\newblock {\em Math. of Oper. Res.}, 44(1):1--18, 2018.

\bibitem{polyak1964some}
B.~T. Polyak.
\newblock Some methods of speeding up the convergence of iteration methods.
\newblock {\em USSR Comput. Math. \& Math. Phys.}, 4(5):1--17, 1964.

\bibitem{rothvoss2017matching}
T.~Rothvo{\ss}.
\newblock The matching polytope has exponential extension complexity.
\newblock {\em Journal of the ACM (JACM)}, 64(6):41, 2017.

\bibitem{schmidt2011convergence}
M.~Schmidt, N.~L. Roux, and F.~R. Bach.
\newblock Convergence rates of inexact proximal-gradient methods for convex
  optimization.
\newblock In {\em Proc. NIPS'11}, 2011.

\bibitem{swoboda2019map}
P.~Swoboda and V.~Kolmogorov.
\newblock {MAP} inference via block-coordinate {Frank-Wolfe} algorithm.
\newblock In {\em Proc. IEEE CVPR'19}, 2019.

\bibitem{tseng2008}
P.~Tseng.
\newblock On accelerated proximal gradient methods for convex-concave
  optimization, 2008.

\end{thebibliography}
